\documentclass[12pt,a4paper,reqno]{amsart}

\usepackage{amssymb,verbatim}
\usepackage{amsfonts}
\usepackage[utf8]{inputenc}
\usepackage[mathscr]{euscript}
\usepackage[hidelinks]{hyperref}
\usepackage{mathrsfs}
\usepackage[all]{xy}

\usepackage[normalem]{ulem}
\usepackage{tikz}
\usepackage{tikz-cd}
\usepackage{etoolbox}
\usepackage{enumitem}
\usepackage[export]{adjustbox}
\usepackage{mathtools}
\usetikzlibrary{arrows,arrows.meta,calc,positioning}
\usepackage[
top=3cm,left=2.5cm,right=2.5cm,bottom=2.3cm,headsep=15pt
]{geometry}

\usepackage{xcolor}
\usepackage{quiver}
\setenumerate{label=\textup{(\arabic*)},itemsep=3pt,topsep=0pt,leftmargin=1.7em}

\tikzset{
>=Stealth,node distance=1.5cm,
inner sep=3pt,font=\footnotesize,
main node/.style={circle,inner sep=2pt},
freccia/.style={->,shorten >=1pt,shorten <=1pt},
ciclo/.style={out=130, in=50, loop, distance=1.7cm, ->},
line width=0.5pt,
}

\newsavebox\Line
\sbox\Line{
\begin{tikzpicture}
\node[main node] (1) {};
\node (2) [right of=1] {};
\node (3) [right of=2] {};
\node (5) [right of=3] {};
\node (6) [right of=5] {};

\filldraw (1) circle (0.05) node[below=2pt] {$n$};
\filldraw (2) circle (0.05) node[below=2pt] {$n-1$};
\filldraw (3) circle (0.05) node[below=2pt] {$n-2$};
\filldraw (5) circle (0.05) node[below=2pt] {$2$};
\filldraw (6) circle (0.05) node[below=2pt] {$1$};
\begin{scope}[font=\footnotesize]
\path[freccia] (1) edge node[above] {$z_n$} (2)
(2) edge node[above] {$z_{n-1}$} (3)
(5) edge node[above] {$z_2$} (6);
\end{scope}

\path[dashed,shorten >=3pt,shorten <=3pt] (3) edge (5);

\end{tikzpicture}
}

\newsavebox\Star
\sbox\Star{
\begin{tikzpicture}
\node[main node] (1) at (0,0) {};

\filldraw (1) circle (0.05) node[below=2pt] {$1$};

\foreach \k/\n in {15/2,65/3,115/4,165/5,215/6,325/n} {
\node[main node] (\k) at (\k:1.5) {};
\filldraw (\k) circle (0.05);
\path[freccia] (\k) edge (1);
\path (1) -- (\k) node[pos=1.3] {$\n$};
}

\begin{scope}[font=\footnotesize]
\path[freccia] (15) edge node[above,sloped] {$y_2$} (1);
\path[freccia] (215) edge node[above,sloped] {$y_6$} (1);
\path[freccia] (325) edge node[above,sloped] {$y_n$} (1);
\end{scope}

\draw[dashed] (235:1.2) arc (235:310:1.2);

\end{tikzpicture}
}

\newsavebox\Football
\sbox\Football{
\begin{tikzpicture}

\clip (-0.3,-1.3) rectangle (3.3,1.3);
\node[main node] (1) {};
\node[main node,right of=1] (2) {};
\node[main node,right of=2] (3) {};

\filldraw (1) circle (0.05) node[left] {$2$};
\filldraw (3) circle (0.05) node[right] {$1$};

\foreach \k in {25,45,65} {
\path[freccia] (1) edge[bend left=\k] (3);
\path[freccia] (1) edge[bend right=\k] (3); }

\draw[densely dotted] ($(2)+(0,0.25)$) edge ($(2)+(0,-0.25)$);

\begin{scope}[font=\footnotesize]
\node at ($(2)+(0,1.1)$) {$x_n$} (3);
\node at ($(2)+(0,-1.1)$) {$x_2$} (3);
\end{scope}

\end{tikzpicture}
}

\newtheorem{theorem}{Theorem}[section]
\newtheorem{lemma}[theorem]{Lemma}
\newtheorem{proposition}[theorem]{Proposition}
\newtheorem{corollary}[theorem]{Corollary}
\newtheorem{conjecture}[theorem]{Conjecture}

\theoremstyle{definition}
\newtheorem{definition}[theorem]{Definition}
\newtheorem{example}[theorem]{Example}

\theoremstyle{remark}
\newtheorem{remark}[theorem]{Remark}

\numberwithin{equation}{section}

\newcommand{\reg}{\operatorname{reg}}
\newcommand{\id}{\operatorname{id}}

\newcommand{\scj}{\subseteq}

\newcommand{\ol}[1]{\overline{#1}}
\renewcommand{\[}{\begin{equation}}
\renewcommand{\]}{\end{equation}}
\makeatletter
\@namedef{subjclassname@2020}{%
  \textup{2020} Mathematics Subject Classification}
\makeatother

\allowdisplaybreaks

\title[The functoriality of moves on graphs]
{\vspace*{-2.5cm}The functoriality of moves on graphs and\\ extended path homomorphisms}

\author[G.~G.~de Castro]{Gilles G.~de Castro}
\address[G.~G.~de Castro]{Departmento de Matem\'atica, Universidade Federal de Santa Catarina, Florian\'opolis, SC, 88040-900, Brazil}
\email{gilles.castro@ufsc.br}

\author[P. M.~Hajac]{Piotr M.~Hajac}
\address[P. M.~Hajac]{Instytut Matematyczny, Polska Akademia Nauk, ul. \'Sniadeckich 8, Warszawa, 00-656 Poland}
\email{pmh@impan.pl }

\author[M.~Lowiel]{Mateusz Lowiel}
\address[M.~Lowiel]{Instytut Matematyki, Wydzia\l Matematyki Informatyki i Mechaniki, Uniwersytet Warszawski, ul. Banacha 2, Warszawa 02-097 Poland}
\email[M.~Lowiel]{m.lowiel@uw.edu.pl}

\author[E. A. ~Pacheco]{Elizabeth A. Pacheco}
\address[E. A. ~Pacheco]{Centre for Research in Mathematics and Data Science, Western Sydney University, Australia}
\email{e.pacheco@westernsydney.edu.au}

\begin{document}
\baselineskip=16pt
\parskip=4\lineskip

\begin{abstract}
 Combinatorics of graphs is a very powerful tool  to unravel various properties of graph algebras. In 
 particular, isomorphisms between graph 
algebras are often implemented by moves between their graphs. 
In this paper, we make these combinatorial methods functorial.
Our categorical approach not only reveals the relationship between moves, 
but also leads to unexpected extensions of known moves (shifts
and reductions) and the characterization of out-splits as admissible epimorphisms admitting a 
lifting of the source map.
The former allowed us to go beyond isomorphisms and Morita equivalences, and the latter yielded a 
factorization lemma for out-splits.
Better still, the functorial understanding of maximal out-splits permitted combining them with the 
derived-graph construction, which resulted
in computing new proper endomorphisms of Cuntz algebras. Finally, 
the functorial implementation of balanced in-splits required the construction
of a new category of graphs, where morphisms are given as homomorphisms of graph inverse semigroups. 
Thus, the combinatorics of moves
brought a new tool to the realm of noncommutative topology where directed graphs are viewed as quantum spaces.
\end{abstract}
\color{black}
\maketitle

\vspace*{-0cm}

{\small\tableofcontents}

\section{Introduction}
\noindent
An effort to make combinatorial methods used in studying graph algebras functorial can be traced back at least to Spielberg and Katsura~\cite{s-j02,k-t06}. 
More recently,
a systematic approach to the functoriality of graph algebras was proposed in~\cite{hrt20,cht21,ht-24,ht-25}. 
Herein, motivated by the problem of functorial understanding of 
the inverses of $*$-isomorphisms induced by  moves such as out-splits and shifts, we construct a new category of graphs using graph inverse semigroups.
The new category enjoys a covariant functor to the category of C*-algebras and $*$-homomorphisms which extends 
the covariant functor from the category of graphs and admissible path homomorphisms introduced in~\cite{ht-24}.

Graph C*-algebras are the universal enveloping C*-algebras of Leavitt path algebras~\cite{aam-17}. 
They can be viewed as groupoid convolution C*-algebras, so
groupoids  automatically play an important role in the study of graph C*-algebras~\cite{kprr-97}. Moreover, several properties shared by graph C*-algebras
and their dense Leavitt path  subalgebras can be explained using the groupoid description of a Leavitt path algebra~\cite{cfst-14}. 
To find a groupoid describing the 
graph C*-algebra of an arbitrary graph, without demanding that all vertices be regular as was done in \cite{kprr-97}, Paterson \cite{p-a02} used 
a suitable reduction of the universal groupoid of the graph inverse semigroup. Inverse semigroups,
 born in the realm of partial bijections~\cite{w-vv52,p-gb54a}, are 
thus an important tool to unravel  C*-algebras 
associated with combinatorial objects \cite{r-e08,p-a99} such as graphs.
Better still,
in \cite{r-e08}, Exel sought a construction of a groupoid that depends only on the intrinsic structure of the underlying inverse semigroup, and such that the C*-
algebra of this groupoid is isomorphic to the C*-algebra associated with the initial combinatorial object. To this end,
 he considered the notions of covers, tight representations and tight groupoids. Inspired by Exel's definitions, we define covers for graphs and we transform the 
 tightness condition 
 for representations into the regularity of extended path homomorphisms.

A key result of the paper is the introduction of a new category of directed graphs and extended path homomorphisms,
 where morphisms are defined as regular
homomorphisms of graph inverse semigroups, and its application to express functorially $*$-isomorphisms given by
moves on graphs. More precisely, by extending the domain of  the known covariant functor defined on the category of graphs and path homomorphisms
to the new category of extended path homomorphisms, we obtain a bigger image in the target category
of C*-algebras and $*$-homomorphisms.  Together with the contravariant functor from the category of graphs and graph homomorphisms
to the  category
of C*-algebras and $*$-homomorphisms, this enables us to recover moves induced formulas  that use not only standard 
generators but also their adjoints.

As a motivating application, assuming that we out-split at a regular vertex, we prove that the inverse of the $*$-isomorphism contravariantly induced 
by collapisng an
out-splitted graph to the original graph
can be implemented using the new covariant functor. Much in the same way, we show that, assuming that we 
shift a graph at a pair of vertices 
$(v,w)$ such that the regular vertex $v$ does not emit an edge that is a loop, the inverse of the $*$-isomorphism covariantly induced by this shift
can also be implemented using the new covariant functor.

\color{black}
\section{Preliminaries}
\subsection{Directed graphs and finite paths}

By a (directed) \emph{graph} $E$ we mean a quadruple $(E^0,E^1,s_E,t_E)$, where \(E^0\) is the set of \emph{vertices}, $E^1$ is the set of \emph{edges}, and
$s_E\colon E^1\to E^0$, $t_E\colon E^1\to E^0$ are the \emph{source} and \emph{target} (range) maps assigning to each edge its beginning and end, respectively. Let \(v\in E^0\). The vertex \(v\) is called a \emph{sink} if \(s_E^{-1}(v)=\emptyset\),
a \emph{source} if \(t_E^{-1}(v)=\emptyset\), and \emph{regular} if it is not a sink
and \(|s_E^{-1}(v)|<\infty\).
The set of regular vertices of a graph \(E\) is denoted by \(\mathrm{reg}(E)\). 

A \emph{finite path} \(p\) in \(E\) is either a vertex or a finite sequence of edges
\(e_1\ldots e_n\) such that
\[
t_E(e_1)=s_E(e_2),\quad
t_E(e_2)=s_E(e_3),\quad
\ldots,\quad
t_E(e_{n-1})=s_E(e_n).
\]
The set of all finite paths in \(E\) is denoted by \( \mathrm{FP}(E)\). 
The \emph{length} $|p|$ of a finite path $p$  is the number of edges in the
sequence if $p$ is not a vertex, and zero if $p$ is a vertex. 
One naturally extends the source and the target maps to $ \mathrm{FP}(E)$ denoting them, respectively, 
by $s_{PE}$ and~$t_{PE}$. We introduce the following 
relation on $\mathrm{FP}(E)$:
\begin{equation}\label{popaths}
\alpha\preceq\beta\iff\exists\;\gamma\in  \mathrm{FP}(E)\colon \beta=\alpha\gamma.
\end{equation}
We say that two paths $\alpha$ and $\beta$ are
\begin{itemize} 
\item {\em comparable} if $\alpha\preceq\beta$ or $\beta \preceq\alpha$, 
\item \emph{orthogonal} if they are not comparable.
\end{itemize}

The {\em extended graph} $\bar{E}:=(\bar{E}^0, \bar{E}^1,s_{\bar{E}},t_{\bar{E}})$
of the graph $E$ is given as follows:
\begin{gather}
\bar{E}^0:=E^0,\quad \bar{E}^1:=E^1\sqcup (E^1)^*,\quad (E^1)^*:=\{e^*~|~e\in E^1\},
\nonumber\\
\forall\; e\in E^1:\quad s_{\bar{E}}(e):=s_E(e),\quad t_{\bar{E}}(e):=t_E(e),
\nonumber\\
\forall\; e^*\in (E^1)^*:\quad s_{\bar{E}}(e^*):=t_E(e),\quad t_{\bar{E}}(e^*):=s_E(e).
\end{gather}
Given a path $p=e_1e_2\ldots e_n\in  \mathrm{FP}(E)$, we put $p^*:=e_n^*\ldots e_1^*$ for the corresponding path in~$ \mathrm{FP}(\overline{E})$. 

\subsection{Graph inverse semigroups}
\begin{definition}[\cite{w-vv52,p-gb54a}]
An \emph{inverse semigroup} $S$ is a semigroup such that for every $s\in S$ there exists a unique $s^*\in S$ such that $ss^*s=s$ and $s^*ss^*=s^*$.
\end{definition}
\noindent
Let us begin by recalling some facts about inverse semigroups, which can be found in~\cite{m-l98}. For starters, inverse semigroups enjoy a natural partial 
order defined as follows: 
for $s,t\in S$, $s\leq t$ if $ss^*t=s$. Next, an element $s\in S$ is idempotent if and only if $s=tt^*$ for some $t\in S$. 
Moreover, one can show that the product of idempotents is again an 
idempotent. In particular, if $s\leq t$ and $t$ is an idempotent, then $s$ is also an idempotent.

The graph inverse semigroup was first introduced in \cite{AshHall} for directed graphs without multiple edges between vertices,
 and later was reintroduced in \cite{p-a02} in the context of graph C*-algebras.
\begin{definition}[\cite{AshHall,p-a02}]
Let $E$ be a graph. The \emph{graph inverse semigroup} $S(E)$ is the
set
\[
S(E):=\{(\alpha,\beta)\in \mathrm{FP}(E)\times \mathrm{FP}(E)\mid t_E(\alpha)=t_E(\beta)\}\cup\{0\}
\]
equipped with the multiplication  defined by
\[
(\alpha,\beta)(\gamma,\delta):=\begin{cases}
	(\alpha\gamma',\delta) & \text{if }\gamma=\beta\gamma' \\
	(\alpha,\delta\beta') & \text{if }\beta=\gamma\beta' \\
	0 & \text{otherwise,}
\end{cases}
\]
and the requirement that $0$ times anything is~$0$.
\end{definition}
\noindent
By \cite[Proposition~3.2]{p-a02}, $S(E)$ has the universal property of being a semigroup with zero generated by
 $E^0\cup E^1\cup (E^1)^*$ and satisfying the relations:
\begin{enumerate}
	\item (V) $uu'=\delta_{u,u'}u$ for all $u,u'\in E^0$,
	\item (E1) $s_E(e)e=e=et_E(e)$ for all $e\in E^1$,
	\item (CK1) $e^*e'=\delta_{e,e'}t_E(e)$ for all $e,e'\in E^1$.
\end{enumerate}
Here we identify a vertex $v$ with the pair $(v,v)$, an edge $e$ with the pair $(e,t_E(e))$, and an edge $e^*$ with the pair $(t_E(e), e)$.
 We observe that the set of non-zero idempotents of $S(E)$ is given by $\{(\alpha,\alpha)\mid \alpha\in  \mathrm{FP}(E)\}$.

Now, we adapt the notion of a cover of an inverse-semigroup idempotent given in \cite[Definition~11.5]{r-e08} to graph inverse semigroups. As idempotents
in graph inverse semigroups are always of the form $(\alpha,\alpha)$,  we effectively obtain:
\begin{definition}
	Let $E$ be a graph and $\alpha\in {\rm FP}(E)$. A finite set of paths $\{\alpha_i\}_{i=1}^k$ is said to be a \emph{cover} of $\alpha$ if $\alpha\preceq \alpha_i$ 
	for all $i\in \{1,\ldots,k\}$ and for every path $\beta$ such that $\alpha\preceq\beta$, there exists $i\in \{1,\ldots,k\}$ such that $\alpha_i$ and $\beta$ 
	are comparable. A cover is called \emph{orthogonal} when all its elements are pairwise orthogonal.
\end{definition}

\section{Extended path homomorphisms}
\noindent
\begin{definition}\label{semi}
	Let $E,F$ be graphs, an \emph{extended path homomorphism} from $E$ to $F$ is a semigroup homomorphism $\theta:S(E)\to S(F)$ that preserves $0$.
\end{definition}

Note that,
	by the universality of the graph inverse semigroup, an extended path homomorphism $\theta:S(E)\to S(F)$ is uniquely defined by a pair of maps $
	\theta^0:E^0\to S(F)$ and $\theta^1:E^1\to S(F)$ such that:
	\begin{enumerate}
		\item $\theta^0(u)\theta^0(u')=\delta_{u,u'}\theta^0(u)$ for all $u,u'\in V$,
		\item $\theta^0(s_E(e))\theta^1(e)=\theta^1(e)=\theta^1(e)\theta^0(t_E(e))$ for all $e\in E^1$,
		\item $\theta^1(e)^*\theta^1(e')=\delta_{e,e'}\theta^0(t_E(e))$ for all $e,e'\in E^1$.
	\end{enumerate}
	From $\theta$ to the pair $(\theta^0,\theta^1)$, we just take the restrictions to $E^0$ and $E^1$. From the pair $(\theta^0,\theta^1)$, 
	we define $\varphi_\theta:S(E)\to S(F)$ in the obvious way for elements $(x,x')\in S(E)$, where $x,x'\in \mathrm{FP}(E)$.

 The category whose objects are all graphs and morphisms are given all  extended path homomorphisms we denote by~$\mathsf{EG}$. Here, the composition 
 of morphisms is simply the composition of maps, and identities are the identity maps on each graph inverse semigroup.
 
 \subsection{Monotonicity}
\begin{definition}
Let $E$ and $F$ be graphs.
A {\em path homomorphism} from $E$ to $F$  is a~map \mbox{$f\colon  \mathrm{FP}(E)\to  \mathrm{FP}(F)$} satisfying:
\begin{enumerate}
\item
$f(E^0)\subseteq F^0$,
\item
$s_{F}\circ f=f\circ s_{PE}\,,\quad t_{F}\circ f=f\circ t_{E}\,$,
\item
$\forall\;p,q\in  \mathrm{FP}(E)\text{ such that } t_{E}(p)=s_{E}(q)\colon f(pq)=f(p)f(q)$.
\end{enumerate}
We say that a path homomorphisms is {\em monotonic} if
\begin{equation}\label{mcondition}
f(e)\preceq f(e')\quad\Longrightarrow\quad e=e'.
\end{equation}
We denote by \(\mathsf{MIPG}\) the category whose objects are directed graphs and whose
morphisms are path homomorphisms that are injective on vertices and monotonic~\cite{ht-24}.
\end{definition}
\begin{lemma}\label{lem:monotone.paths}
	Let $f:E\to F$ be a morphism in the category $\mathsf{MIPG}$. 
	If $\alpha,\beta\in  \mathrm{FP}(E)$ have at least length 1, then $f(\alpha)\preceq f(\beta)$ implies that $\alpha$ and $\beta$ are comparable.
\end{lemma} 

\begin{proof}
	Assume that $\alpha=e_1\ldots e_m$ and $\beta=x_1\ldots x_n\,$,
where $e_1,\ldots, e_m,x_1,\ldots, x_n$ are edges.
Then
$f(e_1)\ldots f(e_m)\preceq f(x_1)\ldots f(x_n)$ implies that either $f(e_1)\preceq f(x_1)$
or $f(x_1)\preceq f(e_1)$, whence $e_1=x_1$ by monotonicity of $f$. It follows that 
$f(e_2)\ldots f(e_m)\preceq f(x_2)\ldots f(x_n)$. We can iterate the above argument to conclude that
$e_i=x_i$ for all $i\in\{1,\ldots,\mathrm{min}\{m,n\}\}$. Therefore, $\alpha\preceq \beta$ or 
$\beta\preceq \alpha$.
\end{proof}
 
\begin{proposition}
Let $f$ be a morphism in the category $\mathsf{MIPG}$. Then the
formulas
\begin{equation}\label{inducedext}
\varphi_f(0)=0\quad\text{and}\quad \varphi_f(\alpha,\beta):=(f(\alpha),f(\beta))
\end{equation}
render $\mathsf{MIPG}$ a subcategory of $\mathsf{EG}$ via the assignment $f\mapsto \varphi_f$.
\end{proposition}

\begin{proof}
First, note that $\varphi_f$ is well defined because $f$ respects the ranges of paths. Next, we show that $\varphi_f$ is a morphism in $\mathsf{EG}$. 
It maps $0$ to $0$ by definition, so it suffices to check that
\begin{equation}\label{homproof}
\varphi_f((\alpha,\beta)(\gamma,\delta))=\varphi_f(\alpha,\beta)\varphi_f(\gamma,\delta).
\end{equation}
The left-hand side of \eqref{homproof} reads
\begin{equation}
\varphi_f((\alpha,\beta)(\gamma,\delta))
=\begin{cases}\varphi_f(\alpha\gamma',\delta) & \text{if } \gamma=\beta\gamma'\\ \varphi_f(\alpha,\delta\beta') & \text{if } \beta=\gamma\beta'\\ 0 & 
\text{otherwise}\end{cases}=\begin{cases} (f(\alpha)f(\gamma'),f(\delta)) & \text{if }\gamma=\beta\gamma'\\
(f(\alpha),f(\delta)f(\beta')) & \text{if } \beta=\gamma\beta'\\ 0 & \text{otherwise}.\end{cases}
\end{equation}
Assume that $\gamma=\beta\gamma'$. Then, by the path homomorphism property, $f(\gamma)=f(\beta)f(\gamma')$. 
In this case, the right-hand side of~\eqref{homproof} equals
\begin{equation}
\varphi_f(\alpha,\beta)\varphi_f(\gamma,\delta)=(f(\alpha),f(\beta))(f(\gamma),f(\delta))=(f(\alpha)f(\gamma'),f(\delta))
\end{equation}
as needed. In the same way, if $\beta=\gamma\beta'$, we obtain that
\begin{equation}
\varphi_f(\alpha,\beta)\varphi_f(\gamma,\delta)=(f(\alpha),f(\beta))(f(\gamma),f(\delta))=(f(\alpha),f(\delta)f(\beta')).
\end{equation}

Finally, assume that $(\alpha,\beta)(\gamma,\delta)=0$, i.e. that neither $\gamma$ extends $\beta$ nor $\beta$ extends~$\gamma$. 
If $\beta=v$ is a vertex, then $\beta$ extends $\gamma$ if and only if $\gamma=v$.
Moreover, $\gamma$ extends $\beta$ if and only if $s_E(\gamma)=v$.
Hence, $(\alpha,v)(\gamma,\delta)=0$ if and only if $s_E(\gamma)\neq v$. Hence, we assume that $s_E(\gamma)\neq v$. Using the fact that $f$ respects the 
beginnings of paths and is injective on vertices, we conclude that 
\begin{equation}
s_F(f(\gamma))=f(s_E(\gamma))\neq f(v),
\end{equation}
which in turn implies that $(f(\alpha),f(v))(f(\gamma),f(\delta))=0$. If $\gamma$ is a vertex, a similar argument works. Next, assume that $\gamma$ and $\beta$ are 
paths of length $\geq 1$ and that $(\alpha,\beta)(\gamma,\delta)=0$. If $(f(\alpha),f(\beta))(f(\gamma),f(\delta))\neq 0$, then $f(\beta)$ and $f(\gamma)$ are 
comparable. By Lemma~\ref{lem:monotone.paths}, $\beta$ and $\gamma$ are comparable, which contradicts the fact that $(\alpha,\beta)(\gamma,\delta)=0$. The 
proof of~\eqref{homproof} is complete.


Given any graph $E$, we obtain that $\varphi_{\id}={\id}_{S(E)}$. Furthermore, by definition,
\begin{equation}
\varphi_{f\circ g}(\alpha,\beta)=(f(g(\alpha)),f(g(\beta)))=\varphi_f(g(\alpha),g(\beta))=\varphi_f(\varphi_g(\alpha,\beta)).
\end{equation} 
Therefore, we have a functor $F$ from $\mathsf{MIPG}$ to $\mathsf{EG}$ given by the identity on the objects and the assignment of morphisms 
$f\mapsto\varphi_f$. Hence, our functor is injective on the objects. To show that the image of $F$ gives a subcategory of $\mathsf{EG}$, it suffices to show that $F$ 
is faithful, namely that the assignment $f\mapsto \varphi_f$ is injective. Let $f\neq g$ be two path homomorphisms from $E$ to $F$ in $\mathsf{MIPG}$, so there 
exists $\alpha\in  \mathrm{FP}(E)$ such that $f(\alpha)\neq g(\alpha)$. Note that $(\alpha,r_E(\alpha))\in S(E)$ and
\begin{equation}
\varphi_f(\alpha,r_E(\alpha))=(f(\alpha),f(r_E(\alpha)))\neq (g(\alpha),g(r_E(\alpha)))=\varphi_g(\alpha,r_E(\alpha)).
\end{equation}
Hence $\varphi_f\neq \varphi_g$.
\end{proof}

\subsection{Regularity}
For starters, let us recall the concept of regularity for path homomorphisms of graphs.
\begin{definition}[\cite{ht-24}]\label{regularity}
Let $f\colon  \mathrm{FP}(E)\to  \mathrm{FP}(F)$ be a path homomorphism and let 
\begin{equation*}
{\rm reg}_0(E):=\{v\in{\rm reg}(E)~|~s_E^{-1}(v)=\{e\}\text{ and } t_E(e)=v\}
\end{equation*} 
denote the set of {\em 0-regular vertices}. We call $f$ \emph{regular} whenever the following conditions hold:
\vspace*{-2mm}\begin{enumerate}
\item
For any $v\in\mathrm{reg}(E)\setminus {\rm reg}_0(E)$, we require that:
\begin{enumerate}[label=(\alph*)]
\item
$f$ restricted to $s_E^{-1}(v)$ be injective;
\item
$p\in f(s_E^{-1}(v))$ if and only if
\begin{enumerate}[label=(\roman*)]
\item
$\exists\;n\in\mathbb{N}\setminus\{0\}\colon p=e_1\ldots e_n,\, e_i\in F^1 \text{ for all } i$,
\item
$pq\in f(s_E^{-1}(v))\,\Rightarrow\, q=t_{PE}(p)$,
\item
$\forall i\in\{1,\ldots,n\}, e\in s_F^{-1}(s_F(e_i))\,\exists\,r\in  \mathrm{FP}(F):e_1\ldots e_{i-1}er\in f(s_E^{-1}(v))$.
\end{enumerate}
\end{enumerate}
\item 
For any $v\in {\rm reg}_0(E)$, either the above condition holds or $f(s^{-1}_E(v))=f(v)$.
\end{enumerate}
\end{definition}
Restricting morphisms in the category $\mathsf{MIPG}$ to the regular ones yields a subcategory of $\mathsf{MIPG}$~\cite{ht-24}. We denote this subcategory
by~$\mathsf{RMIPG}$, and call its morphisms \emph{admissible path homomorphisms}. 
Now our goal is to construct an appropriate subcategory of $\mathsf{EG}$ extending the category $\mathsf{RMIPG}$ to the realm
of extended path homomorphisms. To this end, inspired by the proof of \cite[Proposition~6.9(2)]{cdh25}, first we show the following:
\begin{lemma}\label{lem:regular.vs.cover}
	Let $f\colon  \mathrm{FP}(E)\to  \mathrm{FP}(F)$ be a path homomorphism of graphs and $v\in \mathrm{reg}(E)$. Then the conditions (1)(a) and (1)(b) in 
	Definition~\ref{regularity} are satisfied for $v$ if and only if the set $f(s_E^{-1}(v))$ is an orthogonal cover of $f(v)$ consisting of 
	 positive-length paths and $|f(s_E^{-1}(v))|=|s_E^{-1}(v)|$.
\end{lemma}
\begin{proof}
	Assume first that the conditions (1)(a) and (1)(b) in Definition~\ref{regularity} hold for $v\in\mathrm{reg}(E)$. 
	 It follows from the condition (1)(b)(i) that $|f(x)|>0$ for any $x\in s_E^{-1}(v)$.
	Now, let $\gamma\in {\rm FP}(F)$ be a path such that $s_{PF}(\gamma)=f(v)$. 
	If $\gamma$ is a vertex, then $\gamma=f(v)$, which is comparable with any $f(x)$ 
	for $x\in s_E^{-1}(v)$ because $s_{PF}(f(x))=f(s_E(x))=f(v)$. Assume then that $\gamma:=\gamma_1\ldots\gamma_k$,  $\gamma_i\in F^1$, $i\in\{1,\ldots,k\}$. 
	If $\gamma_1=f(x)$ for some 
	$x\in s_E^{-1}(v)$, then $f(x)\preceq \gamma$, and we are done. Otherwise,  consider the set 
$$
X_1:=\{x\in s_E^{-1}(v)\mid \gamma_1\preceq f(x)\}.
$$
 This set is nonempty by the condition (1)(b)(iii) because $s_E(\gamma_1)=f(v)$. 
	If $\gamma_1\gamma_2=f(x)$ for some $x\in s_E^{-1}(v)$, then $f(x)\preceq \gamma$, and again we are done. Otherwise, consider the set
$$
X_2=\{x\in X_1\mid \gamma_1\gamma_2\preceq f(x)\}.
$$
Again, this set is nonempty by the condition~(1)(b)(iii) because, for every $x\in X_1$, the source of the second edge of $f(x)$ coincides with 
$t_F(\gamma_1)=s_F(\gamma_2)$. Iterating this procedure, we 
	either find $1\leq l\leq k$ such that $f(x)=\gamma_1\cdots \gamma_l$ for some $x\in s_E^{-1}(v)$, so  $f(x)\preceq \gamma$, or continue using 
	the procedure till we find 
	$x\in s_E^{-1}(v)$ such that $\gamma=\gamma_1\cdots \gamma_k\preceq f(x)$. Hence, $\gamma\sim f(x)$ for some $x\in s_E^{-1}(v)$. 
	Finally,
	to prove that the cover is orthogonal, let $x,x'\in s_E^{-1}(v)$ be such that $f(x)$ and $f(x')$ are comparable. Then it follows from
	the condition~(1)(b)(ii) of 
	Definition~\ref{regularity} that $f(x)=f(x')$. Consequently,  $x=x'$ by the condition~(1)(a), so the cover is orthogonal  and $|f(s_E^{-1}(v))|=|s_E^{-1}(v)|$.
	
	Now, assume that the set $f(s_E^{-1}(v))$ is an orthogonal cover of $f(v)$ consisting  positive-length  paths and that $|f(s_E^{-1}(v))|=|s_E^{-1}(v)|$. 
	Then the condition~(1)(a) is automatically satisfied. Furthermore, it is immediate that $f(x)$, $x\in s_E^{-1}(v)$,
	 satisfies the condition~(1)(b)(i). Next, the condition~(1)(b)(ii) 
	follows from the orthogonality of the cover. To see that $f(x)$ satisfies the condition~(1)(b)(iii), 
	write $f(x):=e_1\ldots e_n$, $e_i\in F^1$, $i\in\{1,\ldots,n\}$ and take $e\in s_F^{-1}(s_F(e_i))$. If $e=e_i$, 
	by choosing $r:=e_{i+1}\ldots e_n$ (or $r=t_E(e_n)$ if $i=n$), we obtain that $e_1\ldots e_{i-1}er=f(x)$. Otherwise, there exists $x'\in s_E^{-1}(v)\setminus\{x\}$ 
	such that $e_1\ldots e_{i-1}e$ and $f(x')$ are comparable. If $e_1\ldots e_{i-1}e\preceq f(x')$, we are done. Otherwise, by the orthogonality of the cover,
	 the only possibility for $f(x')\preceq e_1\ldots e_{i-1}e$ is the equality $f(x')=e_1\ldots e_{i-1}e$, which yields the equality
	$f(x')=e_1\ldots e_{i-1}er$ for  $r:=t_F(e)$. 
	
	Finally, let $p\in {\rm FP(E)}$ be such that conditions (i)-(iii) of (1)(b) are satisfied. It follows from the condition~(1)(b)(iii) applied to $i=1$
	that $s_{PF}(p)=f(v)$. Now, by the 
	definition of a cover, there exists $x\in s_E^{-1}(v)$ such that $f(x)$ and $p$ are comparable. If $p\preceq f(x)$, then $p=f(x)$ by the condition~(1)(b)(ii). 
	If $f(x)\preceq p$ and $p\neq f(x)$, then by applying the condition (1)(iii)(b) to  $i=n$ and $e=e_n$, we find $x'\in s_E^{-1}(v)$ such that $p\preceq f(x')$. This 
	implies that $f(x)\preceq f(x')$ and $f(x)\neq f(x')$, which contradicts the orthogonality of the cover. Therefore, $p=f(x)$ as needed.
\end{proof}

\begin{definition}\label{def:regularity}
	Given an extended path homomorphism $\varphi:S(E)\to S(F)$, we say that it is \emph{regular} whenever $u\in\reg(E)$, $t\in S(F)\setminus \{0\}$ and $t\leq 
	\varphi(u,u)$ implies there exists $e\in s_E^{-1}(u)$ such that $t\varphi(e,e)\neq 0$.
\end{definition}

	Since any homomorphism of semigroups preserves idempotents and any idempotent in a graph inverse semigroup is of the form
	$(\beta,\beta)$ for some path $\beta$,  we conclude that, for any extended path homomorphism $\varphi:S(E)\to S(F)$ and any $u\in E^0$,
	we have $\varphi(u,u)=(\beta,\beta)$. The condition $0\neq t\leq \varphi(u,u)$ then implies that $t=(\beta\gamma,\beta\gamma)$ for some 
	$\gamma\in  \mathrm{FP}(F)$. Furthermore, as $(e,e)\leq (u,u)$ for any $e\in s_E^{-1}(u)$ and $\varphi$  preserves the partial order, we infer that 
	$\varphi(e,e)=(\beta\delta_e,\beta\delta_e)$ for some $\delta_e\in  \mathrm{FP}(F)$. We thus arrive at an alternative formulation of the regularity condition:
\begin{proposition}\label{rmk:regularity} 
An extended path homomorphism $\varphi:S(E)\to S(F)$ is regular if and only if, for any $u\in\reg(E)$, the set
\[\label{cover}
\{\delta_e\in s_{PF}^{-1}(t_F(\beta))\mid e\in s_E^{-1}(u),\;\varphi(u,u)=(\beta,\beta),\; \varphi(e,e)=(\beta\delta_e,\beta\delta_e)\}
\]
  is a cover of~$ t_F(\beta)$.
\end{proposition}
\begin{proof}
Take $u\in\reg(E)$ and $\beta\in\mathrm{FP}(F)$ defined by $\varphi(u,u)=:(\beta,\beta)$.
The regularity of $\varphi$ means then that, for any $\gamma\in s_{PF}^{-1}(t_{PF}(\beta))$, there exists $s_{PF}^{-1}(t_{PF}(\beta))\ni \delta_e\sim \gamma$
such that $\varphi(e,e)=(\beta\delta_e,\beta\delta_e)$ for some $e\in s_E^{-1}(u)$. Therefore, the regularity of $u$ implies that the set \eqref{cover}
is a cover of~$ t_F(\beta)$. Vice versa, if the set \eqref{cover} is a cover for any $u\in\reg(E)$, then $\varphi$ is regular.
\end{proof}
\begin{remark}\label{covrem}
Note that the cover \eqref{cover} is always orthogonal. Indeed, as for any two different edges $e,e'\in s_E^{-1}(u)$ we have $(e,e)(e',e')=0$, also 
\[
(\beta\delta_e,\beta\delta_e)(\beta\delta_{e'},\beta\delta_{e'})=\varphi(e,e)\varphi(e',e')=\varphi\big((e,e)(e',e')\big)=0.
\]
Hence, $\delta_e$ and $\delta_{e'}$ are not comparable, so the cover \eqref{cover} is orthogonal.
\end{remark}
\begin{proposition}\label{equivregularity} 
	A morphism $f$  in the category $\mathsf{MIPG}$ is regular if and only if 
	the induced extended path homomorphism $\varphi_f$ \eqref{inducedext} is regular.
\end{proposition}
\begin{proof}
	Assume first that $f$ is regular and take  $u\in\reg(E)$.  If $u$ satisfies the conditions (1)(a) and (1)(b) of Definition~\ref{regularity}, we take
	 $\delta_e:=f(e)$ for every $e\in s_E^{-1}(v)$, so
	$\varphi_f(e,e)=(f(u)\delta_e,f(u)\delta_e)$. Now, since $t_{PF}(f(u))=f(u)$, the set  \eqref{cover} is a cover of $f(u)$ by 
	Lemma~\ref{lem:regular.vs.cover}. On the other hand, if $u$ satisfies the condition~(2), the set  \eqref{cover} is  $\{f(u)\}$, which is again a cover of~$f(u)$. 
	Hence, it follows from Proposition~\ref{rmk:regularity} that $\varphi_f$ is 
	regular.
	
	Assume now that $\varphi_f$ is regular and take  $u\in\reg(E)$. 
	Then the set \eqref{cover} is an orthogonal cover by Proposition~\ref{rmk:regularity} and Remark~\ref{covrem}. If it contains a vertex, 
	then its orthogonality implies that it is 
	a singleton. Consequently,  $e$ is the unique element of $s_E^{-1}(u)$ by Remark~\ref{covrem}. Also, $\varphi_f(e,e)=\varphi_f(u,u)$, so
	$f(e)=f(u)$. Furthermore, since $f$ is injective on vertices and $f(t_E(e))=t_{PF}(f(e))=t_{PF}(f(u))=f(u)$, we infer  that $t_E(e)=u$. Consequently,
	$u\in \reg_0(E)$ and the condition~(2) of Definition~\ref{regularity} is satisfed. Otherwise, if no element of the set \eqref{cover} is a vertex, 
	then the conditions (1)(a) and (1)(b) of Definition~\ref{regularity} are satisfied by Lemma~\ref{lem:regular.vs.cover}, Proposition~\ref{rmk:regularity} and 
	Remark~\ref{covrem}. Hence, $f$ is regular.
\end{proof}
%
%
%
%
%

\section{Covariant functoriality}
\noindent
\subsection{Cohn path algebras}
\begin{definition}
Let $k$ be a field and $E$ be a graph. The {\em path algebra} $kE$ is the vector space
\begin{equation}
kE:=\{f\in{\rm Map}(\mathrm{FP(E)},k)~|~f(p)\neq 0\text{ for finitely many } p\in \mathrm{FP(E)}\}
\end{equation}
equipped with the pointwise scalar multiplication and addition and the multiplication $m:kE\times kE\to kE$ extending the formulas
\begin{equation}
m(S_p,S_q):=\begin{cases}S_{pq} & \text{if }t_E(p)=s_E(q)\\0 & \text{otherwise}\end{cases},
\end{equation}
where $S_p$ is the characteristic function of a path $p\in {\rm FP}(E)$.
\end{definition}

\begin{definition}
Let $k$ be a field and $E$ be a graph. The {\em Cohn path algebra} $C_k(E)$ is the quotient of the path algebra $k\overline{E}$ of the extended graph 
$\overline{E}$ by the ideal generated by
\[
\{S_{f^*}S_e-S_{t_E(e)}\delta_{e,f}~|~e,f\in E^1\}. \label{CK1}
\]
\end{definition}
We can see the Cohn path algebra $C_k(E)$ as the contracted semigroup algebra $\ol{kS(E)}$, that is, the semigroup ring $kS(E)$ modulo the ideal 
generated by~$0_{S(E)}$.
Now,
given an extended path homomorphism $\varphi:S(E)\to S(F)$, there exists a canonical algebra homomorphism $\widetilde{\varphi}:kS(E)\to kS(F)$.
Hence, since $\varphi$ 
preserves $0$, we obtain an algebra homomorphism $\varphi_*^C:C_k(E)\to C_k(F)$. Explicitly, we have
\[\label{explicit}
\varphi_*^C(S_\alpha):=S_\beta S_\gamma^*=:S_{(\beta,\gamma)}, \quad\varphi(\alpha,t_{PE}(\alpha))=:(\beta,\gamma).
\]
\begin{proposition}\label{prop:functor.cohn}
	The association of Cohn path algebras to graphs and $\varphi_*^C$ to $\varphi$ defines a covariant functor from $\mathsf{EG}$ to $k-\mathsf{Alg}$.
\end{proposition}

\begin{proof}
	It is clear that $(\id_{S(E)})_C^*=\id_{C_k(E)}$. Given extended path homomorphisms $\varphi : S(E)\to S(E)$ and $\psi : S(F)\to S(G)$, it is straightforward to check 
	that $\widetilde{\psi\circ\varphi}=\widetilde{\psi}\circ\widetilde{\varphi}$. Because $\varphi$ and $\psi$ preserve 0, this equality goes through the quotients so 
	that $(\psi\circ\varphi)_C^*=\psi^*_C\circ \varphi^*_C$.
\end{proof}

\subsection{Leavitt path algebras}
\begin{definition}
Let $k$ be a field, $E$ be a graph, and $C_k(E)$ be the Cohn path algebra of~$E$.  The \emph{Leavitt path algebra} $L_k(E)$ of $E$ 
is the quotient of $C_k(E)$ by the ideal $I_E$ generated by 
\begin{equation}
\{S_u-\sum_{e\in s_E^{-1}(u)}S_eS_e^*\mid u\in\reg(E)\}. \label{CK2}
\end{equation}
\end{definition}
For an element $t=(\alpha,\beta)\in S(E)$ we define $S_t:=S_\alpha S_{\beta}^*$ 
 in either $C_k(E)$ or $L_k(E)$.

\begin{lemma}\label{lem:equality.cover}
	Let $E$ be a graph and $L_k(E)$ be its Leavitt path algebra. If $\alpha\in {\mathrm{FP}}(E)$ and $C:=\{\alpha_i\}_{i=1}^k$ is
	an orthogonal cover of $\alpha$, then
	\begin{equation}\label{eq:equality.cover}
		S_\alpha S_{\alpha}^*=\sum_{i=1}^{k} S_{\alpha_i}S_{\alpha_i}^*\in L_k(E).
	\end{equation}
\end{lemma}

\begin{proof}
	First, we consider the case when $\alpha=u\in E^0$. Our proof is by induction on $n:=\max |\alpha_i|$. If $n=0$, then all paths are vertices. Because different vertices are orthogonal, this means that $k=1$ and $\alpha_1=u$. Then \eqref{eq:equality.cover} holds trivially.
	
	Fix $n>0$ and suppose that the result is true for all $0\leq m<n$. Since $n>0$ and all paths begin in $u$, we have that $|\alpha_i|\geq 1$ for all $i$. In particular $u$ is not a sink. We claim that $u$ is regular. Assume on the contrary that $u$ is an infinite emitter and let $e\in s_E^{-1}(u)$ be such that $e$ is not the first edge of any $\alpha_i$. Then $e$ is not comparable with any $\alpha_i$, contradicting the hypothesis that $\{\alpha_i\}_{i=1}^k$ is a cover for $u$. We then have that
	\begin{equation}\label{eq:eq.cover.pf.1}
		S_u=\sum_{e\in s_E^{-1}(u)} S_eS_e^*.
	\end{equation}
	For each $e\in s_E^{-1}(u)$, let $C_e:=\{\beta\in  \mathrm{FP}(E)\mid e\beta\in C\}$. We claim that $C_e$ is a cover for $t_E(e)$ consisting of mutually orthogonal paths. Let $\beta,\beta'\in C_e$ be such that $\beta\neq\beta'$ so that $e\beta,e\beta'\in C$ and $e\beta\neq e\beta'$. By hypothesis $e\beta$ and $e\beta'$ are not comparable so that $\beta$ and $\beta'$ are not comparable. Now, let $\gamma$ be a path such that $s_{PE}(\gamma)= t_E(e)$. Then $e\gamma$ is a path starting in $u$ and so it is comparable to some path in $C$. As we have shown above, all paths in $C$ in this case have length greater or equal to 1, so $e\gamma$ must be comparable to some $e\beta$ for some $\beta\in C_e$, and hence $\gamma$ is comparable to $\beta$. Clearly $\max\{|\beta|\mid \beta\in C_e\}<n$, so that, by the induction hypothesis
	\begin{equation}\label{eq:eq.cover.pf.2}
		S_{t_E(e)}=\sum_{\beta\in C_e} S_{\beta}S_{\beta}^*.
	\end{equation}
	Combining \eqref{eq:eq.cover.pf.1} and \eqref{eq:eq.cover.pf.2}, we obtain
	\begin{align}
		S_u &=\sum_{e\in s_E^{-1}(u)} S_eS_e^*\nonumber\\
		&=\sum_{e\in s_E^{-1}(u)} S_eS_{t_E(e)}S_e^*\nonumber\\
		&=\sum_{e\in s_E^{-1}(u)}\sum_{\beta\in C_e}S_eS_{\beta}S_{\beta}^*S_e^*\nonumber\\
		&=\sum_{\alpha\in C}S_{\alpha}S_{\alpha}^*\nonumber\\
		&=\sum_{i=1}^{k} S_{\alpha_i}S_{\alpha_i}^*
	\end{align}
	completing the proof in the first case.
	
	For an arbitrary $\alpha\in {\rm FP}(E)$, let $C':=\{\delta_i\}_{i=1}^k$ be such that $\alpha_i=\alpha\delta_i$ for every $i$.
	It is straightforward to check that $C'$ is a cover of $t_{PE}(\alpha)$. Hence,
	\begin{align}
		S_{\alpha}S_{\alpha}^* &=S_{\alpha}S_{t_{PE}(\alpha)}S_{\alpha}^* \nonumber\\
		&=\sum_{i=1}^k S_{\alpha}S_{\delta_i}S_{\delta_i}^*S_{\alpha}^* \nonumber\\
		&=\sum_{i=1}^kS_{\alpha_i}S_{\alpha_i}^*,
	\end{align}
	which ends the proof.
\end{proof}

\begin{lemma}\label{lem:equivalence.regularity}
	Let $\varphi:S(E)\to S(F)$ be an extended path homomorphism and $k$ be a field. Then the following statements are equivalent:
	\begin{enumerate}
		\item 
		The extended path homomorphism $\varphi$ is regular.
		\item 
		The equality
		\begin{equation}\label{eq:lemma.regularity}
			S_{\varphi(u,u)}=\sum_{e\in s_E^{-1}(u)} S_{\varphi(e,e)}\in L_k(F)
		\end{equation}
		holds for every $u\in\reg(E)$.
		\item 
		The induced algebra homomorphism $\varphi_*^C\colon C_k(E)\to C_k(F)$ satisfies $\varphi_*^C(I_E)\scj I_F$.
	\end{enumerate}
\end{lemma}
\begin{proof}
	(1)$\Rightarrow$(2) If $\varphi$ is regular,  $u\in\reg(E)$ and $\varphi(u,u)=:(\beta,\beta)$, then, combining \eqref{explicit},
Lemma~\ref{lem:equality.cover},	Proposition~\ref{rmk:regularity} and Remark~\ref{covrem}, we obtain
	\begin{align}
		S_{\varphi(u,u)}&=S_\beta S_\beta^*\nonumber\\
		&=S_\beta S_{t_{PF}(\beta)} S_\beta^*\nonumber\\
		&=\sum_{e\in s_E^{-1}(u)} S_\beta S_{\delta_e}S_{\delta_e}^* S_\beta^*\nonumber\\
		&=\sum_{e\in s_E^{-1}(u)} S_{\beta\delta_e}S_{\beta\delta_e}^*\nonumber\\
		&=\sum_{e\in s_E^{-1}(u)}S_{\varphi(e,e)}.
	\end{align}
	
	(2)$\Rightarrow$(1) If $u\in\reg(E)$ and $0\neq t\leq \varphi(u,u)$, then 
	\begin{equation}
		S_t S_{\varphi(u,u)}=S_{t\varphi(u,u)}=S_t\neq 0.
	\end{equation}
	Therefore, it follows from \eqref{eq:lemma.regularity} that there exists $e\in s_E^{-1}(u)$ such that $S_tS_{\varphi(e,e)}\neq 0$, 
	which implies that $t\varphi(e,e)\neq 0$. Hence, $\varphi$ is regular.
	
	(2)$\Leftrightarrow$(3)
	The equivalence between (2) and (3) is immediate.
\end{proof}

With the above lemma, given a regular extended path homomorphism \mbox{$\varphi:S(E)\to S(F)$}, we can define an algebra homomorphism 
$\varphi_*^L\colon L_k(E)\to L_k(F)$ by
\begin{equation}\label{eq:from.cohn.to.leavitt}
\varphi_*^L(a+I_E):=\varphi_*^C(a)+I_F
\end{equation}
for every $a\in C_k(E)$. Much as in \eqref{explicit}, explicitly, we obtain
\[\label{explicit2}
L_k(E)\ni S_\alpha\longmapsto\varphi_*^L(S_\alpha)=S_{\varphi(\alpha, t_{PE}(\alpha))}=S_\beta S_\gamma^*\in L_k(F), 
\quad (\beta,\gamma):=\varphi(\alpha, t_{PE}(\alpha)).
\]
Better still, using \eqref{inducedext} and remembering  Proposition~\ref{equivregularity}, we can  specialize the above formula to the known formula
\cite{ht-24} for the covariantly induced $*$-homomorphism assigned to an admissible path homomorphism $f\colon \mathrm{FP}(E)\to\mathrm{FP}(F)$:
\[\label{covf}
L_k(E)\ni S_\alpha\longmapsto f_*^L(S_\alpha)=S_{f(\alpha)}\in L_k(F).
\]
\begin{theorem}
	For every graph $E$, the identity map $\id_{S(E)}$ is a regular extended path homomorphism. 
	Moreover, for any regular extended path homomorphisms 
	$\varphi:S(E)\to S(F)$ and $\psi:S(F)\to S(G)$, the composition $\psi\circ\varphi$ is a regular extended path homomorphism. 
	We thus obtain  the subcategory of $\mathsf{EG}$ consisting of all regular extended path homomorphisms and all graphs. 
	We denote this category by~$\mathsf{REG}$.
\end{theorem}

\begin{proof}
	The identity map $\id_{S(E)}$ clearly satisfies (3) of Lemma~\ref{lem:equivalence.regularity} and hence it is regular.
	As for the composition, we observe that
	\begin{equation}
		(\psi\circ\varphi)^*_C(I_E)=\psi^*_C\circ \varphi^*_C (I_E)\subseteq \psi^*_C(I_F)\subseteq I_G
	\end{equation}
	and hence $\psi\circ \varphi$ is regular by Lemma~\ref{lem:equivalence.regularity}.
\end{proof}
Combining the above theorem with Proposition~\ref{equivregularity}, we obtain:
\begin{corollary}
Let $f$ be a morphism in the category $\mathsf{RMIPG}$. Then the
formulas \eqref{inducedext}
render $\mathsf{RMIPG}$ a subcategory of $\mathsf{REG}$ via the assignment $f\mapsto \varphi_f$.
\end{corollary}
Now, Proposition~\ref{prop:functor.cohn} and Equation~\eqref{eq:from.cohn.to.leavitt} imply the main general result of this paper:
\begin{theorem}
	The association of Leavitt path algebras to graphs and $\varphi_*^L$ to $\varphi$ defines a covariant functor from $\mathsf{REG}$ to $k\text{-}\mathsf{Alg}$.
\end{theorem}

Since the {\em graph C*-algebra} $C^*(E)$ of a graph $E$ is 
the universal enveloping C*-algebra of the complex $*$-algebra $L_\mathbb{C}(E)$ (e.g., see~\cite[Definition~5.2.1]{aam-17}),
for any
$*$-homomorphism 
\mbox{$\phi:L_\mathbb{C}(E)\to L_\mathbb{C}(F)$} there exists a unique $*$-homomorphism
 \mbox{$\widetilde{\phi}:C^*(E)\to C^*(F)$} extending $\phi$ (e.g., see~\cite[Theorem~4.4]{at-11}). Therefore, the above theorem enjoys the following
 corollary:
\begin{corollary}\label{C*induct}
Let $\mathsf{C^*}\text{-}\mathsf{Alg}$ denote the category of C*-algebras and $*$-homomorphisms. 
The association of a graph C*-algebra to graphs and $\widetilde{\varphi_*^L}$ to $\varphi$ defines a~covariant functor from $\mathsf{REG}$ 
to $\mathsf{C^*}\text{-}\mathsf{Alg}$.
\end{corollary}
To simplify the notation, in what follows, we will write $\varphi_*$ instead of $\widetilde{\varphi_*^L}$ and $f_*$ instead of~$\widetilde{f_*^L}$.
Our next step is a particularly useful consequence of the fundamental Lemma~\ref{lem:equivalence.regularity}. 
\begin{proposition}\label{lem:from.hom.to.reg}
Let $\varphi\colon S(E)\to S(F)$ be an extended path homomorphism and let\linebreak
\mbox{$\Phi\colon L_k(E)\to L_k(F)$} be a linear map
such that
	\begin{gather}\label{Phi}
		\Phi(S_v)=S_{\varphi(v,v)}\text{ for all }v\in \reg(E), \quad\text{and}\quad
		\Phi(S_eS_e^*)=S_{\varphi(e,e)}\text{ for all }e\in s_E^{-1}(\reg(E)).
	\end{gather}
	Then $\varphi$ is regular.
\end{proposition}
\begin{proof}
	For any $v\in\reg(E)$, we have
	\begin{equation}
		S_{\varphi(v,v)}=\Phi(S_v)=\Phi\left(\sum_{e\in s_E^{-1}(v)} S_eS_e^*\right)=\sum_{e \in s^{-1}_E(v)}S_{\varphi(e,e)}.
	\end{equation}
	Hence, the regularity of $\varphi$ follows from Lemma~\ref{lem:equivalence.regularity}.
\end{proof}
\begin{corollary}\label{cor:from.hom.to.reg}
Let $\varphi\colon S(E)\to S(F)$ be an extended path homomorphism and let
\mbox{$g\colon  C^*(E)\to C^*(F)$} be a $*$-homomorphism such that
	\begin{gather}\label{gather}
		g(S_v)=S_{\varphi(v,v)}\text{ for all }v\in E^0, \quad\text{and}\quad
		g(S_e)=S_{\varphi(e,t_E(e))}\text{ for all }e\in E^1.
	\end{gather}
	Then $\varphi$ is regular and $\varphi_*=g$.
\end{corollary}	
\begin{proof}
Recall first that $L_{\mathbb{C}}(E)\subseteq C^*(E)$ and note that the condtions \eqref{gather} combined with \eqref{explicit} imply that
$g(L_{\mathbb{C}}(E))\subseteq L_{\mathbb{C}}(F)$. Hence, we can define $\Phi$ to be the restriction of $g$ to the Leavitt subspaces. Furthermore, as 
\eqref{gather} implies \eqref{Phi} for the thus defined $\Phi$, we can conclude that $\varphi$ is regular by Proposition~\ref{lem:from.hom.to.reg}. 
Finally,  $\varphi_*=g$ follows 
from \eqref{gather} combined with~\eqref{explicit2}.
\end{proof}		

\section{Homomorphisms given by moves}
\noindent
The most obvious concept of a morphism between graphs is a \emph{graph homomorphism}, which can be defined as a length preserving path
homomorphism. One often writes a graph homomorphism as a pair of maps mapping edges to edges and vertices to vertices. 
Graphs and graph homomorphisms form a category. 
Just as path homomorphisms need three conditions to admit a covariant functor to the category of algebras and algebra homomorphisms,
so do graph homomorphisms to admit a contravariant functor:
\begin{definition}\cite{k-t06,ht-24}
Let $f\colon E\to F$ be a graph homomorphism. We say that $f$ is {\em admissible} when it is
\begin{enumerate}
\item 
{\em proper}, i.e.\ the sets $f^{-1}(v)$ and $f^{-1}(e)$ are finite for any $v\in E^0$ and $e\in E^1$,
\item 
{\em target bijective}, i.e.\ for any $x\in F^1$ the restricted target map 
$$
t_E\colon f^{-1}(x)\longrightarrow f^{-1}(t_F(x))
$$
is bijective,
\item {\em regular}, i.e.\ $f^{-1}(\mathrm{reg}(F))\subseteq \mathrm{reg}(E)$.
\end{enumerate}
\end{definition}

The category of graphs and admissile graph homomorphisms enjoys a contravariant functor to the category of $\mathbb{Z}$-graded algebras and
$\mathbb{Z}$-graded algebra homomorphisms ($U(1)$-C*-algebras and $U(1)$-equivariant $*$-homomorphisms) (see \cite{ht-24}, cf.\ \cite{k-t06}). 
In this section, we will use this functor to encode a out-split isomorphism. We will also combine it with the covariant functor from the category
$\mathsf{REG}$ of graphs and extended path homomorphisms to substantially extend the scope of applications.

Herein, we will need a categorical formulation of the target-bijectivity condition  provided in \cite[Proposition~2.3]{ht-25}: 
a graph homomorphism $f\colon E\to F$ is target bijective if and only if $E^1$ is the pullback of 
$F^1\stackrel{t_F}{\to}F^0 \stackrel{f}{\leftarrow} E^0$ in the category of sets and maps. This point of view allows us to classify
all admissible graph homomorphisms:
To begin with, we need to define  a standard graph homomorphism. To this end, let $E:=(E^0,E^1,s_E,t_E)$ be any  graph
 and $\pi_0\colon E^0_\pi\to E^0$ be any map. Define $E^1_\pi$ as the following fiber product
(pullback in the category of sets and maps):
\[
E_\pi^1:=\{(e,v)\in E^1\times E_\pi^0\;|\;t_E(e)=\pi_0(v)\}.
\]
Next, define $\pi_1:=\mathrm{pr}_1\colon E_\pi^1\to E^1$ to be the canonical projection onto the first component, and
$t_\pi:=\mathrm{pr}_2\colon E_\pi^1\to E_\pi^0$  to be the canonical projection onto the second component. Now assume that there exists a map
$s_\pi\colon E_\pi^1\to E_\pi^0$ such that $\pi_0\circ s_\pi=s_E\circ\pi_1$, and define the graph $E_\pi:=(E_\pi^0,E_\pi^1,s_\pi,t_\pi)$.
Then we obtain the following commutative diagrams 
\begin{equation}\label{doublediag}
	\begin{tikzcd}
	{E_\pi^0} && {E_\pi^1} && {E_\pi^0} \\
	\\
	{E^0} && {E^1} && {E^0}
	\arrow["{\pi_0}", from=1-1, to=3-1]
	\arrow["{s_\pi}"', from=1-3, to=1-1]
	\arrow["{{t_\pi}}", from=1-3, to=1-5]
	\arrow["{{\pi_1}}", from=1-3, to=3-3]
	\arrow["{{\pi_0}}", from=1-5, to=3-5]
	\arrow["{s_E}"', from=3-3, to=3-1]
	\arrow["{{t_E}}", from=3-3, to=3-5]
\end{tikzcd}
\end{equation}
showing that $\pi:=(\pi_0,\pi_1)\colon E_\pi\to E$ is a graph homomorphism. 
We call $\pi$ a \emph{standard graph homomorphism}. 

We will also need the concept of ``an isomorphism of graph homomorphisms''. We say that graph homomorphisms $f\colon E\to F$ and $g\colon G\to H$
are \emph{isomorphic} whenever there exist graph isomorphisms $\alpha\colon E\to G$ and $\beta\colon F\to H$ such that $\beta\circ f= g\circ\alpha$. If $F=H$ and $\beta=\mathrm{id}_F$, then we say that $f$ and $g$ are isomorphic over $F$. The set of isomorphisms between $E\to F$ and $G\to F$ over a fixed graph $F$ will be denoted by $\mathrm{Iso}_F(E,G)$.
Now we are ready to state and prove:
\begin{proposition}\label{class}
Let $E$ be a graph. A standard graph homomorphism $\pi\colon E_\pi\to E$ is admissible if and only if
the vertex map $\pi_0\colon E_\pi^0\to E^0$  is proper (finite to one) and satisfies the inclusion
	\begin{equation}\label{inc}
		\pi_0^{-1}(\reg(E))\subseteq s_{\pi}(E_\pi^1).
	\end{equation}
Moreover, any admissible graph homomorphism  $f\colon F\to E$ is isomorphic to an admissible standard graph homomorphism~$\pi$.
\end{proposition}
\begin{proof}
To prove that $\pi$ is admissible, observe first that, as $\pi_0$ is proper, so is $\pi_1$ by the pullback property of the right square diagram in~\eqref{doublediag}.
Consequently, $\pi$ is proper.  Now, as the target bijectivity is satisfied automatically, it remains to show that $\pi$ is regular.
To this end, first note that $\pi_0^{-1}(\reg(E))$ cannot contain infinite emitters by the properness of~$\pi_1$, so
\eqref{inc} implies the desired $\pi_0^{-1}(\reg(E))\subseteq \reg(E_\pi^1)$. We have thus shown that $\pi\colon E_\pi\to E$ is admissible.
Vice versa, if $\pi$ is admissible, then the vertex map $\pi_0$ is proper and the inclusion \eqref{inc} holds.

Suppose now that $f\colon F\to E$ is an admissible graph homomorphism, and take $\pi_0:=f_0\colon E^0_\pi\to E^0$.
It follows from the target bijectivity of $f$ that  there exists a unique bijection
$\sigma\colon E^1_\pi\to F^1$ such that \mbox{$\pi_1=f_1\circ\sigma$} and $t_\pi=t_F\circ \sigma$. 
Now, we complete the construction of $E_\pi$ by taking $s_\pi:=s_F\circ\sigma$, and verify that $\pi$ is a standard graph homomorphism.
Better still, we observe that
$\alpha:=(\id,\sigma)\colon E_\pi\to F$ is a graph isomorphism. Hence,
as every graph isomorphism is automatically admissible, we infer that $\pi=f\circ \alpha\colon E_\pi\to E$ is an admissible graph
homomorphism. 
\end{proof}
The above proposition allows us to label admissible graph homomorphisms into a fixed graph $E$ by pairs $(\pi_0,s_\pi)$. More precisely, we can conclude the following corollary.
\begin{corollary}\label{classcor}
Up to isomorphism, admissible graph homomorphisms into a given graph $E$ are classified by the equivalence classes of pairs $(\pi_0\colon E_\pi^0\to E^0 ,s_\pi\colon E_\pi^1\to E_\pi^0)$ such that $\pi_0$ is proper and the following conditions hold:
\begin{equation*}
		 \pi_0\circ s_\pi=\pi_1\circ s_E, \qquad \pi_0^{-1}(\reg(E))\subseteq s_{\pi}(E_\pi^1).
\end{equation*}
Here, the equivalence relation is given by:
\begin{equation*}
    (\pi_0,s_\pi)\sim (\tau_0,s_\tau)\iff \mathrm{Iso}_E(E_\pi,E_\tau)\neq\emptyset.
\end{equation*}
\end{corollary}
The following example shows that there exists a vertex map $\pi_0$ that does not admit any source map $s_\pi$ satisfying the conditions of Corollary~\ref{classcor}.
\begin{example}
Let $E$ be this graph
\[
\begin{tikzcd}
	v \, \bullet & \bullet \,  w
	\arrow[from=1-2, to=1-1, in=-55, out=-125]
	\arrow[from=1-1, to=1-2, in=125, out=55 ]
\end{tikzcd},
\]
$E_\pi^0:=\{v\}$, and $\pi_0$  be the inclusion. Then $E_\pi$ must be the graph
\[
\begin{tikzcd}
	\underset{\text{\footnotesize $v$}}{\bullet}
	\arrow[from=1-1, to=1-1, loop, in=55, out=125, distance=10mm]
\end{tikzcd},
\]
and it is clear that $s_\pi\circ\pi_1\neq\pi_0\circ s_E$. 
\end{example}
\subsection{Out-split and out-amalgamation}
We begin by establishing basic definitions.
\begin{definition} [\cite{bp04,lm21}] \label{OutSplit}
Let $E$ be a graph. For every $v\in E^0$ that emits at least one edge, we partition $s_E^{-1}(v)$ into a finite disjoint union of 
nonempty sets $\epsilon_v^1 \cup ... \cup \epsilon_v^{m(v)}$, where $m(v)\geq 1$ and at most one $\epsilon_v^i$ is infinite. 
If $v$ is a sink, we set $m(v):=1$. Let $E_O$ denote the graph given by:
\begin{gather}
	E_O^0:= \{v^i\mid v\in E^0,\,1\leq i\leq m(v)\},\nonumber\\
	E_O^1:= \{e^i\,|\,e\in E^1,\, 1\leq i\leq m(t_E(e))\},\nonumber\\
	s_{E_O}(e^i):=s_E(e)^j\text{ for }e\in\epsilon_{s_E(e)}^j ,	\\
	t_{E_O}(e^i):=t_E(e)^i.\nonumber
\end{gather}
We say that $E_O$ is an \emph{out-split} of~$E$ and that $E$ is an \emph{out-amalgamation} of~$E_O$. 
We call the set
$$
\mathrm{supp}_E(E_O):=\{v\in E^0\;|\;m(v)>1  \}
$$
the \emph{support} of the out-split~$E_O$, and say that $E_O$ is an \emph{out-split of $E$ at $\mathrm{supp}_E(E_O)$}.
The out-split at $\emptyset$ is called the \emph{trivial out-split}. An out-split at $X$ is called \emph{maximal} whenever 
$X\subseteq \mathrm{reg}(E)$ and
$\epsilon_v^i$ is a 
singleton for every $v\in X$ and $i\in\{1,\dots,m(v)\}$. Finally, 
the maximal out-split at $\mathrm{reg}_+(E):=\{v\in\mathrm{reg}(E)\;|\;1<|s_E^{-1}(v)|\}$ is called the 
\emph{total out-split} of~$E$.
\end{definition}
\noindent
It is worth noting here that the above general concept of an out-split at any set of vertices allows us to out-split non-trivially 
at infinitely many, or even uncountably
many, vertices. Such out-splits cannot be obtained 
by iterating out-splits at a vertex~\cite{eilers2019refined}. Better still, one can check that the line graph construction~\cite{Harary1960SomePO} 
for row-finite graph without sinks
can be recast as the total out-split~\cite[Remark~3.3(i)]{bp04}. (Note that, in \cite{bp04}, line graphs are called dual graphs
and total out-split is called maximal out-split.)

\begin{definition}
Let  $E_O$ be an out-split of a graph~$E$. We call the graph homomorphism $f_O\colon E_O\to E$,
defined by $f_O(v^i):=v$ and $f_O(e^i):=e$ for every 
$v^i\in E_O^0$ and every $e^i\in E_O^1$, the \emph{out-split graph homomorphism}. 
\end{definition}
\noindent
One can say that the image of an out-split graph homomorphism is an out-amalgamation of its domain graph.

Let us now establish some basic facts concerning out-splits. For starters, the composition of out-splits need not be an out-split.
 Indeed,  the total out-split of the total out-split of the  graph
\[
\begin{tikzcd}
	\bullet & \bullet
	\arrow[from=1-1, to=1-1, loop, in=55, out=125, distance=10mm]
	\arrow[from=1-1, to=1-2]
\end{tikzcd}
\]
is not an out-split. This
shows that out-split graph homomorphisms  require an ambient category to be composable.
We will now characterize out-split graph homomorphisms as surjective standard graph homomorphisms of a certain type.
\begin{lemma}\label{MOisAd}
For any out-split $E_O$ of $E$, its out-split graph homomorphism is admissible.
\end{lemma}
\begin{proof}
	Since every  $m(v)$ is finite, the graph homomorphism $f_O$ is proper. By construction,  the restricted target map 
	$t_{E_O}\colon f_O^{-1}(e)\to f_O^{-1}(t_E(e))$ is bijective for every $e\in E^1$. Finally, for any vertex $v\in \mathrm{reg}(E)$ and $1\leq i\leq m(v)$, 
	we have
\begin{equation}
|s_{E_O}^{-1}(v^i)|=\sum_{e\in \epsilon_v^i}m(t_{E_O}(e)).
\end{equation} 
	The sum is finite and not zero. Hence $v^i\in\mathrm{reg}(E_O)$.
\end{proof}

\begin{lemma}\label{5.6}\footnote{We owe the idea of this formulation of out-split to Tomasz Maszczyk.}
	Let $\pi:E_\pi\to E$ be a surjective standard graph homomorphism. Then $\pi$ is isomorphic to an out-split graph homomorphism if and only if
	\begin{equation}\label{MOScondMulti}
		\forall\; v\in E^0\colon|\pi_0^{-1}(v)\cap (E_\pi^0\setminus \reg(E_\pi))|\le1
	\end{equation}
	and there exists $s'\colon E^1\to E_\pi^0$ rendering the following diagram commutative:
	\begin{equation}\label{MOSdiag}
		\begin{tikzcd}[ampersand replacement=\&]
			{E_\pi^1} \&\& {E_\pi^0} \\
			\\
			{E^1} \&\& {E^0}.
			\arrow["{s_\pi}", from=1-1, to=1-3]
			\arrow["{\mathrm{pr}_1}", from=1-1, to=3-1]
			\arrow["{\pi_0}", from=1-3, to=3-3]
			\arrow["{s'}", from=3-1, to=1-3]
			\arrow["{s_E}", from=3-1, to=3-3]
		\end{tikzcd}
	\end{equation}
\end{lemma}

\begin{proof}
First, suppose that $\pi\colon E_\pi\to E$ is isomorphic to some out-split graph homomorphism 
$p:=(p_0,p_1)\colon E_O\to E$. Let $\sigma:=(\sigma_0,\sigma_1)\colon E_\pi\to E_O$ be the isomorphism such that~$\pi=p\circ\sigma$. 
Then the condition \eqref{MOScondMulti} for $\pi$ is equivalent to the identical condition for~$p$:
	\begin{equation}\label{MOS2}
		\forall\; v\in E^0\colon|p_0^{-1}(v)\cap (E_O^0\setminus \reg(E_O))|\le1.
	\end{equation}
	Note now that \eqref{MOS2} is evidently true at any regular vertex and at any vertex $v$ such that $|p_0^{-1}(v)|=1$. Next, if $v$ is not regular and
	$|p_0^{-1}(v)|>1$,  then  $v$ must be an infinite emitter, so
	$p_0^{-1}(v)$ will contain exactly one infinite emitter and no sink. Hence, \eqref{MOS2} is satisfied. 
	
	To construct a map $s'$ in \eqref{MOSdiag}, first consider  the  map $s_O\colon E^1\to E_O^0$ given by
	\begin{equation}
		s_O(e) = v^i.
	\end{equation}
	Here $s_E(e)=:v$ and $i$ denotes the index such that $e\in \epsilon_v^i$ and $s_E^{-1}(v)=\epsilon_v^1\cup\dots\cup\epsilon_v^{m(v)}$. 
	Now, it follows directly from the definition of $s_{E_O}$  and the definition of $p$ that the following diagram is commutative:
	\begin{equation*}
		\begin{tikzcd}
			{E_O^1} && {E_O^0} \\
			\\
			{E^1} && {E^0}
			\arrow["{s_{E_O}}", from=1-1, to=1-3]
			\arrow["{p_1}", from=1-1, to=3-1]
			\arrow["{p_0}", from=1-3, to=3-3]
			\arrow["{s_O}", from=3-1, to=1-3]
			\arrow["{s_E}", from=3-1, to=3-3]
		\end{tikzcd}
	\end{equation*}
	Finally, setting  $s':=(\sigma_0)^{-1}\circ s_O$ and remembering that $\sigma$ intertwines $\pi$ and $p$, we obtain 
	the commutative diagram
	\begin{equation*}
		\begin{tikzcd}
	{E_\pi^1} && {E_\pi^0} \\
	\\
	{E_O^1} && {E_O^0} \\
	\\
	{E^1} && {E^0}
	\arrow["{{{s_\pi}}}", from=1-1, to=1-3]
	\arrow["{{{\sigma_1}}}", from=1-1, to=3-1]
	\arrow["{{{\mathrm{pr}_1}}}"', curve={height=18pt}, from=1-1, to=5-1]
	\arrow["{{{\sigma_0}}}"', from=1-3, to=3-3]
	\arrow["{{{\pi_0}}}", curve={height=-18pt}, from=1-3, to=5-3]
	\arrow["{{{p_1}}}", from=3-1, to=5-1]
	\arrow["{{{p_0}}}"', from=3-3, to=5-3]
	\arrow["{{{s'}}}", from=5-1, to=1-3]
	\arrow["{{{s_O}}}"', from=5-1, to=3-3]
	\arrow["{{{s_E}}}", from=5-1, to=5-3]
		\end{tikzcd}
	\end{equation*}
	proving that $\pi$ satifies the second condition in Lemma~\ref{5.6}.
	
On the flip side, suppose that $\pi$ is a surjective standard graph homomorphism satisfying
 both conditions of Lemma~\ref{5.6}. Then $s'\colon E^1\to E_\pi^0$ allows us to define partitions 
of $s_E^{-1}(v)$ for all 
$v\in E^0$. Indeed, as $\pi$ is finite-to-one, first we can write $\pi_0^{-1}(v)=:\{v^1,\dots,v^m\}$ for any $v\in E^0$. Now we can
define $\epsilon_v^i := (s')^{-1}(v^i)$. From the commutativity of \eqref{MOSdiag} and the regularity of $\pi$,
we conclude that $\epsilon_v^i\neq\emptyset$ for $i=1,\dots,m$ form a partition of 
$s_E^{-1}(v)$. Next, \eqref{MOScondMulti} combined with the properness of $\pi$ implies that, if $v$ is an infinite emitter, then $\pi_0^{-1}(v)$ 
contains exactly one infinite 
emitter, which means that exactly one $\epsilon_v^i$ is infinite. Furthermore, if $v$ is a sink, then any $v^i\in \pi_0^{-1}(v)$ is a sink, so
\eqref{MOScondMulti} implies that $|\pi_0^{-1}(v)|=1$. 
Thus, we can construct an out-split $E_O$ of $E$ using our partitions $\epsilon_v^{i}$ 
for all $v\in E^0$. 
	
It remains to show that $E_O$ is isomorphic with~$E_\pi$. To this end,	
note that the properness of $\pi$ allows us to write $\pi_0^{-1}(v)=:\{v_1,\dots,v_{m(v)}\}$ for any $v\in E^0$, which defines~$m(v)$.
Now we can define $\sigma:=(\sigma_0,\sigma_1)\colon E_\pi\to E_O$ by the formulas:
	\begin{equation}\label{outsplitstandard}
		\sigma_0(v_i):=v^i,\quad \sigma_1(e,v_i) := e^i.
	\end{equation}
One can easily verify that $\sigma$ is a graph isomorphism 	satifying $\pi=p\circ\sigma$, where 
$p\colon E_O\to E$ is the out-split graph homomorphism.
\end{proof}
Combining these two lemmas, we obtain the following description of out-split graph homomorphisms.
\begin{theorem}\label{outsplitclass}
	Let $F\to E$ be a graph homomorphism. Then the following are equivalent:
	\begin{enumerate}
	\item $F\to E$ is isomorphic with an out-split graph homomorphism into~$E$.
	\item $F\to E$ is isomorphic with a surjective standard graph homomorphism satsifying the two conditions of Lemma~\ref{5.6}.
	\end{enumerate}
\end{theorem}

Our next step is to show that, if an out-split graph homomorphism $\pi$ factorizes through an out-split graph homomorphism $\tau$
so that $\pi = \tau\circ\rho$, then also $\rho$ is an out-split graph homomorphism. This is in contradistinction with the factorization
of an admissible graph homomorphism through the image graph, where the inclusion of the image graph is always admissible but the 
surjection onto the image graph need not be admissible (see \cite[Remark~3.5]{ht-25} for a counterexample). 
Here we do away with this counterexample by considering only 
surjective graph homomorphisms. More precisely, we have:
\begin{proposition}\label{23Ref}
	Let $E_\pi\to E\leftarrow E_\tau$ be two out-splits, and let 
	 \mbox{$(\rho_0,\rho_1)\colon E_\pi\stackrel{\rho}{\to} E_\tau$} be a surjective graph 
	homomorhpism such that $\pi = \tau\circ\rho$. Then $\rho$ is isomorphic with an out-split graph homomorphism into~$E_\tau$.
\end{proposition}
\begin{proof}
	For starters, recall that, by Theorem~\ref{outsplitclass}, it suffices to consider out-split graph homomorphisms in their standard form.
	Throughout the proof, we will use decorated versions of $\mathrm{pr}_1$ to signify the appropriate canonical surjections onto the first components.	
	Next, observe that our $\rho$ fits into  the following commutative diagram:
	\begin{equation*}
		\begin{tikzcd}[ampersand replacement=\&]
		{E_\pi^1} \&\& {E_\pi^0} \\
		\\
		{E_\tau^1} \&\& {E_\tau^0} \\
		\\
		{E^1} \&\& {E^0}.
		\arrow["{t_\pi}", from=1-1, to=1-3]
		\arrow["{\rho_1}", from=1-1, to=3-1]
		\arrow["{\mathrm{pr_1^\pi}}"', curve={height=24pt}, from=1-1, to=5-1]
		\arrow["{\rho_0}", from=1-3, to=3-3]
		\arrow["{\pi_0}", curve={height=-24pt}, from=1-3, to=5-3]
		\arrow["{t_\tau}", from=3-1, to=3-3]
		\arrow["{\mathrm{pr}_1^\tau}", from=3-1, to=5-1]
		\arrow["{\tau_0}", from=3-3, to=5-3]
		\arrow["{t_E}", from=5-1, to=5-3]
		\end{tikzcd}
	\end{equation*}
	Note now that, as the lower square is a pullback and the outer square is a pullback, so is the upper square, i.e.\ $\rho$ is target-bijective. 
	Hence,	there is the unique bijection $\sigma_1\colon E_\pi^1\to \widetilde{E_\pi^1}:=E_\tau^1\times_{E_\tau^0} E_\pi^0$ such that 
	$\rho_1=\widetilde{\mathrm{pr}}_1\circ\sigma_1$. Our map $\rho_0$ is surjective by assumption, and its surjectivity implies the surjectivity of $\rho_1$ by the 
	pullback property. Also, $\rho_0$ is automatically proper as $\pi_0$ is proper and, much as before, the properness of $\rho_0$ implies the 
	properness of $\rho_1$. As for the regularity of $\rho$, observe first that from the regularity of $\pi$ it follows that, for any $v\in \tau_0^{-1}(\reg(E))$, 
	we have 
	\[
	\rho_0^{-1}(v)\subseteq \pi_0^{-1}(\reg(E)) \subseteq\reg(E_\pi).
	\]
	 Now consider $v\in\reg (E_\tau)\setminus \tau_0^{-1}(\reg(E))$ and 
	suppose that $\rho_0^{-1}(v)$ is not contained in~$\reg(E_\pi)$. Pick a vertex $w\in \rho_0^{-1}(v)$ that is not regular and note that $w$ must be an 
	infintie emitter. Indeed, $\pi_0(w)=\tau_0(v)$ is an infinite emitter, so there is an infinite emitter in the preimage $\pi_0^{-1}(\pi_0(w))$ by the 
	properness and surjectivity of $\pi$. Next, $\pi_0^{-1} (\pi_0(w))$ contains at most one singular vertex by \eqref{MOScondMulti}, 
	so $w$ cannot be a sink. 
	Finally, it follows from the properness of $\rho$ that $\rho(w)=v$ is an infinite emitter, which contradicts the assumption that $v$ is regular. 
	Summarizing, we have proven that $\rho$ is admissible, so we can apply Proposition~\ref{class} to conclude that 
	$\widetilde{\rho}:=\rho\circ\sigma^{-1}$ is a surjective standard graph homomorphism. 
	Here $\sigma:=(\mathrm{id}_{E_\pi^0},\sigma_1)\colon E_\pi\to \widetilde{E_\pi}$ is an isomorphism of graphs.
	
	The next step is to find a map $s'\colon E_\tau^1\to E_\pi^0$ such that we obtain two commutative triangles:
	\begin{equation}\label{sourcelift}
		\begin{tikzcd}[ampersand replacement=\&]
			{\widetilde{E_\pi}^1} \&\& {E_\pi^0} \\
			\\
			{E_\tau^1} \&\& {E_\tau^0.}
			\arrow["{\widetilde{s_\pi}}", from=1-1, to=1-3]
			\arrow["{\widetilde{\mathrm{pr}}_1}", from=1-1, to=3-1]
			\arrow["{\rho_0}", from=1-3, to=3-3]
			\arrow["{s'}", from=3-1, to=1-3]
			\arrow["{s_\tau}", from=3-1, to=3-3]
		\end{tikzcd}
	\end{equation}
	To this end,  let $s_\pi'\colon E^1\to E_\pi^0$ be an analogous map  for 
	the out-split graph homomorphism $E_\pi\to E$. Now consider the composition
	\begin{equation}
		s' := s_\pi'\circ\mathrm{pr}_1^\tau\colon E_\tau^1\longrightarrow E^1\longrightarrow E_\pi^0.
	\end{equation}
	Since $\mathrm{pr}_1^\tau\circ\widetilde{\mathrm{pr}}_1\colon \widetilde{E_\pi}^1\to E^1$ is $\mathrm{pr}_1^\pi\circ\sigma_1^{-1}$, we obtain
	\begin{equation}\label{above}
		s'\circ\widetilde{\mathrm{pr}}_1= s_\pi'\circ\mathrm{pr}_1^\tau\circ\widetilde{\mathrm{pr}}_1 
		= s_\pi'\circ\mathrm{pr}_1^\pi\circ\sigma_1^{-1} = s_\pi\circ\sigma_1^{-1} = \widetilde{s_\pi}.
	\end{equation}
	Thus, we have shown the commutativity of the upper triangle in~\eqref{sourcelift}. 
	On the flip side, to see what happens with $\rho_0\circ s'$, let us investigate the composition 
	$\rho_0\circ s'\circ\widetilde{\mathrm{pr}}_1$. It follows from \eqref{above} and the commutativity of the square in \eqref{sourcelift} that
	\begin{equation}
		\rho_0\circ s'\circ\widetilde{\mathrm{pr}}_1 
		= \rho_0\circ \widetilde{s_\pi} = 
		s_\tau\circ\widetilde{\mathrm{pr}}_1.
	\end{equation}
	As $\widetilde{\mathrm{pr}}_1$ is surjective, we conclude the commutativity of the lower triangle in \eqref{sourcelift}.
	
	To end with, note that, for any vertex $v\in E_\tau^0$, we have
	\begin{equation}
		\rho_0^{-1}(v)\subseteq \rho_0^{-1}(\tau_0^{-1}(\tau_0(v)))=\pi_0^{-1}(\tau_0(v)).
	\end{equation}
	This yields the following estimate
	\begin{equation}
		|\rho_0^{-1}(v)\cap (E_\pi^0\setminus \reg(E_\pi))|\le |\pi_0^{-1}(\tau_0(v))\cap (E_\pi^0\setminus \reg(E_\pi))|\le 1.
	\end{equation}
	Here the last inequality follows from the fact $\pi$ is an out-split graph homomorphism. 
	Consequently, \mbox{$\widetilde{\rho}\colon \widetilde{E_\pi}\to E_\tau$} is isomorphic with an  
	out-split graph homomorphism by Lemma~\ref{5.6}. Finally, as $\rho$ and $\widetilde{\rho}$ are also isomorphic, we conclude the proof. 
\end{proof}


It turns out that an out-split graph homomorphism induces a $*$-isomorphism:
\begin{lemma}\cite[Theorem~3.2]{bp04}  \label{lem:outsplit}
 Let $E_O$ be an out-split of $E$ and \mbox{$f_O\colon E_O\to E$} be the out-split graph homomorphism. 
 Then the  inverse of the contravariantly induced $*$-ho\-mo\-mor\-phism 
\begin{gather}
f_O^*\colon C^*(E)\longrightarrow  C^*(E_O),\nonumber\\
(f_O^*)(S_{v}):=\sum_{i=1 }^{m(v)} S_{v^i},\,v\in E^0,\quad
(f_O^*)(S_{e}):=\sum_{i=1}^{m(t_E(e))} S_{e^i},\, e\in E^1, \label{contraout}
\end{gather}
is given by
\begin{gather}
(f_O^*)^{-1}(S_{v^i})=\begin{cases}S_v & \text{ if }s_E^{-1}(v)=\emptyset\\
\sum_{e \in \epsilon^i_v }S_e S_e^* & \text{ if }s_E^{-1}(v)\neq\emptyset,\;  |\epsilon^i_v|<\infty \\
S_{v}-\sum_{e\in s^{-1}_E(v)\setminus\epsilon^i_v}S_{e} S_{e}^*  &\text{ if } s_E^{-1}(v)\neq\emptyset,\; |\epsilon^i_v|=\infty\\
		\end{cases}\,,\label{inverse}\\
(f_O^*)^{-1}(S_{e^i})=\begin{cases} S_{e}  & \text{ if }s_E^{-1}(t_E(e))=\emptyset\\
S_{e}\sum_{e' \in \epsilon^i_{t_E(e)}}S_{e'} S_{e'}^*  &\text{ if } s_E^{-1}(t_E(e))\neq \emptyset,\,|\epsilon^i_{t_E(e)}|<\infty\\
S_{e}-S_{e}\sum_{e'\in s^{-1}_E(t_E(e))\setminus\epsilon^i_{t_E(e)}}S_{e'} S_{e'}^*  &\text{ if } s_E^{-1}(t_E(e))\neq \emptyset,\,|\epsilon^i_{t_E(e)}|=\infty\\
		\end{cases}\,.\nonumber
\end{gather}
\end{lemma}

We are now ready for the first main application of our new category $\mathsf{REG}$ of regular extended path homomorphisms of graphs.
\begin{theorem}\label{thm:maxoutsplit}
	Let $E_{\widehat{O}}$ be the maximal out-split of a graph $E$ at $X\subseteq\mathrm{reg}(E)$ and let 
	\mbox{$f_{\widehat{O}}^*\colon C^*(E)\to C^*(E_{\widehat{O}})$}
	be the contravarianty induced $*$-isomorphism~\eqref{contraout}. Then the maps $\varphi^0\colon E_{\widehat{O}}^0\to S(E)$ and 
	\mbox{$\varphi^1\colon E_{\widehat{O}}^1\to S(E)$} given by 
	\begin{gather*}
		\varphi^0(v^i) := \begin{cases}
			(v,v) & \text{ if }v\in E^0\setminus X\\
			(e_i,e_i) & \text{ if }v\in X,\,\{e_i\}:=\epsilon_v^{i}
		\end{cases},\\
		\varphi^1(e^i) := \begin{cases}
			(e,t_E(e))  &\text{ if } t_E(e)\in E^0\setminus X,\\
			(ee_i,e_i) & \text{ if }t_E(e)\in X,\,\{e_i\}:=\epsilon_{t_E(e)}^{i}
		\end{cases},
	\end{gather*}
	determine a regular extended path homomorphism of graphs $\varphi\colon S(E_{\widehat{O}})\rightarrow S(E)$ 
	that covariantly induces the inverse \eqref{inverse}:
	$\varphi_*=(f_{\widehat{O}}^*)^{-1}$.
\end{theorem}
\begin{proof} 
		In order for $\varphi^0\colon E_{\widehat{O}}^0\to S(E)$ and $\varphi^1\colon E_{\widehat{O}}^1\to S(E)$ to extend to a well-defined map 
		$\varphi\colon S(E_{\widehat{O}})\to S(E)$, we need to check the three axioms given below Definition~\ref{semi}. Let us start with the vertex axiom,
		followed by the path relation axiom, and end with the first Cuntz--Krieger relation axiom.
		
\underline{Axiom (1):}		
Let $v^i,u^j\in E_{\widehat{O}}^0$. We need to check that
		\begin{equation}
			\varphi^0(v^i)\varphi^0(u^j) = \delta_{v^i,u^j}\varphi^0(v^i).
		\end{equation}
		We shall consider four cases. First, suppose that  $u,v
		\notin X$, so $i=1=j$. Then we have
		\begin{equation}
			\varphi^0(v^1)\varphi^0(u^1) = (v,v)(u,u) = \delta_{u,v}(v,v) = \delta_{u^1,v^1}\varphi^0(v).
			\end{equation}
			Next, assume that $i=1$ and $v\notin X$, but $u^j$ projects onto $u\in X$.  Let $e_j\in E^1$  be the edge such that $\{e_j\}:=\epsilon_u^j$, 
			so $s_E(e_j)=u\in X$. It follows that
			\begin{equation}
				\varphi^0(v^1)\varphi^0(u^j) = (v,v)(e_j,e_j) = \delta_{s_E(e_j),v}(e_j,e_j) = 0.
			\end{equation}
			Swapping the roles of $u$ and $v$ yields an analogous calculation.
			Finally, suppose that both $v^i$ and $u^j$ project, respectively, onto $v,u\in X$, and define $\{e_i\}:=\epsilon_v^i$, 
			$\{e_j'\}:=\epsilon_u^j$. Then
			\begin{equation}
				\varphi^0(v^i)\varphi^0(u^j) = (e_i,e_i)(e_j',e_j') = \delta_{v^i,u^j}(e_i,e_i) = \delta_{v^i,u^j}\varphi^0(v^i).
			\end{equation}
			
			\underline{Axiom (2):}
			 For any $e^i\in E_{\widehat{O}}^1$, we need to show that
			\begin{equation}
				\varphi^0(s_{E_{\widehat{O}}}(e^i))\varphi^1(e^i) = \varphi^1(e^i) = \varphi^1(e^i)\varphi^0(t_{E_{\widehat{O}}}(e^i)).
			\end{equation}
			Here we also distinguish four cases. Let us start with the case where $s_E(e), t_E(e)\notin X$. Then we have
			\begin{gather}
				\varphi^0(s_{E_{\widehat{O}}}(e^i))\varphi^1(e^i) = (s_{E}(e),s_{E}(e))(e,t_{E}(e)) = (e,t_{E}(e)) = \varphi^1(e^i),\nonumber\\
				\varphi^1(e^i)\varphi^0(t_{E_{\widehat{O}}}(e^i)) = (e,t_E(e))(t_E(e),t_E(e)) = (e,t_E(e)) = \varphi^1(e^i).
			\end{gather}
			Next, suppose that $s_{E_{\widehat{O}}}(e^i)=v^j$ for some $j$ and $v\in X$, but $t_{E}(e)\notin X$. 
			In this case $i=1$, $\epsilon_v^j=\{e\}$, and we can compute:
			\begin{gather}
				\varphi^0(v^j)\varphi^1(e^1) = (e,e)(e,t_{E}(e)) = (e,t_{E}(e)) = \varphi^1(e^1),\nonumber\\
				\varphi^1(e^1)\varphi^0(t_{E_{\widehat{O}}}(e^1)) = (e,t_E(e))(t_E(e),t_E(e))) = (e,t_E(e)) = \varphi^1(e^1).
			\end{gather}
			Now assume that $s_{E}(e) \notin X$ and $t_{E}(e) = w\in X$ so  $t_{E_{\widehat{O}}}(e^i)=w^i$. We can define $\{e_i\}:=\epsilon_w^i$ and 
			compute:
			\begin{gather}
				\varphi^0(s_{E_{\widehat{O}}}(e^i))\varphi^1(e^i) = (s_E
					(e),s_{E}(e))(ee_i,e_i) = (ee_i,e_i) = \varphi^1(e^i),\nonumber\\
					\varphi^1(e^i)\varphi^0(w^i) = (ee_i,e_i)(e_i,e_i) = (ee_i,e_i) = \varphi^1(e^i).
			\end{gather}
			Lastly, suppose that $s_{E}(e^i)=v^j$ for some $j$ and $v\in X$ and $t_{E_{\widehat{O}}}(e^i)=w^i$ for some $i$ and $w\in X$. 
			Then $\epsilon_v^j=\{e\}$. Now we can define $\{e_i'\}:=\epsilon_u^i$ and compute:
			\begin{gather}
				\varphi^0(v^j)\varphi^1(e^i) = (e,e)(ee_i',e_i') = (ee_i',e_i') = \varphi^1(e^i),\nonumber\\
				\varphi^1(e^i)\varphi^0(w^i) = (ee_i',e_i')(e_i',e_i') = (ee_i',e_i') = \varphi^1(e^i).
			\end{gather}
			
			\underline{Axiom (3):}
			We need to check that for any pair of edges $x^i,y^j\in E_{\widehat{O}}^1$ the following formula holds:
			\begin{equation}
				\varphi^1(x^i)^*\varphi^1(y^j) = \delta_{x^i,y^j}\varphi^0(t_{E_{\widehat{O}}}(x^i)).
			\end{equation}
			Again, we need to check four cases. First, assume that $t_E(x),t_E(y)\notin X$, so $i=1=j$. Then the above formula becomes evidently true: 
			\begin{equation}
				\varphi^1(x^1)^*\varphi^1(y^1) = (t_E(x),x)(y,t_E(y)) = \delta_{x,y}(t_E(x),t_E(x)) = \delta_{x^1,y^1}\varphi^0(t_{E_{\widehat{O}}}(x^1)).
			\end{equation}
			Next, consider $t_{E}(x)=w\in X$  and $t_{E}(y)\notin X$, so  $t_{E_{\widehat{O}}}(x^i)=w^i$  and $j=1$. Then define $\{e_i\}:=\epsilon_w^i$
			and compute:
			\begin{equation}
				\varphi^1(x^i)^*\varphi^1(y^1) = (e_i,xe_i)(y,t_E(y)).
			\end{equation}
			This product vanishes. Indeed,
			if $y$ were comparable to $xe_i$, then we would have $y=x$, which  is impossible because $y$ and $x$ 
			have different targets by assumption. Next, by taking the $*$ of the above equation, we see that also the product 
			$\varphi^1(y^1)^*\varphi^1(x^j)$ vanishes,  which takes care of the symmetric case. 
			Finally, consider the case $t_{E}(x)=v\in X$ and $t_{E}(y)=u\in X$. Then
			 $t_{E_{\widehat{O}}} (x^i)=v^i$ and $t_{E_{\widehat{O}}} (y^j)=u^j$. Now we can define $\{e_i\}:=\epsilon_v^i$, $\{e_j'\}:=\epsilon_u^j$,
			 and compute:
			\begin{equation}
				\varphi^1(x^i)^*\varphi^1(y^j) = (e_i,xe_i)(ye_j',e_j').
			\end{equation}
			Observe that $xe_i$ and $ye_j'$ are comparable if and only if they are equal, 
			so the above product  doe note vanish only if $x=y$ and $e_i=e_j'$. 
			Hence, we get
			\begin{equation}
				(e_i,xe_i)(ye_j',e_j') = \delta_{x^i,y^j}(e_i,e_i) = \delta_{x^i,y^j}\varphi^0(v^i,v^i).
			\end{equation}
			
			Summarizing, we have shown that our maps $\varphi^0\colon E_{\widehat{O}}^0\to S(E)$ and $\varphi^1\colon E_{\widehat{O}}^1\to S_E$ 
			lift to an extended path homomorphism 
			$\varphi\colon S({E_{\widehat{O}}})\to S(E)$. 
			We still need to show that $\varphi$ is regular and that $\varphi_*=(f_{\widehat{O}}^*)^{-1}$, but it follows immediately from
Corollary~\ref{cor:from.hom.to.reg}.
\end{proof}

Combining the above theorem with Proposition~\ref{23Ref}, we conclude the much desired:
\begin{corollary}
	Let $E_O$ be an out-split of $E$ at $X\subseteq\mathrm{reg}(E)$,  let $E_{\widehat{O}}$ be the maximal out-split of $E$ at $X$,
	and let \mbox{$f_O^*\colon C^*(E)\to C^*(E_O)$}
	be the contravarianty induced $*$-isomorphism~\eqref{contraout}. 
	Then 
	\begin{equation}
		(f_O^*)^{-1} = \varphi_*\circ g^*,\quad g\colon E_{\widehat{O}}\longrightarrow E_O,\quad\varphi\colon E_{\widehat{O}}\longrightarrow E,
	\end{equation}
	where $g$ is isomorphic to an out-split graph homomorphism and $\varphi$ is the regular extended path homomorphism
	of Theorem~\ref{thm:maxoutsplit}.
\end{corollary}
\begin{proof}
	Let $f_{\widehat{O}}\colon E_{\widehat{O}}\to E$ be the out-split graph homomorphism corresponding to $E_{\widehat{O}}$, 
	and let $g:E_{\widehat{O}}\to E_O$ be the obvious surjective graph homomorphism obtained by collapsing vertices and edges according to 
	the partition in~$E_O$. 
	Then the diagram
	\begin{equation}
		\begin{tikzcd}[ampersand replacement=\&]\label{cor:outsplitdiag}
			E \&\& {E_O} \\
			\& {E_{\widehat{O}}}
			\arrow["f_O"', from=1-3, to=1-1]
			\arrow["f_{\widehat{O}}", from=2-2, to=1-1]
			\arrow["g"', from=2-2, to=1-3]
		\end{tikzcd}
	\end{equation}
	commutes, and $g$ is isomorphic to an out-split graph homomorphism by Proposition~\ref{23Ref}. We can now 
    apply the contravariant induction to 
	conclude that $f_{\widehat{O}}^*=g^*\circ f_O^*$. On the other hand, both $f_O^*$ and $f_{\widehat{O}}^*$ are 
    invertible by Lemma~\ref{lem:outsplit},
	so $(f_O^*)^{-1} = (f_{\widehat{O}}^*)^{-1}\circ g^*$. Finally, as
	$\varphi_*=(f_{\widehat{O}}^*)^{-1}$ by Theorem~\ref{thm:maxoutsplit}, we conclude the desired 
	$(f_O^*)^{-1}=\varphi_*\circ g^*$.
\end{proof}
We hope that the above corollary can be extended to general out-splits, including at infinite emitters, 
using groupoid actors~\cite{cm26}.

\subsection{Shift and unital reduction}
In this section, we will consider mutually inverse moves inducing the same path homomorphism of graphs
that covariantly yields a $*$-isomorphism of their graph C*-algebras: shift and unital reduction. 
As in the preceding section, we will need
an additional assumption to write the inverse $*$-isomorphism as functorially induced from morphisms of graphs.
We begin with the shift move considered in the context of Leavitt path algebras.
\begin{definition} [cf.\ \cite{ABRAMS20081983}] \label{def:shift.move}
Let $E$ be a graph, $v \in \reg(E)$, and $\emptyset \neq W \subseteq E^0 \backslash \{ v\}$ such that for all 
$w \in W$ there exists $I_w \neq \emptyset$ admitting for all $w\in W$ and $i \in I_w$ maps 
$\theta^i_w: s_E^{-1}(v)\to s_E^{-1}(w)$ satisfying:
\begin{itemize}
\item $\theta^i_w$ is injective,
\item $t_E|_{s_E^{-1}(v)}=t_E\circ \theta^i_w$,
\item if $\theta^i_w(s^{-1}_E(v))\cap \theta^j_w(s^{-1}_E(v)) \neq \emptyset$ then $i=j$.
\end{itemize}
 Let $E_{(v,W)}$ denote the graph:
\begin{gather*}
E_{(v,W)}^0:= E^0,\quad
E_{(v,W)}^1:= \left(E^1 \setminus \bigcup_{w \in W}\bigcup_{i \in I_w}\theta^i_w(s_E^{-1}(v)) \right) \cup \left(\bigcup_{w \in W}\bigcup_{i \in I_w}\{ x^i_{w,v}\} \right), \\
s_{E_{(v,W)}}(e):=\begin{cases} s_E(e) & \text{for } e \in E^1  \setminus \bigcup_{w \in W}\bigcup_{i \in I_w}\theta^i_w(s_E^{-1}(v))  \\
					w & \text{for } e= x^i_{w,v}\ \\
		\end{cases},\\
t_{E_{(v,W)}}(e):=\begin{cases} t_E(e) & \text{for } e \in E^1  \setminus \bigcup_{w \in W}\bigcup_{i \in I_w}\theta^i_w(s_E^{-1}(v)) \\
					v & \text{for } e= x^i_{w,v}\ \\
		\end{cases}.
\end{gather*}
We say that $E_{(v,W)}$ is formed by performing a \emph{shift} from $W$ to $v$ on $E$.
\end{definition}

In what follows, we will need to know how shifts interact with out-splits at regular vertices. To this end,
we need to claim two lemmas. A proof of the first lemma is straightforward, so we leave it to the reader.  
\begin{lemma}\label{commutelemma}
    Let $E$ be a graph, $v$ and $v'$ be distinct regular vertices, $\emptyset\neq W\subseteq E^0\setminus\{v,v'\}$ 
    and $E_{(v,W)}$ and $E_{(v',W)}$ be  shifted graphs. Also, let us define $\Theta,\Theta'\subseteq s_E^{-1}(W)$ 
    as follows:
    \begin{equation*}
        \Theta:=\bigcup_{w\in W,\, i\in I_w}\theta_w^i(s_E^{-1}(v)),\qquad \Theta':=\bigcup_{w\in W,\, j\in J_w}\theta_w^{'j}(s_E^{-1}(v')).
    \end{equation*}
    Assume $\Theta\cap \Theta'=\emptyset$. Then
    \begin{equation*}
        (E_{(v,W)})_{(v',W)}=(E_{(v',W)})_{(v,W)}.
    \end{equation*}
\end{lemma}
For a fixed $W$, the above lemma allows us to define the shifted graph $E_{(V,W)}$ with respect to
a finite set $V\subseteq\reg(E)$,  whose every pair of vertices satisfies the assumptions of the lemma,
as the iterated shift over $v\in V$. We are now ready to prove the second lemma.
\begin{lemma}\label{outshiftlemma}
    Let $E$ be a graph, $E_{(v,W)}$ be its shift, $E_{O}$ be an out-split of $E$ at~$v$, and $(E_{(v,W)})_O$
be the same out-split at $v$ of~$E_{(v,W)}$. Also, let $V\subseteq E_O^0$ be the finite set of all vertices 
that are mapped to $v$ via the out-split graph homomorphism. Then the shift and out-split moves commute in the 
following sense:
    \begin{equation*}
        (E_{(v,W)})_O = (E_{O})_{(V,W)}.
    \end{equation*}
\end{lemma}
\begin{proof}
    For starters, let us denote by $\theta_w^i\colon s_E^{-1}(v)\to s_E^{-1}(w)$ the $\theta$'s that determine
    $E_{(v,W)}$, and by $\{\epsilon_u\}_{u\in V}$ the partition of 
    $s_E^{-1}(v)$ that we use to construct~$E_{O}$. Here we use the identification \eqref{outsplitstandard} to describe the out-split. Also, let us define
    \begin{equation}\label{BigTheta}
        \Theta:=\bigcup_{w\in W,\, i\in I_w}\theta_w^i(s_E^{-1}(v))\subseteq s_E^{-1}(W).
    \end{equation}
    To construct $(E_{O})_{(V,W)}$, we iterate the shift for each $u\in V$. For a fixed $u\in V$, we have
    \begin{equation}
        s_{E_O}^{-1}(u) = \{(e,t_E(e)) \;| \; e\in\epsilon_u\setminus t_E^{-1}(v)\}\sqcup\{(l,u') \;| \; l\in 
        \epsilon_u\cap t_E^{-1}(v),u'\in V\}.
    \end{equation}
    For every $w\in W,\, i\in I_w$, we define
    \begin{equation}
      \theta_{u,w}^i\colon  s_{E_O}^{-1}(u)\ni (e,z)\longmapsto \theta_{u,w}^i(e,z) := (\theta_w^i(e),z)
      \in s_{E_O}^{-1}(w).
    \end{equation}
    Next, the  image of all $\theta$'s in $E_O$ is
    \begin{equation}\label{BigThetaO}
       \Theta_O:=\bigcup_{u\in V,\, w\in W,\,i\in I_w}\theta_{u,w}^i(s_{E_O}^{-1}(u)). 
    \end{equation} 
    Note that, as $\theta_w^i$'s preserve targets, we infer that $\Theta_O=\Theta\times_{E^0}E_O^0\subseteq E^1\times_{E^0}E_O^0$.
    Now, the edges of $(E_{O})_{(V,W)}$ form the following set:
    \begin{equation}
        (E_{O})_{(V,W)}^1 = (E_O^1\setminus \Theta_O) \sqcup \{x_{w,u}^i\; | \; w\in W,i\in I_w,u\in V\}.
    \end{equation}
    Here the newly introduced edges $x_{w,u}^i$ start at $w$ and end at $u$. 
    
    On the other hand, performing the shift move followed by the out-split move, we obtain the graph 
    $(E_{(v,W)})_O$ with the following set of edges:
    \begin{equation}
        (E_{(v,W)})_O^1 = \big(\!(E^1\setminus\Theta)\sqcup
        \{x_{w,v}^i\; | \;w\in W,i\in I_w\} \!\big)\times_{E^0}E_O^0.
    \end{equation}
    To compare the two constructions, observe first that the sets of vertices clearly coincide.
     Finally, we combine the bijective correspondence 
    $(x_{w,v}^i,u)\leftrightarrow x_{w,u}^i$ with the obvious equalities
    \begin{equation}
        E_O^1\setminus\Theta_O=(E^1\times_{E^0}E_O^0) \setminus (\Theta\times_{E^0}E_O^0) =(E^1\setminus \Theta)\times_{E^0}E_O^0
    \end{equation}
    to see that the sets of edges also coincide.
\end{proof}

Our next step is to turn the shift construction to an admissible path homomorphism of graphs.
\begin{proposition}\label{shiftprop}
Let $E$ be a graph, and $E_{(v,W)}$ be its shift. 
Then the following formulas define an admissible path homomorphism
$f\colon E\to E_{(v,W)}$ (called \emph{shift path homomorphism}):
\label{shiftcov}
\begin{gather*}
f(v)=v \text{ for } v \in E^0,\\
f(e)=\begin{cases} e & \text{ for } \in E^1  \setminus \bigcup_{w \in W}\bigcup_{i \in I_w}\theta^i_w (s_E^{-1}(v)) \\
			x^i_{w,v}(\theta^i_w)^{-1}(e) & \text{ for } e \in \theta^i_w(s_E^{-1}(v)), \, w \in W, \,  i \in I_w\\
\end{cases}.
\end{gather*}
\end{proposition}
\begin{proof}
Note first that $f$ is obviously injective on vertices. To see that it is monotonic, suppose that $f(e)\preceq f(e')$ for some edges $e,e'\in E^1$. Clearly,
 if both $e,e'\notin   \bigcup_{w \in W}\bigcup_{i \in I_w}\theta^i_w(s_E^{-1}(v))$, then $e=e'$. Otherwise, if $e = \theta_w^i(y)$ for some  $y\in s_E^{-1}(v)$,
then $f(e) = x^i_{w,v}y \preceq f(e')$, so
$f(e')$ has length at least two, which is possible only when $e'=\theta^j_{w^\prime}(y')$ for some $y'\in s_E^{-1}(v)$.
Hence, $x^i_{w,v}y\preceq x^j_{w^\prime,v}y'$, so $i=j$, $w=w^\prime$, and $y=y'$. Now it follows that and $e=\theta^i_w(y)=\theta^j_{w^\prime}(y^\prime)=e'$. 
Next, if $e'=\theta^j_{w^\prime}(y')$ for some $y'\in s_E^{-1}(v)$, then $f(e)\preceq  x^j_{w^\prime,v}y'$, so
$f(e)=x^j_{w^\prime,v}$ or $f(e)=x^j_{w^\prime,v}y'$. The former case is excluded as $x^j_{w^\prime,v}$ is not in the image of~$f$, and the latter case implies that $e=\theta^j_{w^\prime}(y')=e'$. 

Finally, the regularity of $f$ is obvious everywhere except for the vertices of $W \cap \reg(E)$. For $w \in W \cap \reg(E)$, all edges not in $\bigcup_{i \in I_w}\theta_w^i(s_E^{-1}(v))$
are mapped to themselves, and all edges in $\bigcup_{i \in I_w}\theta_w^i(s^{-1}(v))$ are mapped bijectively into paths of length two obtained by extending the special edges
$x^i_{w,v}$ by all edges emitted from their target, which is the regular vertex~$v$. Hence, $f$ is regular.
\end{proof}

Now we can apply \eqref{covf} to the shift path homomorphism $f$ in Proposition~\ref{shiftprop} to obtain the 
$*$-homomorphism $f_*\colon C^*(E) \to  C^*(E_{(v,W)})$
given by
\[\label{shiftiso}
f_*(S_e)=\begin{cases}
S_e &\text{ for } e \in E^0 \sqcup \big(E^1  \setminus  \bigcup_{w \in W}\bigcup_{i \in I_w}\theta^i_w(s_E^{-1}(v))\big) \\
S_{x^i_{w,v}}S_{(\theta^i_w)^{-1}(e)} &\text{ for } e \in \bigcup_{w \in W}\bigcup_{i \in I_w}\theta^i_w(s_E^{-1}(v))
\end{cases}.
\]
It turns out to be a $*$-isomorphism with the inverse provided by the following lemma.

\begin{lemma}\label{shiftinv}(cf.\ \cite{ABRAMS20081983})
The   inverse of the $*$-homomorphism \eqref{shiftiso} is determined by
\begin{equation}\label{shiftinverse}
(f_*)^{-1}(S_e)=\begin{cases}
S_e &\text{ for } e \neq  x^i_{w,v}, \, w \in W, \, i \in I_w,\\
\sum_{ y\in s^{-1}_E(v)} S_{\theta^i_w(y)}S_y^*  &\text{ for } e =  x^i_{w,v}, \, w \in W, \, i \in I_w .
\end{cases}
\end{equation}
\end{lemma}

\begin{proof}
Let us first check that the proposed inverse respects the necessary relations.  To prove that $(f_*)^{-1}$ is well defined, we use the fact that $C^*(E_{(v,W)} )$ is the universal $C^*$ algebra given by generators and relations as seen in \cite[Definition~1.2.3]{aam-17}. Note that all path algebraic relations are clearly satisfied, so we only need to verify the Cutz--Krieger relations. Concerning the first Cuntz--Krieger relation \eqref{CK1}, the only nontrivial cases are when at least one arrow is $x^i_{w,v}$. Firstly, we verify
$$ (f_*)^{-1}(S^*_{x^i_{w,v}} S_{x^i_{w,v}})= \left(\sum_{y \in s^{-1}_E(v)}S_{y}S^*_{\theta^i_w(y)}\right) \left(\sum_{z \in s^{-1}_E(v)}S_{\theta^i_w(z)}S_z^*\right) =  \sum_{y \in s^{-1}_E(v)}S_{y}S^*_{y} = (f_*)^{-1}(S_v).  $$
Secondly, we check the case when one arrow is $x^i_{w,v}$ and the other is $x^j_{w^\prime,v}$ when $w \neq w^\prime$:
\[ (f_*)^{-1}(S^*_{x^i_{w,v}} S_{x^j_{w^\prime,v}})= \left(\sum_{y \in s^{-1}_E(v)}S_{y}S^*_{\theta^i_w(y)}\right) \left( \sum_{z \in s^{-1}_E(v)} S_{\theta^j_{w^\prime}}  (z)S_z^* \right) = 0.  \]
This product is zero by the path relations of $C^*(E)$. Thirdly, consider the case when one arrow is $x^i_{w,v}$ and the other is $x^j_{w,v}$ when $i \neq j $:
\[ (f_*)^{-1}(S^*_{x^i_{w,v}} S_{x^j_{w,v}})= \left(\sum_{y \in s^{-1}_E(v)}S_{y}S^*_{\theta^i_w(y)}\right) \left(\sum_{z \in s^{-1}_E(v)}S_{\theta^j_w(z)}S_z^*\right) = 0.  \]
This product is zero because the images of $\theta^i_w$ and $\theta^j_w$ are disjoint for $i \neq j$. Next, we check the case when  $e \neq x^i_{w,v}$:
\[ (f_*)^{-1}(S^*_e S_{x^i_{w,v}})= S^*_e \sum_{y \in s^{-1}_E(v)}S_{\theta^i_w(y)}S_y^*=  \sum_{y \in s^{-1}_E(v)} S^*_e S_{\theta^i_w(y)}S_y^*= 0.\] Here $e \neq \theta^i_w (y)$, so $S_e^*S_{\theta^i_w(y)}=0$ for any $ y \in s_E^{-1}(v)$.
Finally,
\[ (f_*)^{-1}(S^*_{x^i_{w,v}} S_{e}) =   (f_*)^{-1}(S^*_e S_{x^i_{w,v}})^*=  0. \]

Concerning the second Cuntz--Krieger relation \eqref{CK2}, we need to check that $w$, should it 
be a regular vertex, satisfies the necessary relation after the map $(f_*)^{-1}$. To this end, we compute:
\begin{align} \nonumber &\sum_{e \in s_{E_{(v,W)}}^{-1}(w)  \backslash \left( \cup_{i \in I_w} \{ x^i_{w,v} \} \right)} S_e S_e^* +\sum_{i \in I_w}S_{x^i_{w,v}}S_{x^i_{w,v}}^*   \xmapsto{f_*^{-1}}   \\ & \sum_{e \in s_{E}^{-1}(w)  \backslash \left( \cup_{i \in I_w}  \theta^i_w(s^{-1}_E(v)) \right)}S_e S_e^* +\sum_{i \in I_w} \left( \sum_{y \in  \theta^i_w(s_E^{-1}(v))} S_{\theta^i_w(y)}S_y^* \right) \left( \sum_{z \in  \theta^i_w(s_E^{-1}(v))} S_{z}S_{\theta^i_w(z)}^* \right) = \nonumber  \\ & \sum_{e \in s_E^{-1}(w) \backslash \left( \cup_{i \in I_w}\theta^i_w(s_E^{-1}(v)) \right) }S_e S_e^* + \sum_{i \in I_w}\sum_{y \in  \theta^i_w(s_E^{-1}(v))}  S_{\theta^i_w(y)} S_{\theta^i_w(y)}^*  =  \sum_{e \in s_E^{-1}(w)}S_e S_e^*=(f_*)^{-1}(S_w).  \nonumber 
 \end{align}
All other cases are trivial.

To end the proof, we check that the compositions are inverses of each other. The only instances  where it is not immediate are the following.
First, for $e \in \theta^i_w(s_E^{-1}(v))$, we compute:
\begin{align}&S_e  \xmapsto{f_*}  S_{x^i_{w,v}}S_{(\theta^i_w)^{-1}(e)} \xmapsto{f_*^{-1}}   \left( \sum_{y \in s^{-1}_E(v)}S_{\theta^i_w(y)}S_y^* \right)  S_{(\theta^i_w)^{-1}(e)} =  \nonumber \\ & S_{\theta^i_w((\theta^i_w)^{-1}(e))}S_{(\theta^i_w)^{-1}(e)}^* S_{(\theta^i_w)^{-1}(e)} = S_e. \end{align}
Next, for $x^i_{w,v}$, we compute:
\[ 
S_{x^i_{w,v}}  \xmapsto{f_*^{-1}}   \sum_{y \in  s_E^{-1}(v)} S_{\theta^i_w(y)}S_y^*  \xmapsto{f_*}  \sum_{y \in  s_{E}^{-1}(v)} S_{x^i_{w,v}}S_{(\theta^i_w)^{-1}(\theta^i_w(y))} S_y^* = S_{x^i_{w,v}}  \sum_{y \in s_{E}^{-1}(v)} S_{y}S_y^*=S_{x^i_{w,v}}. \nonumber  
\]
Thus these are inverse  homomorphisms between $C^*(E)$ and $C^*(E_{(v,W)})$. 
\end{proof}

The above lemma allows to unravel the inverse $*$-isomorphism in terms of morphisms of graphs. 
Here the key step is:
	\begin{theorem}\label{thm:shiftinv}
		Let $E$ be a graph and $E_{(v,W)}$ be its shift. Assume that $s_E^{-1}(v)=\{ y \}$.
		Then the maps $\varphi^0\colon E_{(v,W)}^0\rightarrow S(E)$ and 
        $\varphi^1\colon E_{(v,W)}^1\rightarrow S(E) $ given by
		\begin{gather*}
			\varphi^0(u) := \begin{cases}
				(u,u) & \text{for } u\neq v,\\
				(y,y) & \text{for } u=v,\\
			\end{cases}\\
			\varphi^1(e) := \begin{cases}
				(e,t_E(e))  & \text{for } t_{E}(e) \neq v,\\
				(ey,y) &  \text{for } e\neq x^i_{w,v}\,,\,t_{E}(e) =v,\\
				(\theta^i_w(y),y) & \text{for } e=x^i_{w,v}\,,
			\end{cases}
		\end{gather*}
		determine a regular extended path homomorphism of graphs $\varphi\colon S(E_{(v,W)})\rightarrow S(E)$ that covariantly induces the inverse $*$-isomorphism \eqref{shiftinverse}:
		$\varphi_*=(f_*)^{-1}$.
	\end{theorem}
	
	\begin{proof}
		Similar as to the proof of Theorem \ref{thm:maxoutsplit} we will check the three necessary axioms that $\varphi^0$ and $\varphi^1$ need to satisfy in order to lift to a map $\varphi:S(E_{(v,W)})\to S(E)$.
		
		\underline{Axiom (1):}	Let $u,u'\in E_{(v,W)}^0$. We need to check that
		\begin{equation}
			\varphi^0(u)\varphi^0(u')=\delta_{u,u'}\varphi^0(u).
		\end{equation}
		We can distinguish four different cases: first suppose that $u,u'\neq v$. Then everything trivializes as we get
		\begin{equation}
			\varphi^0(u)\varphi^0(u') = (u,u)(u',u') = \delta_{u,u'}(u,u) = \delta_{u,u'}\varphi^0(u).
		\end{equation}
		Next, assume that $u=v$ and $u'\neq v$. We see that
		\begin{equation}
			\varphi^0(v)\varphi^0(u') = (y,y)(u',u') = 0,
		\end{equation}
		as $u'$ is not be the starting vertex of $y$. The dual case of $u\neq v, u'=v$ follows by commutativity of idempotents. 
		Lastly, if $u=u'=v$ then we get
		\begin{equation}
			\varphi^0(v)\varphi^0(v) = (y,y)(y,y) = (y,y)=\varphi^0(v)
		\end{equation}
		and we are done with Axiom (1).
		
		\underline{Axiom (2):} For $e\in E_{(v,W)}^1$ we need to check that the following formulas hold
		\begin{equation}
			\varphi^0(s_{E_{(v,W)}}(e))\varphi^1(e) = \varphi^1(e) = \varphi^1(e)\varphi^0(t_{E_{(v,W)}}(e)).
		\end{equation}
		Again, this will be done via casework. Suppose that $e=x^i_{w,v}$, then we have
		\begin{gather}
			\varphi^0(w)\varphi^1(x^i_{w,v}) = (w,w)(\theta^i_w(y),y) = (\theta^i_w(y),y)=\varphi^1(e),\\
			\varphi^1(x^i_{w,v})\varphi^0(v) = (\theta^i_w(y),y)(y,y) = (\theta(y),y)=\varphi^1(e).
		\end{gather}
		Next, suppose that $e\in E_{(v,W)}^1\setminus ( \{ x^i_{w,v} \, | \, w \in W, \, i \in I_w \} \sqcup \{ y\} )$ and $t_E(e)\neq v$. Note that $e\neq y$ forces the condition $s_E(e)\neq v$. Hence
		\begin{gather}
			\varphi^0(s_{E_{(v,W)}}(e))\varphi^1(e) = (s_E(e),s_E(e))(e,t_E(e)) = (e,t_E(e))=\varphi^1(e),\\
			\varphi^1(e)\varphi^0(t_{E_{(v,W)}}(e)) = (e,t_E(e))(t_E(e),t_E(e)) = (e,t_E(e))=\varphi^1(e).
		\end{gather}
		In the same vein, supposing that $e\in E_{(v,W)}^1\setminus ( \{ x^i_{w,v} \, | \, w \in W, \, i \in I_w \} \sqcup \{ y\} )$ and $t_E(e)=v$ we get
		\begin{gather}
			\varphi^0(s_{E_{(v,W)}}(e))\varphi^1(e) = (s_E(e),s_E(e))(ey,y) = (ey,y)=\varphi^1(e),\\
			\varphi^1(e)\varphi^0(v) = (ey,y)(y,y) = (ey,y)=\varphi^1(e).
		\end{gather}
		Finally, let $e=y$. Here we need to treat the cases of $y$ being a loop separately. If $t_E(y)\neq y$, then we get
		\begin{gather}
			\varphi^0(v)\varphi^1(y) = (y,y) (y,t_{E}(y)) = (y,t_E(y)) = \varphi^1(y),\\
			\varphi^1(y)\varphi^0(t_E(y)) = (y,t_E(y))(t_E(y),t_E(y)) = (y,t_E(y)) = \varphi^1(y).
		\end{gather}
		If $y$ is a loop, then we obtain
		\begin{gather}
			\varphi^0(v)\varphi^1(y) = (y,y)(yy,y) = (yy,y) = \varphi^1(y),\\
			\varphi^1(y)\varphi^0(v) = (yy,y)(y,y) = (yy,y) = \varphi^1(y).
		\end{gather}
		\underline{Axiom (3):} Here we need to check that for $e,e'\in E_{(v,W)}^1$ we have
		\begin{equation}
			\varphi^1(e)^*\varphi^1(e') = \delta_{e,e'}\varphi^0(t_{E_{(v,W)}}(e)).
		\end{equation}
		As always, we shall consider some cases. Let us start with the trivial case, $e,e'\in E^1$ and $t_{E_{(v,W)}}(e),t_{E_{(v,W)}}(e')\neq v$. In this scenario we get
		\begin{equation}
			\varphi^1(e)^*\varphi^1(e') = (t_E(e),e)(e',t_E(e')) = \delta_{e,e'}(t_E(e),t_E(e)) = \delta_{e,e'}\varphi^0(t_{E_{(v,W)}}(e)).
		\end{equation}
		Next, let $e,e'\neq x^i_{w,v}$ for any $i \in I_w$ and $w \in W$, and  $t_{E_{(v,W)}}(e)=v$ and $t_{E_{(v,W)}}(e')\neq v$. Then we see that
		\begin{equation}
			\varphi^1(e)^*\varphi^1(e') = (y,ey)(e',t_E(e')) = 0.
		\end{equation}
		Observe that the above product vanishes as $e'\neq e$ because these edges have different targets. The dual case follows by taking * on both sides of the above equation. Next, consider $e,e'\neq x^i_{w,v}$ for any $i \in I_w$ and $w \in W$, and and $t_{E_{(v,W)}}(e)=t_{E_{(v,W)}}(e')=v$. We have
		\begin{equation}
			\varphi^1(e)^*\varphi^1(e') = (y,ey)(e'y,y) = \delta_{e,e'}(y,y) = \delta_{e,e'}\varphi^0(v).
		\end{equation}
		Next let us assume that $e=x^i_{w,v}$, which implies that
		\begin{gather*}
			\varphi^1(x^i_{w,v})^*\varphi^1(e') = \begin{cases}
				(y,\theta^i_w(y))(\theta^i_w(y),y) = (y,y) = \varphi^0(v) & \text{if } e'=x^i_{w,v},\\
				(y,\theta^i_w(y))(e',t_E(e')) = 0 & \text{if } e'\neq x^j_{w^\prime,v},\,\forall j,w^\prime,\, t_E(e')\neq v,\\
				(y,\theta^i_w(y))(e'y,y) = 0 &  \text{if } e'\neq x^j_{w^\prime,v},\,\forall j,w^\prime,\, t_E(e') = v,\\
				(y,\theta^i_w(y))(\theta^j_{w^\prime}(y),y) = 0 & \text{if } e'= x^j_{w^\prime,v}, \text{ for } i \neq j \text{ or } w \neq w^\prime. 
			\end{cases}
		\end{gather*}
		The second and third product vanishes as both $e'$ and $e'y$ can be comparable to $\theta^i_w(y)$ if and only if $e'=\theta^i_w(y)$ which is not the case as $e'\in E^1\setminus \big(\bigcup_{w^\prime\in W} \bigcup_{j \in I_w} \theta^j_{w^\prime}(y) \big)$ by hypothesis. The fourth case is due to the images of $\theta^i_w$ being disjoint. The dual case follows by taking * on both sides of the above equations. Therefore the maps $\varphi^0,\varphi^1$ determine a extended path homomorphism $\varphi:S(E_{(v,W)})\to S(E)$.

		To finish the proof, we show that $\varphi$ and $(f_*)^{-1}$ satisfy the conditions of Corollary \ref{cor:from.hom.to.reg}. Most equalities are immediate. We check those that are not. The first is for $v$:
		\begin{equation}
			S_{\varphi(v,v)}=S_{(y,y)}=S_yS_y^*=S_v=(f^*)^{-1}(S_v),
		\end{equation}
		The second is for $e\in E^1\setminus \big(\bigcup_{w \in W} \bigcup_{i \in I_w}\theta^i_w(y)\big)$ such that $t_{E_{(v,W)}}(e)=v$:
		\begin{equation}
			S_{\varphi(e,t_E(e))}=S_{(ey,y)}=S_eS_yS_y^*=S_eS_v=S_e=(f^*)^{-1}(S_e).
		\end{equation}
		Finally, for $x^i_{w,v}$, we have
		\begin{equation}
			S_{\varphi(x^i_{w,v},w)}=S_{(\theta^i_w(y),y)}=S_{\theta^i_w(y)}S_y^*=(f^*)^{-1}(S_{x^i_{w,v}}).
		\end{equation}
		By Corollary \ref{cor:from.hom.to.reg}, $\varphi$ is regular and $\varphi_*=(f^*)^{-1}$, concluding the proof.
\end{proof}

	If a regular vertex 
    $v$  does not emit an edge that is a loop, we can also describe the inverse of any isomorphism associated 
    with a shift from $W$ to $v$ in terms of an admissible graph homomorphism and a regular 
    extended path homomorphism. 
	We hope that the general shift, including regular vertices supporting loops, can be handled by extended 
    relation morphisms generalizing relation morphisms  defined in~\cite{cdh25}. 
	\begin{corollary}
		Let $E$ be a graph and $E_{(v,W)}$ its shift such that $s_E^{-1}(v)\cap t_E^{-1}(v)=\emptyset$.
        Then the inverse to $f_*\colon C^*(E)\to C^*(E_{(v,W)})$, as given in 
		$\eqref{shiftinverse}$, can be realized as a composition
		\begin{equation}
			(f_*)^{-1} = \varphi_*\circ g^*,
		\end{equation}
		where $g$ is an admissible graph homomorphism and $\varphi$ is a regular extended path homomorphism.
	\end{corollary}
	\begin{proof}
		Consider the maximal out-splits at $v$ of $E$ and 
        $E_{(v,W)}$, and denote by $V\subseteq E_{\widehat{O}}^0$ the
        finite set of all vertices that are mapped to $v$ via 
        maximal out-split graph homomorphisms. Next, let $\widehat{f}\colon E_{\widehat{O}}\to E$ and 
        $\widehat{F}\colon (E_{(v,W)})_{\widehat{O}}\to E_{(v,W)}$ be corresponding out-split graph homomorphisms.
        On the other hand, let
        $F\colon E_{\widehat{O}}\to (E_{\widehat{O}})_{(V,W)}$ be the iteration over $u\in V$ of the corresponding 
        shift path homomorphisms. Since, by Lemma~\ref{outshiftlemma}, the graphs $(E_{\widehat{O}})_{(V,W)}$ and 
        $(E_{(v,W)})_{\widehat{O}}$ are the same, we obtain the following commutative 
        diagram of C*-algebras:
		\[
        \begin{tikzcd}[ampersand replacement=\&]\label{CommDiagShift}
	{C^*(E)} \&\&\& {C^*(E_{(v,W)})} \\
	\\
	{C^*(E_{\widehat{O}})} \&\& {C^*((E_{\widehat{O}})_{(V,W)})} \& {C^*((E_{(v,W)})_{\widehat{O}})}
	\arrow["{{f_*}}", from=1-1, to=1-4]
	\arrow["{{\widehat{f}^*}}"', shift right, from=1-1, to=3-1]
	\arrow["{{\widehat{F}^*}}", shift left, from=1-4, to=3-4]
	\arrow["{{F_*}}", shift left, from=3-1, to=3-3]
	\arrow["{=}", from=3-3, to=3-4]
\end{tikzcd}.
\]

Indeed, one can track what happens with the generators $S_u$ and $S_e$. 
    Since shifts yield identities on the vertex projections  and  we perform the same out-split on $E$ 
    as we do on~$E_{(v,W)}$, the commutativity of \eqref{CommDiagShift} is true for all vertex projections 
    $S_u\in C^*(E)$. Regarding the partial isometries $S_e$, 
    recall first the sets $\Theta$ and  $\Theta_O$ defined in \eqref{BigTheta} and \eqref{BigThetaO}, respectively.
    Here it is crucial to observe that, since $v$ does not emit a loop,  $\Theta\cap t_E^{-1}(v)=\emptyset$, so 
    the canonical projection identifies the sets $\Theta_O$ and $\Theta$. 
    Now one has to consider three cases. The first case $e\notin t_E^{-1}(v)$ and $e\notin\Theta$ 
    is trivial because all $*$-homomorphisms map $S_e$ to itself.
    Next, if $e\in t_E^{-1}(v)$, then $e\notin\Theta$, so 
    $S_e$ is mapped to the following elements in \eqref{CommDiagShift}:
    \[
    \begin{tikzcd}[ampersand replacement=\&]
	{S_e} \&\& {S_e} \\
	{\sum_{u\in V}S_{(e,u)}} \& {\sum_{u\in V}S_{(e,u)}} \& {\sum_{u\in V}S_{(e,u)}}
	\arrow[maps to, from=1-1, to=1-3]
	\arrow[maps to, from=1-1, to=2-1]
	\arrow[maps to, from=1-3, to=2-3]
	\arrow[maps to, from=2-1, to=2-2]
	\arrow["{=}"{description}, draw=none, from=2-2, to=2-3]
\end{tikzcd}.
\]
     The last case is $e\notin t_E^{-1}(v)$ and $e\in\Theta$. Then we define $e':=(\theta_w^i)^{-1}(e)$ for the
     unique $i\in I_w$, $w:=s_E(e)$, and take $u:=s_{E_O}((e',t_E(e)))$. Then tracking what happens with $S_e$ 
     in~\eqref{CommDiagShift}, we see that
    \[
    \begin{tikzcd}[ampersand replacement=\&]
	{S_e} \&\& {S_{x_{w,v}^i}S_{e'}} \\
	{S_{(e,t_E(e))}} \& {S_{x_{w,u}^i}S_{(e',t_{E}(e))}} \& 
    {\left(\sum_{u'\in V}S_{(x_{w,v}^i,u')}\right)S_{(e',t_{E}(e))}}
	\arrow[maps to, from=1-1, to=1-3]
	\arrow[maps to, from=1-1, to=2-1]
	\arrow[maps to, from=1-3, to=2-3]
	\arrow[maps to, from=2-1, to=2-2]
	\arrow["{=}"{description}, draw=none, from=2-2, to=2-3]
\end{tikzcd}.
\]
Here we use the identification $(x_{w,v}^i,u)=x_{w,u}^i$ and the fact that
\begin{equation}
  S_{(x_{w,v}^i,u')}S_{(e',t_{E}(e))}=  \delta_{u',u}S_{(x_{w,v}^i,u)}S_{(e',t_{E}(e))}.
\end{equation}
Hence, the diagram \eqref{CommDiagShift} commutes, as claimed.
		
		Finally, since $v$ does not emit an edge that is a 
		loop, every $u\in V$ has exactly one edge emitted from it. Therefore,
        we can iterate the application of Theorem~\ref{thm:shiftinv} to  claim that there exists a regular 
        extended path homomorphism 
        $\psi\colon (E_{\widehat{O}})_{(V,W)}\to E_{\widehat{O}}$ such that $\psi_*=(F_*)^{-1}$. 
        On the other hand, by 
        Theorem~\ref{thm:maxoutsplit}, there exists a regular extended path homomorphisms 
        $\phi\colon E_{\widehat{O}}\to E$ such that $\phi_* = (\widehat{f}^*)^{-1}$.
		 We thus obtain the desired decomposition for $(f_*)^{-1}$:
		\begin{equation}
			(f_*)^{-1} = (\phi\circ\psi)_*\circ\widehat{F}^*.
		\end{equation}
        Taking $g:=\widehat{F}$ and $\varphi:=\phi\circ\psi$, we conclude the proof.
	\end{proof}

    We now pass to the unital-reduction  move considered in the context of graph C*-algebras.
\begin{definition} [cf.\ \cite{eilers2019refined}]  \label{ureduction}
Let $E$ be a graph, $v \in \reg(E)$, and $\emptyset \neq X \subseteq t_E^{-1}(v)\setminus s_E^{-1}(v)$. 
We define the graph $E_{R+}$ as follows:
\begin{gather*}
E_{R+}^0:= E^0,\qquad
E_{R+}^1:= (E^1 \setminus X)    \sqcup \{ xe \, | \,x \in X \text{ and } s_E(e)=v \},\\
s_{E_{R+}}(p):=\begin{cases} s_E(p) & \text{for } p \in E^1 \setminus X \\
					s_E(x) & \text{for } p \in  \{ xe  \, | \,x \in X \text{ and } s_E(e)=v \}  \\
		\end{cases},\\
t_{E_{R+}}(p):=\begin{cases} t_E(p) &  \text{for } p \in E^1 \setminus X \\
					t_E(e) & \text{for } p \in  \{ xe  \, | \,x \in X \text{ and } s_E(e)=v \}  \\
		\end{cases}.
\end{gather*}
We say that $E_{R+}$ is formed by performing a \emph{unital reduction} at $v$ on $E$ with respect to~$X$.
\end{definition}

We recover the definition of unital reduction in \cite{eilers2019refined} by taking $v\in \reg(E)$ a vertex that 
does not support any loop and $X=t_E^{-1}(v)$.

\begin{proposition}\label{shiftuniinv}
The unital reduction is the inverse of the shift move. 
\end{proposition}
\begin{proof}
We will argue this by comparing the vertex set, the edge set, as well as the source and target maps of the 
respective graphs.
We first start with the graph $E$ and do a unital reduction at  $v \in \reg(E)$  with
$\emptyset \neq X \subseteq t_E^{-1}(v)\setminus s_E^{-1}(v)$ as in Definition~\ref{ureduction}. 
To obtain the shift move inverting this unital reduction, we take
$W:=s_{E}(X)$, $I_w:=s_{E}^{-1}(w)\cap X$ and $\theta^x_w: s_{E_{R+}}^{-1}(v) \hookrightarrow s_{E_{R+}}^{-1}(w)$ 
is given by $\theta^x_w(e):=xe$ as in Definition~\ref{def:shift.move}. We can now check that the shift
at $(v,W)$ of the unital reduction at $v$ with respect to $X$, yields
the following graph:
\begin{gather*}
(E_{R+})_{(v,W)}^0 = E^0, \\
(E_{R+})_{(v,W)}^1 = \left( \left( \left (E^1 \setminus X \right)   
\sqcup \{ xe \, | \,x \in X,\,e\in s_E^{-1}(v)\} \right) \setminus 
\{ xe \, | \,x \in X ,\,e\in s_E^{-1}(v) \} \right)\cup X, \\
s_{(E_{R+})_{(v,W)}}(p) =\begin{cases} s_{E_{R+}}(p)=s_E(p) & \text{for } p \in E^1  \setminus X \\
					s_E(p) & \text{for } p \in X \\
		\end{cases},\\
t_{(E_{R+})_{(v,W)}}(p) =\begin{cases} t_{E+}(p)=t_E(p) & \text{for }  p \in E^1  \setminus X \\
					v & \text{for } p \in X  \\
		\end{cases}.
\end{gather*}		
This is the same as the original graph $E$, so  shift  undoes  unital reduction.

Vice versa, let us now first shift the graph $E$ from $\emptyset \neq  W \subseteq E^0 \setminus \{ v\} $
to $v \in \reg(E)$, and then apply unital reduction.
Our shift data involves  injections $\theta^i_w: s_E^{-1}(v) \hookrightarrow s_E^{-1}(w)$ 
such that the images of all $\theta^i_w$ are pairwise disjoint. To shorten the notation, we
will use $\Theta$ to stand for the union of the images of all $\theta^i_w$ as in~\eqref{BigTheta}.
Then, if we perform the unital reduction of $E_{(v,W)}$ at $v \in \reg(E)$ with respect to 
$X:=\{x_{w,v}^i \, | \, w \in W \text{ and } i \in I_w \}$, we get:
\begin{gather*}
\big((E_{(v,W)})_{R+}\big)^0 = E^0,\quad \\
\big((E_{(v,W)})_{R+}\big)^1 = \left(\left(\left(E^1 \setminus \Theta \right) \sqcup X\right)  \setminus X \right)
\sqcup \{ xe \, | \,x \in X ,\, e\in s_E^{-1}(v) \},\\
s_{(E_{(v,W)})_{R+}}(p) =\begin{cases} s_{E_{(v,W)}}(p)=s_E(p) &\text{for } p \in E^1  \setminus \Theta  \\
					w &\text{for }  p=x_{w,v}^ie\in \{ x_{w,v}^ie \, | \,x_{w,v}^i \in X ,\, e\in s_E^{-1}(v) \}  \
                    \
		\end{cases},\\
t_{(E_{(v,W)})_{R+}}(p)=\begin{cases}  t_{E_{(v,W)}}(p)=t_E(p) &\text{for }  p \in E^1  \setminus \Theta  \\
					t_{E_{(v,W)}}(e)=t_E(e) &\text{for }  p=x_{w,v}^ie \in  \{ x_{w,v}^ie \, | \,x_{w,v}^i \in X ,
                    \, e\in s_E^{-1}(v) \}  \\
		\end{cases}.
\end{gather*}
This graph is also the same as the original graph because the natural identification of $\theta^i_w(e)$ 
with $x^i_{w,v}e$ identifies 
$\Theta:= \bigcup_{w \in W}\bigcup_{i \in I_w}\theta^i_w(s_E^{-1}(v))$ and 
$\{ x_{w,v}^ie \, | \,x_{w,v}^i \in X,\, e\in s_E^{-1}(v)\}$. We thus have shown that also unital reduction
undoes shift.   
\end{proof}

To end this section, 
let us observe that, under some assumptions, the classical construction of the unital reduction 
\cite{eilers2019refined} is a move inverting the classical construction of the shift move~\cite{ABRAMS20081983}. 
More precisely, 
if  a  vertex $v$ is not only regular but also does not emit any edge that is a loop and has only one incoming 
arrow~$g$, 
then the unital reduction $E_{R+}$ at $v$ (with $X:=\{ g\}$) yields a graph with $v\in\reg(E_{R+})$ that is a 
source. If $E_{R+}$ is then shifted from $W:=\{s_E(g)\}$ to $v$ using the unique $\theta_w^g(e):=ge$ for 
all $e \in s^{-1}_E(v)$, we recover the original graph $E$. 
 Much in the same way, if  a  vertex $v$ is not only regular but also is a source such that there is a 
 $w\neq v$ with an injection $\theta_w\colon s_E^{-1}(v)\hookrightarrow s_E^{-1}(w)$ preserving the targets,
 then the vertex $v$ in the shifted graph $E^0_{(v,w)}$ not only is regular and does not emit an edge that is 
 loop,  but also it satisfies 
 $|t_{E_{(v,W)}}^{-1}(v)|=1$, and the unitaly reduced graph 
$(E_{(v,W)})_{R+}$ is~$E$.

\subsection{Reduction and subdivision}
Unlike in the previous two sections, herein we will consider  moves (reduction and subdivision)
inducing $*$-homomorphisms instead of 
$*$-isomorphisms. The standard reduction move was considered  in the context of Leavitt path algebras to yield
Morita equivalence~\cite{KOC20201297}. 
Its limited version, called source removal, was proven to provide Morita equivalence
between graph C*-algebras~\cite{eilers2019refined}. On the other hand, our extended version presented below
still defines an admissible path homomorphism (see Theorem~\ref{reduction-aph}) 
but it fails in general to induce  Morita equivalence
(see~Example~\ref{nomorita}).
\begin{definition} [cf.\ \cite{KOC20201297}]  \label{reduction}
Let $E$ be a graph with $v \in \reg(E)$. Denote by $E_{R}$ the graph:
\begin{gather*} 
 E_{R}^0= (E^0\setminus \{v\})\cup s_E\big(s_E^{-1}(v)\cap t_E^{-1}(v)\big),\\
 E_{R}^1= \left(E^1 \setminus \left( s_E^{-1}(v) \cup t_E^{-1}(v) \right)\right)   
 \sqcup \{ eg \, | \, s_E(e) \neq v, \,  t_E(e)=v \text{ and } s_E(g)=v \},\\
 s_{E_R}(p)=\begin{cases} s_E(p) & \text{for } p \in E^1 \setminus \left(  s_E^{-1}(v) \cup t_E^{-1}(v) \right)  \\
 					s_E(e) & \text{for } p = eg \in \{ eg \, |\,\, s_E(e) \neq v, \, t_E(e)=v \text{ and } s_E(g)=v \} \ ,\\
		\end{cases}\\
 t_{E_R}(p)=\begin{cases} t_E(p) & \text{for } p \in E^1 \setminus \left(  s_E^{-1}(v) \cup t_E^{-1}(v) \right)  
 \\
 					t_E(g) & \text{for } p = eg\in \{ eg \, | \,\, s_E(e) \neq v, \,  t_E(e)=v \text{ and } 
                    s_E(g)=v \} \ .\\
 		\end{cases}
 \end{gather*}
We say that $E_{R}$ is formed by performing a \emph{reduction} at $v$ on $E$. 
\end{definition}

Note that in the case when $v$ is a source, this move is the same as eliminating the 
source~\cite{eilers2019refined}. Moreover, if $v$
does not emit an edge that is a loop, then the above definition of the reduction at $v$ specializes to
the reduction defined in~\cite{KOC20201297}. 
Finally, if we reduce at a vertex that emits exactly one edge, 
receives exactly one edge, and these two edges 
are distinct, then the inverse move is subdivision defined below.

\begin{definition}\footnote{We owe this definition to Murad \"{O}zayd\i n.}
Let $E$ be a graph and $g \in E^1$ be such that $s_E(g)\in\reg(E)$. 
We define the subdivided graph $E_g$ as follows:
\begin{gather*}
E_{g}^0= E^0\sqcup \{ t_E(g)^\prime\}, \qquad E_{g}^1= E^1\sqcup \{ g^\prime\}, \\
s_{E_{g}}(e)=\begin{cases} s_E(e) & e \in E^1   \\
					 t_E(g)^\prime & e= g^\prime\ \\
		\end{cases},\qquad
t_{E_{g}}(e)=\begin{cases} t_E(e) & e \in E^1 \setminus \{ g \}   \\
					 t_E(g)^\prime & e= g \ \\ 
					 t_E(g) & e= g^\prime\ \\ 
		\end{cases}.
\end{gather*}
We say that $E_{g}$ is formed by performing a \emph{subdivision} at $g$ on $E$.
\end{definition}

We are now ready for the main claim of this section:
\begin{theorem} \label{reduction-aph}
Let $E$ be a graph with $v \in \reg(E)$, and 
$E_{R}$ be its reduction at~$v$. Then the following formulas define an admissible path homomorphism 
$f\colon E_{R}\to E $:
\begin{gather}  \nonumber
f(v):=v \text{ for } v \in E_{R}^0,\\
f(p):=\begin{cases} e  & \text{for }p=e \in E^1 \setminus \left( s^{-1}_E(v) \cup t_E^{-1}(v) \right),  \\
			eg &\text{for } p = eg \in \{ eg\, | \,s_E(e) \neq v, \,  t_E(e)=v \text{ and } s_E(g)=v \}.\\
\end{cases}
 \label{freduction} 
\end{gather}
\end{theorem}
\begin{proof}
To begin with, note that $f$ is obviously injective on vertices. Next,
suppose that we have edges $p,p^\prime\in E^1_{R}$ 
such that $f(p) \preceq f(p^\prime)$. Clearly, if both 
\[
p,p' \in E^1 \setminus \left( s^{-1}_E(v) \cup t_E^{-1}(v) \right),
\]
then we see that our inequality 
implies that $p=p^\prime$. Suppose now that 
\[
p =  eg \in \{ eg\, | \,s_E(e) \neq v, \,  t_E(e)=v \text{ and } s_E(g)=v \}.
\]
Then $f(p) =eg\preceq f(p')$ implies that $p'=e^\prime g^\prime$ and $f(p')=e^\prime g^\prime$, 
so $e=e'$ and $g=g'$. Hence, $p=p'$ as needed. Assume next that 
$p \in E^1 \setminus \left( s^{-1}_E(v) \cup t_E^{-1}(v) \right)$ and
\[
p' =  e'g' \in \{ eg\, | \,s_E(e) \neq v, \,  t_E(e)=v \text{ and } s_E(g)=v \}.
\]
Then $p=f(p)\preceq f(p')=e'g'$ implies that $p=e'$, which contradicts $t_E(p)\neq v$.

Finally, to prove the regularity of $f$, consider $w\in\reg(E_R)$. Then $w\neq v$ because either $v\not\in E_R^0$
or $v$ is a sink. The regularity conditions (a) and (b)(i)
are obviously true. Next, we need to consider two types of paths $p\in f(s_{E_R}^{-1}(w))$. 
They are either of length one
or of length two. In both cases it follows from the definition of $f$ that no non-trivial extension of $p$
belongs to $f(s_{E_R}^{-1}(w))$, so the regularity condition (b)(ii) holds.

As for the remaining regularity condition (b)(iii), let us first consider 
\[
p\in f(s_{E_R}^{-1}(w))\cap \big( E^1 \setminus \left( s^{-1}_E(v) \cup t_E^{-1}(v) \right)\big).
\]
Then for any $e'\in s_E^{-1}(w)$ there exists a path $r$ such that $e'r$ belongs to $f(s_{E_R}^{-1}(w))$. Indeed, 
either $t_E(e')=v$, in which case we can pick any $r\in s_E^{-1}(v)$ to see that 
$e'r=f(e'r)\in f(s_{E_R}^{-1}(w))$, or $t_E(e')\neq v$, in which case $e'=f(e')\in f(s_{E_R}^{-1}(w))$. 
Now we consider $p=eg\in f(s_{E_R}^{-1}(w))$.
If $e'\in s_E^{-1}(w)$, then we repeat the same argument as for $p=e$. 
If $g'\in s_E^{-1}(v)$, then $eg'=f(eg')\in f(s_{E_R}^{-1}(w))$. We have thus shown that $f$ is regular.
 \end{proof}
 
 Since the subdivision move is the inverse of the reduction move at a special vertex, 
they induce the same path homomorphism, so
Theorem~\ref{reduction-aph} implies:
\begin{corollary}
Let $E$ be a graph and $E_{g}$ be its subdivision at~$g$. 
Then the following formulas define an admissible path homomorphism of graphs 
$f\colon E \to E_g$:
\begin{gather}
f(v)=v \text{ for } v \in E^0, \qquad f(e)=\begin{cases} e & \text{ for } \in E^1 \setminus \{g \},\\
			gg^\prime & \text{ for } e=g.\\
\end{cases}
\label{subdiv}
\end{gather}
\end{corollary}

To end with, recall that the upshot of studying special reductions in \cite{KOC20201297} 
is that they induce Morita equivalence.
More precisely, \cite[Theorem~4.1]{KOC20201297} states that, if $E_R$ is the reduction of $E$ at 
$v\in\reg(E)$ such that $s_E^{-1}(v)\cap t_E^{-1}(v)=\emptyset$, then the Leavitt path algebras $L_k(E)$
and $L_k(E_R)$ are Morita equivalent. The following example shows that this no longer needs to be the case when 
$s_E^{-1}(v)\cap t_E^{-1}(v)\neq\emptyset$.
 \begin{example}\label{nomorita}
 Consider the graph~$E$:
\[
\begin{tikzcd}
	\bullet & \bullet
	\arrow[from=1-1, to=1-1, loop, in=145, out=215, distance=10mm]
	\arrow[from=1-1, to=1-2, in = 150, out=30]
    \arrow[from=1-2, to=1-1, in = -30, out=-150]
\end{tikzcd}.
\]
Should we do a reduction at the vertex not supporting a loop, we get the graph~$E_2$:
 \[
\begin{tikzcd}
	\bullet
	\arrow[from=1-1, to=1-1, loop, in=145, out=215, distance=10mm]
    \arrow[from=1-1, to=1-1, loop, in=-35, out=35, distance=10mm]
\end{tikzcd}.
\]
Should we instead do a reduction at the vertex supporting the loop, we obtain the Toeplitz graph~$T$:
 \[
\begin{tikzcd}
	\bullet & \bullet
    \arrow[from=1-2, to=1-1, in = -30, out=-150]
    \arrow[from=1-2, to=1-2, loop, in=-35, out=35, distance=10mm]
\end{tikzcd}.
\]
Suppose now that the Leavitt path algebras $L_{\mathbb C}(T)$ and $L_{\mathbb C}(E)$ are Morita equivalent.
Then so are the Leavitt path algebras $L_{\mathbb C}(T)$ and $L_{\mathbb C}(E_2)$ because  $L_{\mathbb C}(E_2)$
and $L_{\mathbb C}(E)$ are Morita equivalent by \cite[Theorem~4.1]{KOC20201297}. Hence, 
$K_0(L_{\mathbb C}(T))=K_0(L_{\mathbb C}(E_2))$. On the other hand, since $E_2$ is a finite graph without
sinks whose adjacency matrix $A_{E_2}$ satisfies $\det(I-A_{E_2}^T)\neq 0$, it follows from \cite[p.\ 5]{abc09}
that $K_0(L_{\mathbb C}(E_2))=K_0(C^*(E_2))$. Combining it with the fact that 
$K_0(C^*(E_2))=K_0(\mathcal{O}_2)=0$, we conclude that $K_0(L_{\mathbb C}(T))=0$. Finally,
as the inclusion $\iota \colon L_{\mathbb C}(T)\to C^*(T)\cong\mathcal{T}$ is a unital algebra homomorphism, we
arrive at the desired contradiction:
\[
K_0(L_{\mathbb C}(T))\ni 0=[1]\stackrel{\iota_*}{\longmapsto} [1]=1\in\mathbb{Z}\cong K_0(\mathcal{T}).
\]
\end{example}

\subsection{Balanced In-split}
\begin{definition}[\cite{eilers2019refined}]
Let $E$ be a graph and $w \in E^0$ be a regular vertex. Partition $t_E^{-1}(w)$ into a 
finite disjoint union of sets (of which some may be empty) 
$\epsilon_1 \cup ... \cup \epsilon_n $. Let $E_I$ denote the graph:
\begin{gather*}
E_I^0:= \{v^1 \, | \,  v \in E^0 \text{ and } v \neq w \} \sqcup \{w^1, \ldots, w^n \},\\
E_I^1:= \{e^1 \, | \,  e \in E^1 \text{ and } s_E(e) \neq w \} \sqcup \{e^i \, | \, s_E(e) = w \text{ and } 1 \leq i \leq n\},\\
s_{E_I}(e^i):=s_E(e)^i, \qquad
t_{E_I}(e^i):=\begin{cases} t_E(e)^1 & t_E(e) \neq w, \\
                    w^j & t_E(e) =  w \text{ and } e \in \epsilon_j. \\
        \end{cases}        
\end{gather*}
We say that $E_I$ is formed by performing an \emph{in-split} at $w$ on $E$.
\end{definition}
\begin{definition}[\cite{eilers2019refined}] Let $E_l$ and $E_r$ be the graphs resulting from an in-split at $w$, 
where $t_E^{-1}(w)$ has (for each in-split) been 
partitioned into $n$ sets $\mathcal{L}_1,\dots,\mathcal{L}_n$ and $\mathcal{R}_1,\dots,\mathcal{R}_n$,
respectively. 
We then say that  $E_l$ and $E_r$ 
are formed by performing a \emph{balanced in-split} of $E$ at $w$.
\end{definition}
\noindent
To simplify the notation, for any $e\in E^1$, the induced edge both in $E_l$ and $E_r$ will be denoted by the same 
symbol, e.g.,~$e^i$,
despite the fact that for $e\in \mathcal{L}_j\cap\mathcal{R}_k$ we will have
\begin{equation}
    t_{E_l}(e^i) = w^j,\quad t_{E_r}(e^i) = w^k.
\end{equation}

\begin{lemma}[cf.\ \cite{eilers2019refined}]\label{lemma:in-splitplus}
Let $E$ be a graph, and $E_l$ and $E_r$  formed by performing a balanced in-split of $E$  at $w$. 
Then the following formulas define mutually inverse $*$-isomorphisms:
\begin{gather}
g\colon C^*(E_l)\longrightarrow C^*(E_r),\quad
    g( S_{v^i})  =S_{v^i}  \, ,\quad v^i \in E^0_l, \nonumber\\
         g(S_{e^i})  =
        \begin{cases} S_{e^i}  &\text{for } e^i \in E^1_l \text{ and } t_E(e) \neq w,\\
 S_{e^i}\sum_{f\in s_E^{-1}(w)}S_{f^k}S_{f^j}^* & \text{for } e^i \in E^1_l \, , \, e \in \mathcal{L}_j \cap 
 \mathcal{R}_k,\\ 
 \end{cases}\label{phi}\\
  h\colon C^*(E_r)\longrightarrow C^*(E_l),\quad
   h(S_{v^i})  =S_{v^i}  \, ,\quad v^i \in E^0_r, \nonumber\\
    h (S_{e^i})  =\begin{cases} S_{e^i}  &\text{for } e^i \in E^1_r \text{ and } t_E(e) \neq w,\\
 S_{e^i}\sum_{f\in s_E^{-1}(w)}S_{f^k}S_{f^j}^*&\text{for } e^i \in E^1_r \, , \,e \in \mathcal{R}_j \cap 
 \mathcal{L}_k.\\ 
 \end{cases}\label{varphi}
\end{gather}
\end{lemma}

The next theorem shows how to functorially obtain the above formulas  when the vertex $w$ at which we in-split 
emits exactly one edge.
\begin{theorem}\label{thm:in-splitplus}
Let $E$ be a graph, $E_l$ and $E_r$ be its balanced in-split at $w$ such that $s_E^{-1}(w)=\{f\}$. 
Denote the partitions defining $E_l$ as $\mathcal{L}_1,\dots,\mathcal{L}_n$ and the partitions defining $E_r$ as 
$\mathcal{R}_1,\dots,\mathcal{R}_n$. 
Then the maps $\phi^0\colon E_l^0\to S(E_r)$ and $\phi^1\colon E_l^1\to S(E_r)$ defined by
\begin{gather}\label{phizero}
    \phi^0(v^i) = \begin{cases}
        (v^i,v^i )&\text{for } v\neq w,\\
        (f^i,f^i) &\text{for } v=w^i,
    \end{cases}\\
    \phi^1(e^i) = \begin{cases}
        (e^i,t_{E_r}(e_i)) &\text{for } t_{E}(e)\neq w,\\
        (e^if^k,f^j) &\text{for } t_E(e)=w,\; e\in \mathcal{L}_j\cap \mathcal{R}_k,
    \end{cases}\label{phione}
\end{gather}
yield a regular extended path homomorphism $\phi\colon S(E_l)\to S(E_r)$ covariantly inducing the 
$*$-isomorphism~\eqref{phi}: $\phi_*=g$.
\end{theorem}
\begin{proof}
    Let us start by proving that $\phi^0$ and $\phi^1$ determine an extended path homomorphism 
    $\phi\colon S(E_l)\to S(E_r)$.
    
    \underline{Axiom (1):} Let $v^i,u^j\in E_l^0$. We want to show that
    \begin{equation}
        \phi^0(v^i)\phi^0(u^j) = \delta_{v^i,u^j}\phi^0(v^i).
    \end{equation}
    We need to distinguish three cases. First, if both $v^i$and $u^j$ are not any of the $w^l$'s, 
    then everything trivializes:
    \begin{equation}
        \phi^0(v^i)\phi^0(u^j) = (v^i,v^i)(u^j,u^j) = \delta_{v^i,u^j}(v^i,v^i) = \phi^0(v^i).
    \end{equation}
    Suppose next that $v^i=w^i$ and $u^j$ is not any of the $w^l$'s. Our formula then becomes
    \begin{equation}
        \phi^0(w^i)\phi^0(u^j) = (f^i,f^i)(u^j,u^j) = 0
    \end{equation}
    because $s_{E_r}(f^i)=w^i\neq u^j$. Lastly, consider $v^i=w^i$ and $u^j=w^j$. Then we have
    \begin{equation}
        \phi^0(w^i)\phi^0(w^j) = (f^i,f^i)(f^j,f^j).
    \end{equation}
    Of course, $f^i$ is comparable to $f^j$ if and only if they are equal, so we get exactly what we want:
    \begin{equation}
        \phi^0(w^i)\phi^0(w^j) = \delta_{w^i,w^j}(f^i,f^i) = \delta_{w^i,w^j}\phi^0(w^i).
    \end{equation}
    
    \underline{Axiom (2):} Given $e^i\in E_l^1$, we need to prove that
    \begin{equation}
        \phi^0(s_{E_l}(e^i))\phi^1(e^i) = \phi^1(e^i) = \phi^1(e^i)\phi^0(t_{E_l}(e^i)).
    \end{equation}
    First, let us assume that $t_E(e)\neq w$. If also $s_E(e)\neq w$, then one can immediately check that
    \begin{gather}
        \phi^0(s_{E_l}(e^i))\phi^1(e^i) = (s_{E_l}(e^i),s_{E_l}(e^i))(e^i,t_{E_r}(e^i)) = (e^i,t_{E_r}(e^i)) = 
        \phi^1(e^i),\\
        \phi^1(e^i)\phi^0(t_{E_l}(e^i)) = (e^i,t_{E_r}(e^i))(t_{E_l}(e^i),t_{E_l}(e^i)) = (e^i,t_{E_r}(e^i)) = 
        \phi^1(e^i).
    \end{gather}
    On the other hand, if $s_E(e)=w$, then $e=f$ and our formulas become
    \begin{gather}
        \phi^0(w^i)\phi^1(f^i) = (f^i,f^i)(f^i,t_{E_r}(f^i)) = (f^i, t_{E_r}(f^i)) = \phi^1(f^i),\\
        \phi^1(f^i)\phi^0(t_{E_l}(f^i)) = (f^i,t_{E_r}(f^i))(t_{E_l}(f^i),t_{E_l}(f^i)) = (f^i,t_{E_r}(f^i)) = 
        \phi^1(f^i).
    \end{gather}
    Now suppose that $t_E(e)=w$ and further assume that $e\in\mathcal{L}_j\cap\mathcal{R}_k$. Similar to our 
    reasoning above, we need to differentiate between the source being $w$ or not. If $s_E(e)\neq w$, then our 
    formulas become
    \begin{gather}
        \phi^0(s_{E_l}(e^i))\phi^1(e^i) = (s_{E_l}(e^i),s_{E_l}(e^i))(e^if^k,f^j) = (e^if^k,f^j) = \phi^1(e^i),\\
        \phi^1(e^i)\phi^0(w^j) = (e^if^k,f^j)(f^j,f^j) =  (e^if^k,f^j) = \phi^1(e^i).
    \end{gather}
    Here, observe that $e\in \mathcal{L}_j$ implies $t_{E_l}(e^i)=w^j$, and $f^j$ is the unique edge starting at 
    $w^j$. In the case $s_E(e)=w$, we obtain $e=f$ and note that $f$ is a loop supported at~$w$. 
    Then our formulas become
    \begin{gather}
        \phi^0(w^i)\phi^1(f^i) = (f^i,f^i)(f^if^k,f^j) = (f^if^k,f^j) = \phi^1(f^i),\\
        \phi^1(f^i)\phi^0(w^j) = (f^if^k,f^j)(f^j,f^j) = (f^if^k,f^j) = \phi^1(f^i).
    \end{gather}
    
    \underline{Axiom (3):} Given $x^i,y^{i'}\in E_l^1$, we need to check that the  formula 
    \begin{equation}
        (\phi^1(x^i))^*\phi^1(y^{i'}) = \delta_{x^i,y^{i'}}\phi^0(t_{E_l}(x^i))
    \end{equation}
    holds.
    Here we need to distinguish three cases. First, assume that  $t_E(x)\neq w\neq t_E(y)$. Then everything is 
    trivial because
    \begin{equation*}
        (\phi^1(x^i))^*\phi^1(y^{i'}) = (t_{E_r}(x^i),x^i)(y^{i'},t_{E_r}(y^{i'})) = \delta_{x^i,y^{i'}}(t_{E_r}
        (x^i),t_{E_r}(x^i)) = \delta_{x^i,y^{i'}}\phi^0(t_{E_r}(x^i)).
    \end{equation*}
    Next, suppose that $t_E(x)=w\neq t_E(y)$. Then our formula becomes
    \begin{equation}\label{prodzero}
        (\phi^1(x^i))^*\phi^1(y^{i'}) =(f^j,x^if^k)(y^{i'},t_{E_r}(y^{i'})),
    \end{equation}
    where $x\in \mathcal{L}_j\cap \mathcal{R}_k$. Paths $x^if^k$ and $y^{i'}$ can be comparable only when
    $x^i=y^{i'}$. However, this is impossible because $t_{E_l}(x^i)\in\{w^1,\dots,w^n\}$ whereas $t_{E}(y)\neq w$ 
    implies that $y^{i'}$ does not have any of the $w^l$'s as its target, so the  product \eqref{prodzero} 
    vanishes.  Lastly, consider the case $t_E(x)=w=t_E(y)$, and  further assume that
    \begin{equation}
        x\in \mathcal{L}_j\cap \mathcal{R}_k,\qquad y\in \mathcal{L}_{j'}\cap \mathcal{R}_{k'}.
    \end{equation}
    Then we get
    \begin{equation}
        (\phi^1(x^i))^*\phi^1(y^{i'}) = (f^j,x^if^k)(y^{i'}f^{k'},f^{j'}).
    \end{equation}
    Observe now that paths $x^if^k$ and $y^{i'}f^{k'}$ are comparable if and only if they are equal, so $x=y$. 
    Hence, $j=j'$ and $k=k'$ as there is exactly one $\mathcal{L}_j$ and exactly one 
    $\mathcal{R}_k$ containing~$x$. Therefore, the above formula becomes
    \begin{equation}
        (\phi^1(x^i))^*\phi^1(y^{i'}) = \delta_{x^{i},y^{i'}}(f^j,f^j)=\delta_{x^{i},y^{i'}}\phi^0(w^j).
    \end{equation}
    We thus arrive where needed because  $x\in\mathcal{L}_j$ implies $t_{E_l}(x^i)=w^j$.
    
Summarizing, the maps $\phi^0$ and $\phi^1$ defined respectively in \eqref{phizero} and $\eqref{phione}$
lift to an extended path homomorphism \mbox{$\phi\colon S(E_l)\to S(E_r)$}. 
    The fact that $\phi$ is regular and induces $\phi_*=g$ as in \eqref{phi} follows from 
    Corollary~\ref{cor:from.hom.to.reg}.
\end{proof}
\begin{corollary}\label{cor:in-splitplus}
    Let $E$ be a graph and $w\in\mathrm{reg}(E)$ be such that $w$ does not emit an edge that is a loop. 
    Then, for any balanced in-splits $E_l$ and  $E_r$ at~$w$, the $*$-homomorphism  
    $g\colon C^*(E_l)\to C^*(E_r)$ given in \eqref{phi} can be decomposed as
    \begin{equation}
    C^*(E_l)\stackrel{\Psi_L^*}{\longrightarrow} C^*(E_{\widehat{l}}) \stackrel{ \xi_*}{\longrightarrow}C^*(E_r).
    \end{equation}
    Here $E_{\widehat{l}}$ is the maximal out-split of $E_l$ at the set of vertices replacing $w$ 
    due to in-split,  
    $\Psi_L\colon E_{\widehat{l}}\to E_{l}$ is the out-split graph homomorphism and
    $\xi\colon E_{\widehat{l}}\to E_r$ is a regular extended path homomorphism.
\end{corollary}
\begin{proof}
    Consider the maximal out-split $E_{\widehat{O}}$ at $w$ and denote by $w^1,\dots,w^m$ 
    the vertices in $E_{\widehat{O}}$ that replace $w$. One 
    can clone any partition $\mathcal{E}_1,\dots,\mathcal{E}_n$ of edges incoming to $w$ to partitions 
    $\mathcal{E}_1^j,\dots,\mathcal{E}_n^j$ of edges 
    incoming to~$w^j$:
    \begin{equation}\label{cor:in-splitpluspartitionlift}
        \mathcal{E}_i^j := \{ e^j\in E_{\widehat{O}}^1 \;|\; e\in\mathcal{E}_i \}.
    \end{equation}
    As $w$ does not emit any edge that is a loop, neither do any $w^i$, so the order of 
    in-splits at $w^i$'s do not matter. Next, observe 
    that $E_{\widehat{l}}$ is  the same as the graph obtained as
    the result of in-splitting of 
    $E_{\widehat{O}}$ at all $w^j$'s using the cloned partitions $\mathcal{L}_i^j$.
    The same can be concluded about~$E_{\widehat{r}}$ by considering the cloned partitions~$\mathcal{R}_i^j$.
    
    On the other hand, the two graphs $E_{\widehat{l}}$ and $E_{\widehat{r}}$ can be connected 
    via a sequence of balanced in-splits, each of them satisfying the assumptions of 
    Theorem~\ref{thm:in-splitplus}, and thus connecting the two graphs via a sequence of extended path 
    isomorphisms. To construct such a sequence of balanced in-splits, to any sequence
    $(\varepsilon^1,\dots,\varepsilon^m)$ with $\varepsilon^k\in \{l,r\}$, we assign the graph
    $E_{(\varepsilon^1,\dots,\varepsilon^m)}$  that is the result of in-splitting $E_{\widehat{O}}$ 
    at the vertices $w^1,\dots,w ^m$ using the partitions $\mathcal{L}^k$ if $\varepsilon^k=l$, or 
    $\mathcal{R}^k$ if $\varepsilon^k=r$. 
    For instance,
    \begin{equation}
        E_{(l,\dots,l)} = E_{\widehat{l}}, \qquad E_{(r,\dots,r)} = E_{\widehat{r}}.
    \end{equation}
    
    Since  the order of in-splits of $E_{\widehat{O}}$ does not matter, 
    Theorem~\ref{thm:in-splitplus} implies that, if two sequences 
    $\varepsilon=(\varepsilon^1,\dots,\varepsilon^m)$ and $\epsilon=(\epsilon^1,\dots,\epsilon^m)$  
    differ only for one $k$, i.e., $\varepsilon^i=\epsilon^i$ for $i\neq k$, then there are extended path 
    homomorphisms connecting these two graphs:
    \[
    \begin{tikzcd}[ampersand replacement=\&]
	{E_{\varepsilon}} \& {E_{\epsilon}}
	\arrow["\phi", shift left, from=1-1, to=1-2]
	\arrow["\varphi", shift left, from=1-2, to=1-1]
\end{tikzcd}.
\]
    To be more precise, we can in-split first at all vertices $w^i$ with $i\neq k$, and then a balanced 
    in-split at $w^k$ will yield exactly the morphisms
    $\varphi$ and $\phi$ via Theorem~\ref{thm:in-splitplus}. 
    Hence, we can get from $E_{\widehat{l}}$ to $E_{\widehat{r}}$ by swapping out one letter at a time and 
    applying Theorem~\ref{thm:in-splitplus}:
    \[
    \begin{tikzcd}[ampersand replacement=\&]
	{E_{(l,\dots,l)}} \& {E_{(r,l,\dots,l)}} \& {E_{(r,r,l,\dots,l)}} \& \dots \& {E_{(r,\dots,r,l)}} 
    \& {E_{(r,\dots,r)}}
	\arrow["{\phi_1}", shift left, from=1-1, to=1-2]
	\arrow["{\varphi_1}", shift left, from=1-2, to=1-1]
	\arrow["{\phi_2}", shift left, from=1-2, to=1-3]
	\arrow["{\varphi_2}", shift left, from=1-3, to=1-2]
	\arrow[shift left, from=1-3, to=1-4]
	\arrow[shift left, from=1-4, to=1-3]
	\arrow[shift left, from=1-4, to=1-5]
	\arrow[shift left, from=1-5, to=1-4]
	\arrow["{\phi_m}", shift left, from=1-5, to=1-6]
	\arrow["{\varphi_m}", shift left, from=1-6, to=1-5]
\end{tikzcd}.\]
    Denoting the composition of all left-to-right morphisms  by $\widehat{\phi}$ and
    all right-to-left morphisms by $\widehat{\varphi}$, and applying the covariant functor, we obtain the following diagram:
    \[\begin{tikzcd}[ampersand replacement=\&]
        {C^*(E_{\widehat{l}})} \&\& {C^*(E_{\widehat{r}})} \\
        \\
        {C^*(E_l)} \&\& {C^*(E_r)}
        \arrow["{\widehat{\phi}_*}", shift left, from=1-1, to=1-3]
        \arrow["{(\psi_L)_*}", from=1-1, to=3-1]
        \arrow["{\widehat{\varphi}_*}", shift left, from=1-3, to=1-1]
        \arrow["{(\psi_R)_*}", from=1-3, to=3-3]
    \end{tikzcd}.\]
    Here we apply Theorem~\ref{thm:maxoutsplit} to write $(\psi_L)_* = (\Psi_L^*)^{-1}$ 
    and $(\psi_R)_*=(\Psi_R^*)^{-1}$, where
    $\Psi_L\colon E_{\widehat{l}}\to E_l$ and $\Psi_R\colon E_{\widehat{r}}\to E_r$ 
    are out-split graph ho\-mo\-mor\-phisms. The $*$-ho\-mo\-mor\-phisms
    $\phi_*$ and $\varphi_*$, as defined in 
    Lemma~\ref{lemma:in-splitplus}, complete the above diagram into the following commutative square:
    \[
    \begin{tikzcd}[ampersand replacement=\&]\label{BiSDiag}
        {C^*(E_{\widehat{l}})} \&\& {C^*(E_{\widehat{r}})} \\
        \\
        {C^*(E_l)} \&\& {C^*(E_r)}
        \arrow["{\widehat{\phi}_*}", shift left, from=1-1, to=1-3]
        \arrow["{(\psi_L)_*}", from=1-1, to=3-1]
        \arrow["{\widehat{\varphi}_*}", shift left, from=1-3, to=1-1]
        \arrow["{(\psi_R)_*}", from=1-3, to=3-3]
        \arrow["{g}", shift left, from=3-1, to=3-3]
        \arrow["{h}", shift left, from=3-3, to=3-1]
    \end{tikzcd}.
    \]
    
    Indeed, the above diagram is clearly commutative with respect to generators $S_v$ for 
    $v\in E_{\widehat{l}}^0$, so we only need to see what happens with generators $S_e$ for 
    $e\in E_{\widehat{l}}^1$. Edges in $E_{\widehat{l}}^1$ are in one of the following forms:
    \begin{equation}
        \begin{cases}
            (e,s_E(e),t_E(e)) & \text{if } s_E(e)\neq w\neq t_E(e),\\
            (e,w_i,t_E(e)) & \text{if } s_E(e)=w, t_E(e)\neq w,\\
            (e,s_E(e),w_i^j) & \text{if } s_E(e)\neq w, e\in\mathcal{L}_i.
        \end{cases}
    \end{equation}
    Notice that there is no case with $s_E(e)=w=t_E(e)$ as, by assumption, 
    $w$ does not emit any edge that is 
    a loop. Note also that here we view $E_{\widehat{l}}$ as the maximal-out-split of the graph $E_l$ at the 
    newly created vertices $w_1,\dots,w_m$. Furthermore, we identify the set of edges $E_l^1$ of the in-split 
    graph with the pullback $E^1\xrightarrow{s_E}E^0\leftarrow E_l^0$, 
    where $E_l^0\to E^0$ is given by
    \begin{equation*}
        w_1,\dots,w_m \longmapsto w, \qquad w\neq v\longmapsto v.
    \end{equation*}
    Thus, $E_{\widehat{l}}^1$ is a pullback of a pullback, so its elements can be written as
    compatible triples. 
    
    Now, applying the $*$-homomorphisms given by both the out-splits 
    and the balanced in-splits
    to the first and second type edges from $E_{\widehat{l}}^1$ keeps them intact. 
    The only non-trivial thing happens in the third case, so this is the only case we need to check. 
    Applying first the 
    out-split $*$-homomorphism according to Lemma~\ref{lem:outsplit}, and then $g$ according to 
    Lemma~\ref{lemma:in-splitplus}, to such a generator, we obtain:
    \begin{align*}
        (g\circ(\psi_L)_*)(S_{(e,s_E(e),w_i^j)}) & = g\left( S_{(e,s_E(e))}S_{(f_j,w_i)}S_{(f_j,w_i)}^* 
        \right)\\
        & = S_{(e,s_E(e))} \left(\sum_{s=1}^mS_{(f_s,w_k)}S_{(f_s,w_i)}^*\right)S_{(f_j,w_i)}S_{(f_j,w_i)}^* 
        \\
        & = S_{(e,s_E(e))}S_{(f_j,w_k)}S_{(f_j,w_i)}^*.
    \end{align*}
    Here $i$ and $k$ are such that $e\in\mathcal{L}_i\cap \mathcal{R}_k$, and
    $\{f_1,\dots,f_m\}:=s_E^{-1}(w)$. 
    
    To compute the composition $(\psi_R)_*\circ\widehat{\phi}$, first we need to determine 
    $\widehat{\phi}_*(S_{(e,s_E(e),w_i^j)})$. To this end, recall that $\widehat{\phi}_*$ is a 
    composition of the balanced in-split $*$-homomorphisms $(\phi_k)_*$,
    and each $\phi_k$ changes $S_{(e,s_E(e),w_i^j)}$ if and
    only if the target of the pre-in-split edge $(e,w^j)$ is~$w^k$. Therefore, 
    $S_{(e,s_E(e),w_i^j)}$  changes only when we apply~$(\phi_j)_*$. 
    Consequently, we obtain
    \begin{align*}
        ((\psi_R)_*\circ \widehat{\phi}_*) (S_{(e,s_E(e),w_i^j)}) &= (\psi_R)_*(S_{(e,s_E(e),w_k^j)}
        S_{(f_j,w_k,t_E(f_j))}S_{(f_j,w_i,t_E(f_j))}^*)\\
        &=S_{(e,s_E(e))}S_{(f_j,w_k)}S_{(f_j,w_k)}^*S_{(f_j,w_k)}S_{(f_j,w_i)}^* \\
        &= S_{(e,s_E(e))}S_{(f_j,w_k)}
        S_{(f_j,w_i)}^*.
    \end{align*}

    Finally, as the proof of $h\circ(\psi_R)_*=(\psi_L)_*\circ\widehat\varphi_*$ is symmetric, we conclude
    the claimed commutativity of the diagram~\eqref{BiSDiag}.
     We thus obtain the desired decompositions
    \begin{equation}
        g = (\psi_R\circ\widehat{\phi})_*\circ \Psi_L^* ,\quad h = 
        (\psi_L\circ\widehat{\varphi})_*\circ\Psi_R^*,
    \end{equation}
    as needed.
\end{proof}

\color{black}
\section{Examples and applications}

\subsection{Locally derived graphs}\label{derived}
    We open this section with derived graphs, which are a special case of locally derived graphs. The former were defined in \cite[Section 2.1.1]{GTTopologicalGraphTheory} and the latter are a generalization defined in \cite[Definition~5.7]{ht-25}. We recall the precise formulation only for derived graphs for brevity of this section.
    
	Let $E$ be a graph, $\Gamma$ be a group and $L\colon E^1\to \Gamma$ be a labelling map. The \textit{derived graph} constructed from this data is the graph $E_L=(E_L^0,E_L^1,s_L,t_L)$, where
	\begin{gather*}
		E_L^0 := E^0\times \Gamma, \quad E_L^1 := E^1\times\Gamma,\\
		s_L(e,g) := (s_E(e),g), \quad t_L(e,g) := (t_E(e),L(e)g), \quad e\in E^1, \, g\in\Gamma.
	\end{gather*}
	This graph comes with a natural projection $p\colon E_L\to E$.
    \begin{example}[Cuntz algebra $\mathcal{O}_n$]
        Let $E_n$ be the graph with one vertex $v$ and $n$ loops denoted by $e_0,\dots,e_{n-1}$. If we pick $\Gamma=\mathbb{Z}/n\mathbb{Z}$ and the labelling map to be defined by $L(e_i):=i$, then the resulting 
        derived graph $E_{n,L}$ is exactly the same as the total out-split of~$E_n$. 
        We sketch the $n=2$ case below:
        \begin{equation*}\label{cuntzfold}
\begin{gathered}
\begin{tikzpicture}[scale=0.4,auto,swap]
\centering
\tikzstyle{vertex}=[circle,fill=black,minimum size=3pt,inner sep=0pt]
\tikzstyle{edge}=[draw,->]
\tikzset{every loop/.style={min distance=20mm,in=130,out=50,looseness=40}}
    \node[vertex,label=below left:{$(v,0)$}] (1) at (-2,0) {};
    \node[vertex,label=below right:{$(v,1)$}] (2) at (2,0) {};
    \path (1) edge [edge, anchor=center, loop above] node[above left] {$(e_0,0)$} (1);
    \path (2) edge [draw,<-, anchor=center, loop above] node[above right] {$(e_0,1)$} (2);
    \path (1) edge [edge, bend left] node[above] {$(e_1,0)$} (2);
    \path (2) edge [edge, bend left] node[below] {$(e_1,1)$} (1);
\end{tikzpicture}
\begin{tikzpicture}[scale=0.4,auto,swap]
\centering
\tikzstyle{vertex}=[circle,fill=black,minimum size=3pt,inner sep=0pt]
\tikzstyle{edge}=[draw,->]
    \node (1) at (0,0) {};
    \node (2) at (0,-1) {$\longrightarrow$};
    \node (3) at (0,-2) {};    
    \node (4) at (0,-3) {};
\end{tikzpicture}
\begin{tikzpicture}[scale=0.4,auto,swap]
\centering
\tikzstyle{vertex}=[circle,fill=black,minimum size=3pt,inner sep=0pt]
\tikzstyle{edge}=[draw,->]
    \node[vertex,label=below:$v$] (1) at (0,0) {};
    \node (2) at (0,-1) {};
    \node (3) at (0,-2) {};
    \path (1) edge [edge, anchor=center, loop above, min distance=20mm, in=130, out=50, looseness=40] node[below] {$e_0$} (1);
    \path (1) edge [edge, anchor=center, loop below, in=130, out=50, looseness=60] node[above] {$e_1$} (1);
\end{tikzpicture}
\end{gathered}.
\end{equation*}
Here the covering 
projection $p\colon E_{n,L}\to E_n$ is given by $(v,i)\mapsto v$ and $(e_i,j)\mapsto e_i$. Kumjian and Pask proved in 
\cite[Theorem 1.1]{KUMJIAN_PASK_1999}, that the graph C*-algebra $C^*(E_{n,L})$ of the derived graph is isomorphic with the crossed product 
$C^*(E_n)\rtimes \Gamma$. On the other hand, as $E_{n,L}$ is also the total out-split of $E_n$, we know that
$C^*(E_{n,L})\cong C^*(E_n)$. More precisely, when $\Gamma$ is finite, the covering projection $p$ is an admissible graph homomorphism
\cite[Proposition~5.5]{ht-25}, and $p^*\colon C^*(E_n)\to C^*(E_{n,L})$ is the embedding onto the 
$\Gamma$-fixed-point subalgebra, where the action $\widetilde\alpha$  
of $\Gamma$ on $C^*(E_{n,L})$ is given by the obvious action 
of $\Gamma$ on~$E_{n,L}$.
Combining it with the out-split graph homomorphism $f\colon E_{n,L}\to E_n$, we obtain a proper endomorphism 
of~$\mathcal{O}_n$:
    \begin{equation}\label{specialendo}
        \mathcal{O}_n\stackrel{p^*}{\longrightarrow}\mathcal{O}_n\rtimes(\mathbb{Z}/n\mathbb{Z})
        \stackrel{(f^*)^{-1}}{\longrightarrow}\mathcal{O}_n.
    \end{equation}
    Indeed, as $p^*$ is not surjective, neither is $\rho:=(f^*)^{-1}\circ p^*\colon \mathcal{O}_n\hookrightarrow\mathcal{O}_n$.
 On generators it is given by:
    \begin{equation}
        \rho(S_{e_i})=\sum_{j=0}^{n-1} S_{e_j}S_{e_{i+j}}S_{e_{i+j}}^*.
    \end{equation}
    Here we think of the addition $i+j$ as the addition in $\mathbb{Z}/n\mathbb{Z}$. 
    
    We thus enter the  galaxy of endomorphisms of Cuntz algebras. Recent hitchhiker's guide to this galaxy \cite{acr21} is over 100-page long and contains uncountably many open questions. For starters,
    Cuntz has shown in \cite{c80} that any 
    endomorphism $\varphi\colon\mathcal{O}_n\to\mathcal{O}_n$ is given by multiplying the generating isometries by the
    unitary $u_\varphi: = \sum_{i=0}^{n-1}\varphi(S_{e_i})S_{e_i}^*$. In our case, we get 
    \begin{equation}
        u_\rho = \sum_{i,j=0}^{n-1}S_{e_j}S_{e_{i+j}}S_{e_{i+j}}^*S_{e_i}^*.
    \end{equation}
    Since $u_\rho=u_\rho^*$, we infer that $u_\rho^2=1$.

    Our graph-combinatorial construction much resembles the construction in \cite[Example~3.7]{i92}. Indeed,
    if we take in this example the finite abelian group $G$ to be $\mathbb{Z}/n\mathbb{Z}$, then our action 
    $\alpha$ on $\mathcal{O}_n$ induced by the given action $\widetilde\alpha$ via
    \begin{equation}
\alpha_g:=(f^*)^{-1}\circ\widetilde\alpha_g\circ f^*,\quad g\in\mathbb{Z}/n\mathbb{Z},
    \end{equation}
    coincides with the action given by~\cite[(3.7.1)]{i92}. We also have that the endomorphism 
    $\rho$ is invariant under the action $\alpha$, and its image equals to the fixed-point subalgebra,
    which is isomorphic with~$\mathcal{O}_n$. However, our endomorphism $\rho$ does not coincide with
    the endomorphism given by~\cite[(3.7.3)]{i92}, although the unitary element associated with the latter 
    endomorphism is also self-adjoint. 

    Let us end this example, in the spirit of~\cite{acr21}, by posing some naturally arising open questions:
    \begin{gather}
        u_{\rho^k}\stackrel{?}{=}u_{\rho^k}^*,\\
        u_{\rho^k}=u_{\rho^l}\stackrel{?}{\implies} k=l,\\ 
        \mathbb{Z}/(n-1)\mathbb{Z}\cong K_0(\mathcal{O}_n)\ni\left[\frac{1+u_{\rho^k}}{2}\right]=\;?,
    \end{gather}
    where $k,l\in\mathbb{N}\setminus\{0\}$ and $n\in\mathbb{N}\setminus\{0,1\}$.
    
    \begin{remark}
        The  $*$-isomorphism $\mathcal{O}_n\cong\mathcal{O}_n\rtimes(\mathbb{Z}/n\mathbb{Z})$ much resembles 
        the classical $*$-isomorphism $C(S^1)\cong C(S^1)\rtimes(\mathbb{Z}/n\mathbb{Z})$ coming from the $n$-covering of the
        circle by the circle. Here the embedding $*$-homomorphism $C(S^1)\hookrightarrow C(S^1)$ is given by $u\to u^n$,
        where $u$ is the generating unitary of~$C(S^1)$.
    \end{remark}
    \begin{remark}
        Applying shift moves and unital reductions allows one to construct explicit embeddings
        of Cuntz algebras into matrix algebras tensored with Cuntz algebras~\cite{hl24}:
        \[
        \mathcal O_p\longrightarrow M_k(\mathcal O_q),\quad \text{ where } (p-1)k=k'(q-1), 
        \;k'\in\mathbb{N}\setminus\{0\}.
        \]
        In particular, we have the embedding $\mathcal O_n\rightarrow M_n(\mathcal O_n)$, which
        is equivalent to the composition $(p^*\otimes \id)\circ p^*$ (see~\eqref{specialendo}).
    \end{remark}
    
    The construction of a locally derived graph allows us to immediately generalize the phenomenon of an out-split graph coinciding
    with a derived graph. 
    Indeed,  if we take a graph $E$ whose set $L^0$ of all regular vertices emitting only loop edges is not empty,
    then taking the maximal out-split of $E$ at $L^0$ coincides with the locally derived graph of $E$ obtained
    by considering every vertex in $L^0$ with all its loops as a subgraph with the labelling map as in the above example.
    The following graph illustrates the situation:
    \begin{equation*}
        E:=\begin{tikzcd}[ampersand replacement=\&]
	\bullet \& \hspace{-50pt} v_1 \hspace{50pt} v_2 \hspace{-50pt}\& \bullet \&  \&\& \bullet \& \hspace{-50pt} v_3 \\
	\\
	\& \bullet \&\& \dots \&\& \bullet
	\arrow[from=1-1, to=1-1, loop, in=105, out=165, distance=5mm]
	\arrow[from=1-1, to=1-1, loop, in=95, out=175, distance=15mm]
	\arrow[from=1-3, to=1-3, loop, in=15, out=75, distance=5mm]
	\arrow[from=1-6, to=1-6, loop, in=60, out=120, distance=5mm]
	\arrow[from=1-6, to=1-6, loop, in=55, out=125, distance=10mm]
	\arrow[from=1-6, to=1-6, loop, in=50, out=130, distance=15mm]
	\arrow[from=3-2, to=1-1]
	\arrow[from=3-2, to=1-3]
	\arrow[from=3-4, to=3-2]
	\arrow[from=3-6, to=1-6]
	\arrow[from=3-6, to=3-4]
\end{tikzcd}.
    \end{equation*}
    Here $L^0=\{v_1,v_2,v_3\}$ and we have three disjoint subgraphs: $E_i:=(\{v_i\},s^{-1}_E(v_i), s_i, t_i)$, $i\in\{1,2,3\}$.
    The labelling maps are bijective  maps to the groups $\mathbb{Z}/2\mathbb{Z}$, $\{e\}$, and  $\mathbb{Z}/3\mathbb{Z}$, respectively.
    The locally derived graph of $E$ with respect to this structure coincides with the maximal out-split at $L^0$:
    \begin{equation*}
        \begin{tikzcd}[ampersand replacement=\&]
	\&\&\&\&\&  \& \bullet \, \,  v_3^2 \&   \\
	v_1^1 \, \bullet  \& \bullet \,  v_1^2 \&\& v_2^1 \, \bullet \&\& v_3^1 \, \bullet \&\& \bullet \, v_3^3 \\
	\\
	\&\&  \bullet \&\& \dots \&\& \bullet 
	\arrow[from=1-7, to=1-7, loop, in=120, out=60, distance=5mm]
	\arrow[shift right, curve={height=6pt}, from=1-7, to=2-6]
	\arrow[curve={height=-10pt}, from=1-7, to=2-8]
	\arrow[from=2-1, to=2-1, loop, in=105, out=165, distance=5mm]
	\arrow[curve={height=-6pt}, from=2-1, to=2-2]
	\arrow[curve={height=-6pt}, from=2-2, to=2-1]
	\arrow[from=2-2, to=2-2, loop, in=15, out=75, distance=5mm]
	\arrow[from=2-4, to=2-4, loop, in=15, out=75, distance=5mm]
	\arrow[shift right=2, curve={height=-6pt}, from=2-6, to=1-7]
	\arrow[from=2-6, to=2-6, loop, in=150, out=210, distance=5mm]
	\arrow[curve={height=-6pt}, from=2-6, to=2-8]
	\arrow[shift left=2, curve={height=6pt}, from=2-8, to=1-7]
	\arrow[curve={height=-6pt}, from=2-8, to=2-6]
	\arrow[from=2-8, to=2-8, loop, in=330, out=30, distance=5mm]
	\arrow[from=4-3, to=2-1]
	\arrow[from=4-3, to=2-2]
	\arrow[from=4-3, to=2-4]
	\arrow[from=4-5, to=4-3]
	\arrow[from=4-7, to=1-7]
	\arrow[curve={height=-12pt}, from=4-7, to=2-6]
	\arrow[curve={height=12pt}, from=4-7, to=2-8]
	\arrow[from=4-7, to=4-5]
\end{tikzcd}.
    \end{equation*}
    \end{example}

        Although  the locally-derived-graph construction and out-split can coincide for a large family of graphs described above, 
        these two constructions are vastly different in general. 
        For completeness,  let us also provide examples that demonstrate this.
\begin{example}
        Consider the graph $E$ with two vertices $E^0=\{v_1,v_2\}$ and two edges $E^1=\{e,f\}$ that both start at $v_1$ and end at $v_2$. The only non-trivial out-split of $E$ yields:
        \begin{equation*}
            \begin{tikzcd}[ampersand replacement=\&]
	{v_1}\,  \bullet \& \bullet \, {v_2}
	\arrow["e", curve={height=-6pt}, from=1-1, to=1-2]
	\arrow["f"', curve={height=6pt}, from=1-1, to=1-2]
\end{tikzcd}\longmapsto\begin{tikzcd}[ampersand replacement=\&,row sep=tiny]
	{v_1^1} \, \bullet \& \\
	\& {v_2^1}  \, \bullet \\
	    {v_1^2} \, \bullet
	\arrow["e^1", from=1-1, to=2-2]
	\arrow["f^1"', from=3-1, to=2-2]
\end{tikzcd}.
        \end{equation*}
        On the other hand, the out-split graph cannot be obtained as a locally derived graph of~$E$.
       Indeed, the only locally-derived-graph construction producing the correct amount of edges and vertices comes from
       considering $\{v_1\}$ as the unique subgraph with the labelling group~$\mathbb{Z}/2\mathbb{Z}$.
       However, then our locally-derived-graph construction leads to:
        \begin{equation*}
            \begin{tikzcd}[ampersand replacement=\&]
	{v_1} \, \bullet \&   \bullet \, {v_2}
	\arrow["e", curve={height=-6pt}, from=1-1, to=1-2]
	\arrow["f"', curve={height=6pt}, from=1-1, to=1-2]
\end{tikzcd}\longmapsto \begin{tikzcd}[ampersand replacement=\&,column sep=large,row sep=tiny]
	{(v_1,2)\, \bullet} \& \\
	{(v_1,1) \, \bullet} \& {\bullet  \, {v_2^1}}
	\arrow["{(f,1)}"', curve={height=6pt}, from=2-1, to=2-2]
	\arrow["{(e,1)}", curve={height=-6pt}, from=2-1, to=2-2]
\end{tikzcd}.
        \end{equation*}
        We thus obtained a locally derived graph not only different from the out-split graph, but also with
        a different (non-isomorphic) C*-algebra.
\end{example}
    
\begin{example}
        Similarly, \cite[Example~{5.10}]{ht-25} is an example of a locally derived graph that is not an out-split and whose graph 
        C*-algebra is not isomorphic with the graph C*-algebra of the initial graph. Indeed,
        for the graph with one vertex and two loops, we can take as the unique subgraph the vertex with just one loop~$f$ and the
        non-trivial labelling map to $\mathbb{Z}/2\mathbb{Z}$. Then the locally-derived-graph construction yields
        \begin{equation*}
            \begin{tikzcd}[ampersand replacement=\&,row sep=tiny]
	\bullet
	\arrow[from=1-1, to=1-1, loop, in=50, out=130, distance=15mm, "e"]
	\arrow[from=1-1, to=1-1, loop, in=60, out=120, distance=5mm, "f"]
\end{tikzcd}\longmapsto \begin{tikzcd}[ampersand replacement=\&,row sep=tiny]
	\bullet \&\& \bullet
	\arrow[from=1-1, to=1-1, loop, in=100, out=170, distance=10mm, "(e{,}1)"]
	\arrow[from=1-1, to=1-3, "(f{,}1)"]
	\arrow[curve={height=-18pt}, from=1-1, to=1-3, "(e{,}2)"]
	\arrow[curve={height=-12pt}, from=1-3, to=1-1, "(f{,}2)"]
\end{tikzcd}.
        \end{equation*}
        This clearly cannot be an out-split as the only non-trivial out-split that we can form is the following:
        \begin{equation*}
            \begin{tikzcd}[ampersand replacement=\&,row sep=tiny]
	\bullet
	\arrow[from=1-1, to=1-1, loop, in=50, out=130, distance=15mm, "e"]
	\arrow[from=1-1, to=1-1, loop, in=60, out=120, distance=5mm, "f"]
\end{tikzcd}\longmapsto\begin{tikzcd}[ampersand replacement=\&,row sep=tiny]
	\bullet \&\& \bullet
	\arrow[from=1-1, to=1-1, loop, in=100, out=170, distance=10mm, "e^1"]
	\arrow[curve={height=-12pt}, from=1-1, to=1-3, "e^2"]
	\arrow[curve={height=-12pt}, from=1-3, to=1-1, "f^2"]
	\arrow[from=1-3, to=1-3, loop, in=10, out=80, distance=10mm, "f^1"]
\end{tikzcd}.
        \end{equation*}
        As in the previous example, not only the right-hand-side graphs are different, but also they yield non-isomorphic
       graph C*-algebras, as is proven in \cite[Example~5.10]{ht-25}.
    \end{example}

\subsection{Amplified graphs}\footnote{We are very grateful to Francesco D'Andrea for amplifying  our attention towards these graphs.}
In this section, we apply our newly developed unital reduction to provide
a new perspective on some key results in~\cite{ers12}.
We begin by providing a considerably shorter proof of the following lemma.
\begin{lemma}[{\cite[Lemma 3.1]{ers12}}]\label{ESRLemma1}
    Let $G$ be a graph, $u\in G^0$ be an infinite emitter and $v\in G^0$ be a finite emitter that $u$ countably infinitely emits to. Let $E:=(G^0,E^1,s_E,t_E)$ be the graph with the same vertices as $G$ and with the following set of edges
    \begin{equation*}
        E^1 := G^1\sqcup\{f^n \: | \: n\in\mathbb{N},f\in s_G^{-1}(v)\}.
    \end{equation*}
    We extend $s_G$ and $t_G$ to $s_E$ and $t_E$, respectively, by setting $s_E(f^n):=u$ and ${t_E(f^n):=t_G(f)}$ for all $n\in\mathbb{N}$. Then $C^*(E)\cong C^*(G)$.
\end{lemma}
\begin{proof}
    Note that our assumptions force upon us the condition $u\neq v$. Consider a set $X\subseteq s_G^{-1}(u)\cap t_G^{-1}(v)$ such that both $X$ and $(s_G^{-1}(u)\cap t_G^{-1}(v))\setminus X$ are infinite. As $u\neq v$, the set $X$ does not contain any loops, so we can form the unital reduction $G_{R^+}$ of $G$ at $v$ with respect to $X$. Observe that $G_{R^+}$ is exactly the graph $E$. Indeed, since $(s_G^{-1}(u)\cap t_G^{-1}(v))\setminus X$ is infinite, then there are still infinitely many edges from $u$ to $v$, so $G^1\subseteq G_{R^+}^1$. If we enumerate $X=\{e_1,e_2,\dots\}$, then the unital reduction construction yields an edge $e_nf$ for each $n\in\mathbb{N}$ and $f\in s_G^{-1}(v)$. These edges correspond exactly to $f^n$ in $E^1$ because they also start at $u$ and end at $t_G(f)$. Hence, $E\cong G_{R^+}$, so $C^*(E)\cong C^*(G)$ by Proposition~\ref{shiftuniinv} and Lemma~\ref{shiftinv}.
\end{proof}

To prove the next lemma, it suffices to combine an out-split at an infinite emitter with Lemma~\ref{ESRLemma1}.
\begin{lemma}[{\cite[Lemma 3.6]{ers12}}]\label{ESRLemma2}
    Let $G$ be a graph, $u\in G^0$ be an infinite emitter and $v\in G^0$ be a non-sink vertex that $u$ countably infinitely emits to. 
    Fix an edge $f\in s_G^{-1}(v)$ and take $E:=(G^0,E^1,s_E,t_E)$ to be the graph with the same vertices as $G$ and with the following 
    set of edges
    \begin{equation*}
        E^1 := G^1\sqcup\{f^n \: | \: n\in\mathbb{N}\}.
    \end{equation*}
    We extend $s_G$ and $t_G$ to $s_E$ and $t_E$, respectively, by setting $s_E(f^n):=u$ and $t_E(f^n):=t_G(f)$ for all $n\in\mathbb{N}$. Then $C^*(E)\cong C^*(G)$.
\end{lemma}
%
\begin{example}
    Below are some examples of these moves at work. First, consider the following graph:
    \begin{equation*}
        \begin{tikzcd}[ampersand replacement=\&]
	\dots \&\& {v'} \\
	u \& v \\
	\dots \&\& \dots
	\arrow[from=1-1, to=2-1]
	\arrow[from=2-1, to=2-1, loop, in=150, out=210, distance=5mm]
	\arrow["\infty", from=2-1, to=2-2]
	\arrow[from=2-1, to=3-1]
	\arrow["f"', from=2-2, to=1-3]
	\arrow[shift left, curve={height=-6pt}, from=2-2, to=1-3]
	\arrow[from=2-2, to=2-2, loop, in=300, out=240, distance=5mm]
	\arrow["\infty", from=2-2, to=3-3]
\end{tikzcd}.
    \end{equation*}
    An empty arrow denotes either a finite or an infinite amount of edges. Out-splitting at $v$ with respect to the partition $\{f\}\sqcup s_E^{-1}(v)\setminus\{f\}$, we get
    \begin{equation*}
        \begin{tikzcd}[ampersand replacement=\&]
	\dots \& {v^0} \& {v'} \\
	u \\
	\dots \& {v^1} \& \dots
	\arrow[from=1-1, to=2-1]
	\arrow["f", from=1-2, to=1-3]
	\arrow["\infty", from=2-1, to=1-2]
	\arrow[from=2-1, to=2-1, loop, in=150, out=210, distance=5mm]
	\arrow[from=2-1, to=3-1]
	\arrow["\infty"', from=2-1, to=3-2]
	\arrow[from=3-2, to=1-2]
	\arrow[from=3-2, to=1-3]
	\arrow[from=3-2, to=3-2, loop, in=195, out=255, distance=5mm]
	\arrow["\infty", from=3-2, to=3-3]
\end{tikzcd}.
    \end{equation*}
    Then our unital reduction from Lemma~\ref{ESRLemma1} yields
    \begin{equation*}
        \begin{tikzcd}[ampersand replacement=\&]
	\dots \&\& {v^0} \& {v'} \\
	u \\
	\dots \&\& {v^1} \& \dots
	\arrow[from=1-1, to=2-1]
	\arrow["f"', from=1-3, to=1-4]
	\arrow["\infty"', from=2-1, to=1-3]
	\arrow["\infty", curve={height=-30pt}, from=2-1, to=1-4]
	\arrow[from=2-1, to=2-1, loop, in=150, out=210, distance=5mm]
	\arrow[from=2-1, to=3-1]
	\arrow["\infty"', from=2-1, to=3-3]
	\arrow[from=3-3, to=1-3]
	\arrow[from=3-3, to=1-4]
	\arrow[from=3-3, to=3-3, loop, in=195, out=255, distance=5mm]
	\arrow["\infty", from=3-3, to=3-4]
\end{tikzcd}.
    \end{equation*}
    Now, out-amalgamation at $v^0$ and $v^1$ gives us the desired result:
    \begin{equation*}
        \begin{tikzcd}[ampersand replacement=\&]
	\dots \&\&\& {v'} \\
	u \&\& {v} \\
	\dots \&\&\& \dots
	\arrow[from=1-1, to=2-1]
	\arrow["\infty", curve={height=-30pt}, from=2-1, to=1-4]
	\arrow[from=2-1, to=2-1, loop, in=150, out=210, distance=5mm]
	\arrow["\infty"', from=2-1, to=2-3]
	\arrow[from=2-1, to=3-1]
	\arrow[from=2-3, to=1-4]
	\arrow[from=2-3, to=2-3, loop, in=240, out=300, distance=5mm]
	\arrow["\infty", from=2-3, to=3-4]
\end{tikzcd}.
    \end{equation*}
    Consider also another case, where $f$ is an edge from $v$ to $u$:
    \begin{equation*}
        \begin{tikzcd}[ampersand replacement=\&]
	\dots \&\& \\
	u \& v \& \dots \\
	\dots
	\arrow[from=1-1, to=2-1]
	\arrow[from=2-1, to=2-1, loop, in=150, out=210, distance=5mm]
	\arrow["\infty"', from=2-1, to=2-2]
	\arrow[from=2-1, to=3-1]
	\arrow["f"', curve={height=12pt}, from=2-2, to=2-1]
	\arrow[from=2-2, to=2-2, loop, in=300, out=240, distance=5mm]
	\arrow["\infty", from=2-2, to=2-3]
\end{tikzcd}.
    \end{equation*}
    Then out-split and Lemma~\ref{ESRLemma1} yield the following graph:
    \begin{equation*}
        \begin{tikzcd}[ampersand replacement=\&]
	\dots \&\& {v^0} \& \\
	u \\
	\dots \&\& {v^1} \& \dots
	\arrow[from=1-1, to=2-1]
	\arrow["f"', curve={height=12pt}, from=1-3, to=2-1]
	\arrow["\infty"', from=2-1, to=1-3]
	\arrow["\infty", from=2-1, to=2-1, loop, in=150, out=210, distance=5mm]
	\arrow[from=2-1, to=3-1]
	\arrow["\infty"', from=2-1, to=3-3]
	\arrow[from=3-3, to=1-3]
	\arrow[from=3-3, to=3-3, loop, in=300, out=240, distance=5mm]
	\arrow["\infty", from=3-3, to=3-4]
\end{tikzcd}.
    \end{equation*}
    Finally, out-amalgamation at $v^0$ and $v^1$ gives us the desired result:
    \begin{equation*}
        \begin{tikzcd}[ampersand replacement=\&]
	\dots \&\& \\
	u \& v \& \dots \\
	\dots
	\arrow[from=1-1, to=2-1]
	\arrow["\infty", from=2-1, to=2-1, loop, in=150, out=210, distance=5mm]
	\arrow["\infty"', from=2-1, to=2-2]
	\arrow[from=2-1, to=3-1]
	\arrow["f"', curve={height=12pt}, from=2-2, to=2-1]
	\arrow[from=2-2, to=2-2, loop, in=300, out=240, distance=5mm]
	\arrow["\infty", from=2-2, to=2-3]
\end{tikzcd}.
    \end{equation*}
\end{example}

To apply Lemma~\ref{ESRLemma2} to amplified graphs, first recall the definitions:
\begin{definition}[e.g., see Definition~2.4 in \cite{ers12}]
    Let $G$ be a graph. An \emph{amplification} of $G$ is the graph $\overline{G}=(G^0,\overline{G}^1,s_{\overline{G}},t_{\overline{G}})$ where the set of edges is given by
    \begin{equation*}
        \overline{G}^1:=\{ e(v,w)^n \: | \: n\in\mathbb{N} \text{ and there exists an edge from } v\text{ to }w \},
    \end{equation*}
    and the source and target maps are defined by $s_{\overline{G}}(e(v,w)^n):=v$ and $t_{\overline{G}}(e(v,w)^n):=w$. We say that a graph $G$ is \emph{amplified} if it is equal to its amplification.
\end{definition}
\begin{definition}[e.g., see Definition~2.5 in \cite{ers12}]
    The \emph{transitive closure} of a graph $G$ is the graph $tG:=(G^0,tG^1,s_{tG},t_{tG})$ where the edges are given by
    \begin{equation*}
        tG^1:=\{ e(v,w) \: | \: \text{there exists a path but no edge from } v\text{ to }w \},
    \end{equation*}
    and the source and target maps are defined by $s_{tG}(e(v,w)):=v$ and $t_{tG}(e(v,w)):=w$.
\end{definition}
Now we are ready to claim:
\begin{corollary}[{\cite[Corollary~3.9]{ers12}}]
    Let $G$ be a graph such that $|G^0|<\infty$. Then $C^*(\overline{G})\cong C^*(\overline{tG})$, where $\overline{G}$ denotes the amplification of $G$ and $tG$ is the transitive closure of $G$.
\end{corollary}

\begin{example}[Vaksman--Soibelman quantum complex projective spaces $\mathbb{C}P_q^{n}$]
    For any $n\ge 1$, we have the C*-algebra $C(S^{2n+1}_q)$ of the 
Vaksman--Soibelman \cite{vs91} odd quantum sphere~$S^{2n+1}_q$, and we define the C*-algebra  $C(\mathbb{C}P_q^{n})$ 
of the Vaksman--Soibelman quantum complex projective space $\mathbb{C}P_q^{n}$ as its $U(1)$-fixed-point subalgebra for
the diagonal action on~$C(S^{2n+1}_q)$. Hong and Szyma\'nski \cite[Section~4.3]{hs02} showed that $C(\mathbb{C}P_q^{n})$
is the graph C*-algebra of the following graph:
    \begin{equation*}
       F_n:= \begin{tikzcd}[ampersand replacement=\&]
	{v_1} \&\& {v_2} \&\& {v_3} \& \dots \& {v_{n+1}}
	\arrow["\infty", from=1-1, to=1-3]
	\arrow["\infty", curve={height=-18pt}, from=1-1, to=1-5]
	\arrow["\infty", curve={height=-30pt}, from=1-1, to=1-6]
	\arrow["\infty", shift left=5, curve={height=-30pt}, from=1-1, to=1-7]
	\arrow["\infty", from=1-3, to=1-5]
	\arrow["\infty", curve={height=18pt}, from=1-3, to=1-6]
	\arrow["\infty", curve={height=30pt}, from=1-3, to=1-7]
	\arrow["\infty", from=1-5, to=1-6]
	\arrow["\infty", from=1-6, to=1-7]
\end{tikzcd}.
    \end{equation*}
    The graph $F_n$ is exactly the amplification of the transitive closure of the following amplified graph:
    \begin{equation*}
        \tilde{F_n}:=\begin{tikzcd}[ampersand replacement=\&]
	{v_1} \& {v_2} \& {v_3} \& \dots \& {v_{n+1}}
	\arrow["\infty", from=1-1, to=1-2]
	\arrow["\infty", from=1-2, to=1-3]
	\arrow["\infty", from=1-3, to=1-4]
	\arrow["\infty", from=1-4, to=1-5]
\end{tikzcd}.
    \end{equation*}
    Consequently, we can use the above simpler graph $\tilde{F_n}$ to present $C(\mathbb{C}P_q^{n})$ as a graph C*-algebra.
\end{example}

To end this section, let us relate  the above example to yet another well known application of the shift move. First, 
recall that Hong and Szyma\'nski \cite[Section~4.1]{hs02} also proved that the graph C*-algebra of the  graph 
    \begin{equation*}
	\begin{tikzpicture}
		\node[main node] (1) {};
		\node (2) [right of=1] {};
		\node (3) [right of=2] {};
		\node (4) [right of=3] {$\dots$};
		\node (5) [right of=4] {};
        \node (6) [left of=1] {$L_{2n+1}:=$};
		\filldraw (1) circle (0.05) node[below=2pt] {$v_1$};
		\filldraw (2) circle (0.05) node[below=2pt] {$v_2$};
		\filldraw (3) circle (0.05) node[below=2pt] {$v_3$};
		\filldraw (5) circle (0.05) node[below right=2pt] {$v_{n+1}$};
		\path[freccia] (1) edge[ciclo] (1);
		\path[freccia] (2) edge[ciclo] (2);
		\path[freccia] (3) edge[ciclo] (3);
		\path[freccia] (5) edge[ciclo] (5);
		\path[freccia] (1) edge (2)
		(2) edge (3)
		(3) edge (4)
		(4) edge (5)
		(1) edge[bend right=30] (3)
		(1) edge[bend right=60] (5)
		(2) edge[bend right =55] (5)
		(3) edge[bend right=50] (5);
	\end{tikzpicture}
    \end{equation*}
is isomorphic with the C*-algebra~$C(S^{2n+1}_q)$. Now, we can choose $v:=v_n$ and $W:=\{v_1,\ldots,v_{n-1}\}$ and
observe that the shifted graph  $(L_{2n+1})_{(v,W)}$ is

\begin{equation*}
	\begin{tikzpicture}
		\node[main node] (1) {};
		\node (2) [right of=1] {};
		\node (3) [right of=2] {};
		\node (4) [right of=3] {$\dots$};
		\node (5) [right of=4] {};
        \node (6) [right of=5] {};
        \node (7) [right of=6] {.};
		\filldraw (1) circle (0.05) node[below=2pt] {$v_1$};
		\filldraw (2) circle (0.05) node[below=2pt] {$v_2$};
		\filldraw (3) circle (0.05) node[below=2pt] {$v_3$};
		\filldraw (5) circle (0.05) node[below right=2pt] {$v_{n}$};
        \filldraw (6) circle (0.05) node[below right=2pt] {$v_{n+1}$};
		\path[freccia] (1) edge[ciclo] (1);
		\path[freccia] (2) edge[ciclo] (2);
		\path[freccia] (3) edge[ciclo] (3);
		\path[freccia] (5) edge[ciclo] (5);
        \path[freccia] (6) edge[ciclo] (6);
		\path[freccia] (1) edge (2)
		(2) edge (3)
		(3) edge (4)
		(4) edge (5)
        (5) edge (6)
		(1) edge[bend right=30] (3)
		(1) edge[bend right=60] (5)
		(2) edge[bend right =55] (5)
		(3) edge[bend right=50] (5);
	\end{tikzpicture}
    \end{equation*}
Iterating the shift, we arrive at the  much simpler graph
  \begin{equation*}
     \begin{tikzpicture}
		\node[main node] (1) {};
		\node (2) [right of=1] {};
		\node (3) [right of=2] {};
		\node (4) [right of=3] {$\dots$};
		\node (5) [right of=4] {};
        \node (6) [left of=1] {$\tilde{L}_{2n+1}:=$};
        \node (7) [right of=5] {.};
		\filldraw (1) circle (0.05) node[below=2pt] {$v_1$};
		\filldraw (2) circle (0.05) node[below=2pt] {$v_2$};
		\filldraw (3) circle (0.05) node[below=2pt] {$v_3$};
		\filldraw (5) circle (0.05) node[below=2pt] {$v_{n+1}$};
		\path[freccia] (1) edge[ciclo] (1);
		\path[freccia] (2) edge[ciclo] (2);
		\path[freccia] (3) edge[ciclo] (3);
		\path[freccia] (5) edge[ciclo] (5);
		\path[freccia] (1) edge (2)
		(2) edge (3)
		(3) edge (4)
		(4) edge (5);
	\end{tikzpicture}
 \end{equation*}
Finally, as is explained in detail in~\cite{dz26}, the graph C*-algebra $C^*(F_n)$ is the core of $C^*(L_{2n+1})$
and the graph C*-algebra $C^*(\tilde{F_n})$ is the core of $C^*(\tilde{L}_{2n+1})$.

\subsection{Minimal unitization of the C*-algebra of compact operators}
For starters, let us consider  the following finite graphs and morphisms between them:
\begin{center}
\begin{tikzpicture}
 \draw (-5,0) node(Football){\usebox{\Football}};
 \draw (-6,1.3) node{$A_n$};
 \draw (0,3.5) node(Star){\usebox{\Star}};
 \draw (0,1.3) node{$B_n$};
 \draw (5,0) node(Line){\usebox{\Line}};
  \draw (6,1.3) node{$C_n$};
 \draw [->, dashed, shorten >=-15pt, shorten <=-7pt ] (Star)-- node[above left] {$f_n$} (Football);
 \draw [->, dashed, shorten >=7pt ] (Star)-- node[above right] {$g^n$} (Line);

\end{tikzpicture}.
\end{center}
\noindent
Here $f_n$ is the  out-split graph homomorphism corresponding to the maximal out-split of $A_n$ at the vertex~$2$, 
and $g_n$ is the composition of path homomorphisms corresponding to the following shifts:
we shift the edge $y_n$ so that it 
ends at the vertex $n-1$, then the edge $y_{n-1}$ so that it ends at $n-2$, and so on until we shift the edge
$y_3$ to end at~$2$. Now,
let us write out explicitly the $*$-homomorphisms $f_n^*\colon C^*(A_n)\to C^*(B_n)$ 
and $g_*^n\colon C^*(B_n)\to C^*(C_n)$:
\begin{gather}
	f_n^*(S_1) = S_1,\quad f_n^*(S_2) = \sum_{k=2}^nS_k = 1 - S_1,\quad f_n^*(S_{x_i}) = S_{y_i},\label{flabel}\\
	g_*^n(S_i) = S_i,\quad g_*^n (S_{y_i}) = S_{z_i}\cdots S_{z_2}.\label{glabel}
\end{gather}

To invert these $*$-homomorphisms, we can use
Theorem~\ref{thm:maxoutsplit} and Theorem~\ref{thm:shiftinv}, respectively.
The inverses will be covariantly induced from extended path homomorphisms composing to an extended
path homomorphism $\varphi^n\colon S(C_n)\to S(A_n)$.
On the level of graph $C^*$-algebras, we have
\begin{equation}
	\varphi^n_* \colon C^*(C_n) \longrightarrow C^*(A_n), \quad \varphi^n_*:=(f_n^*)^{-1}\circ (g_*^n)^{-1}=( g^n_* \circ   f_n^* )^{-1}.
\end{equation}
Combining explicit formulas in Theorem~\ref{thm:maxoutsplit} and Theorem~\ref{thm:shiftinv}, we infer
 that $\varphi^n$ is given as follows:
\begin{gather}
	\varphi^n(i,i) = \begin{cases}
		(1,1) & i=1,\\
		(x_i,x_i) & i>1,\\
	\end{cases}\qquad
	\varphi^n(z_i,1) = \begin{cases}
		(x_2,1) & i=2,\\
		(x_i,x_{i-1}) & i>2.
	\end{cases}
\end{gather}
Hence, we obtain
\begin{equation}
	\varphi_*^n(S_i) = \begin{cases} S_1 & i=1,\\
 S_{x_i}S_{x_i}^*  & i>1,
 \end{cases}\qquad
   \varphi_*^n(S_{z_i}) = \begin{cases} S_{x_2} & i=2,\\
 S_{x_i}S_{x_{i-1}}^*  & i>2.
 \end{cases}
\end{equation}

As the above formulas do not depend on~$n$, we can extend them to the $n=\infty$ case. Indeed,
as $C_n$ is a subgraph of $C_\infty$ and $A_n$ is a subgraph of $A_\infty$, we infer that
$S(C_n)$ is a subsemigroup of $S(C_\infty)$ and $S(A_n)$ is a subsemigroup of $S(A_\infty)$.
Furthermore, since any element of $S(C_\infty)$ is an element of $S(C_n)$ for some $n$, and our formulas
do not depend on~$n$, we easily conclude that $\varphi^\infty$ defines a homomorphism of semigroups.
Finally, it is straightforward to directly verify its regularity. Hence, by Corollary~\ref{C*induct}, we obtain  
a $*$-homomorphism
$\varphi_*^\infty\colon C^*(C_\infty)\to C^*(A_\infty)$, which is given by the same formula as~$\varphi_*^n$. 

Passing to the minimal unitization of $C^*(C_\infty)$, we get
the induced $*$-homomorphism $(\varphi_*^\infty)^+\colon C^*(C_\infty)^+\to C^*(A_\infty)$
rendering the diagram
\begin{equation}
\begin{tikzcd}[ampersand replacement=\&]
	{C^*(C_\infty)} \&\& {C^*(A_\infty)} \\
	{C^*(C_\infty)^+}
	\arrow["{\varphi_*^\infty}", from=1-1, to=1-3]
	\arrow[hook, from=1-1, to=2-1]
	\arrow["{(\varphi_*^\infty)^+}"', from=2-1, to=1-3]
\end{tikzcd}
\end{equation}
commutative. The $*$-homomorphism $(\varphi_*^\infty)^+$ is surjective because all generators of 
$C^*(A_\infty)$ are in the image of $\varphi_*^\infty$  except for $S_2$, and $S_1+S_2=1$.
It is also injective because otherwise $1$ would be in the image of $\varphi_*^\infty$, which is
impossible because $C^*(C_\infty)$ is simple and non-unital. We have thus shown that 
the graph C*-algebra $C^*(A_\infty)$ is the minimal unitization of the graph C*-algebra $C^*(C_\infty)$
of compact operators~\cite{kpr98}.

On the other hand, combining the formulas \eqref{flabel} and \eqref{glabel}, we have:
\begin{equation}\label{middle}
	(\varphi_*^n)^{-1}(S_1) =  S_1,\quad (\varphi_*^n)^{-1}(S_2) = \sum_{k=2}^nS_k = 1 - S_1, \quad 
    (\varphi_*^n)^{-1}
    (S_{x_i}) = S_{z_i}\cdots S_{z_2}.
\end{equation}
Now one can see that
the inverse $*$-homomorphism $(\varphi_*^n)^{-1}$ cannot be extended to $C^*(A_\infty)$ 
because of the middle formula in~\eqref{middle} and the non-unitality of~$C^*(C_\infty)$. 
However, it is straightforward to verify that the inverse $\psi$ of $(\varphi_*^\infty)^+$
is given by:
\begin{equation}
	\psi(S_1) = (0,S_1),\quad \psi(S_2) = (1,-S_1),\quad \psi(S_{x_i}) = (0,S_{z_i}\cdots S_{z_2}).
\end{equation}
Comparing the above two displayed formulas one immediately sees that $\psi$ can be viewed as
the $n=\infty$ version of $(\varphi_*^n)^{-1}$.

The above result concerning the minimal unitization of a non-unital graph C*-algebra can be viewed 
as obtained via a ``minimal unitization" of a graph with infinitely many vertices. Indeed, defining
$C_\infty^+:=A_\infty$, we can claim that $C^*(C_\infty)^+\cong C^*(C_\infty^+)$. We thus arrive
at the following:
\begin{conjecture}
Let $E$ be a graph with infinitely many vertices. Also, let 
 all but finitely many vertices of $E$ belong to finitely many infinite paths without repeated
 vertices. Then respective
folding of these paths to countably-infinitely
many edges between two different vertices yields a graph $E^+$ satisfying $C^*(E)^+\cong C^*(E^+)$.
\end{conjecture}
\begin{remark}
It is very tempting to think of $A_n$ as converging to $A_\infty$ in some sense.
Such a convergence would yield a convergence $M_n(\mathbb{C})\to C(S^2_{q})$,
which was already established  in \cite{akk22}
 in the setting of noncommutative metric geometry~\cite{r00}. Indeed,
the graph C*-algebra $C^*(A_\infty)$ is isomorphic (see \cite[p.\ 109]{vs91} and \cite[Section 2.3]{hs02}) 
to the C*-algebra  $C(S^2_{q})$ 
of the standard Podle\'s quantum sphere~\cite{p-p87}. 
Therefore,
this example could be a starting point of investigations of convergence of graphs in the sense of 
noncommutative metric geometry.
\end{remark}

\section*{Acknowledgements}
\noindent 
This research is part of the EU Staff Exchange project 101086394 \emph{Operator Algebras That One Can See}.
The project is co-financed by the Polish Ministry of Education and Science under the program PMW 
(grant agreements 5305/HE/2023/2 and 5448/HE/2023/2).
 G.\ ~G.~de Castro was  
partially supported by Fundacao de Amparo a Pesquisa e Inovacao do Estado de Santa Catarina (FAPESC), 
Edital 21/2024 and Conselho Nacional de 
Desenvolvimento Científico e Tecnológico (CNPq) - Brazil. We are happy to thank
Roberto Conti, Joachim Cuntz, Francesco D'Andrea,
Adam Dor-On, Dimitris M.\ Gerontogiannis, Roozbeh Hazrat,  Tomasz Maszczyk and Murad {\"O}zaydın for 
fruitful discussions.
\bibliographystyle{abbrv}
\bibliography{ref}

\end{document}